\documentclass[11pt]{amsart}
\usepackage[T1]{fontenc}
\usepackage[utf8]{inputenc} % ok for pdflatex; remove if compiling with lualatex/xelatex
\usepackage[
  backend=biber,
  style=numeric,
  sorting=none,
  giveninits=false,
  doi=true,
  isbn=true,
  url=true,
  eprint=true
]{biblatex}

\usepackage{lmodern}
\usepackage{microtype}

\usepackage{amsmath, amssymb, amsfonts, amsthm}
\usepackage{mathtools}
\usepackage{enumitem}
\usepackage{hyperref}
\usepackage{cleveref}
\usepackage{tikz-cd}
\usepackage{geometry}
\usepackage{thmtools}
\usepackage{float}
\hypersetup{
  colorlinks=true,
  linkcolor=blue,
  citecolor=blue,
  urlcolor=blue,
  pdftitle={Mixed Hodge Modules and Canonical Perverse Extensions},
  pdfauthor={Abdul Rahman}
}
\numberwithin{equation}{section}
\theoremstyle{plain}
\newtheorem{theorem}{Theorem}[section]
\newtheorem{proposition}[theorem]{Proposition}
\newtheorem{lemma}[theorem]{Lemma}
\newtheorem{corollary}[theorem]{Corollary}

\theoremstyle{definition}
\newtheorem{definition}[theorem]{Definition}

\theoremstyle{remark}
\newtheorem{remark}{Remark}[section]

\newcommand{\Q}{\mathbb{Q}}
\newcommand{\Z}{\mathbb{Z}}
\newcommand{\C}{\mathbb{C}}

\newcommand{\Perv}{\mathrm{Perv}}
\newcommand{\Hom}{\mathrm{Hom}}
\newcommand{\Ext}{\mathrm{Ext}}

\newcommand{\can}{\mathrm{can}}
\newcommand{\var}{\mathrm{var}}
\newcommand{\Var}{\mathrm{Var}}
\newcommand{\rat}{\mathrm{rat}}

\newcommand{\End}{\mathrm{End}}
\DeclareMathOperator{\Cone}{Cone}

\DeclareMathOperator{\supp}{supp}
\DeclareMathOperator{\id}{id}

\title[Mixed Hodge Modules and Canonical Perverse Extensions]{Mixed Hodge Modules and Canonical Perverse Extensions for Multi-Node Conifold Degenerations}

\author{Abdul Rahman}
\thanks{Email: arahman@alum.howard.edu}
\subjclass[2020]{14D06, 32S30, 18G80} % edit
\keywords{conifold degeneration, perverse sheaves, Picard--Lefschetz theory, spherical twists} % edit

\begin{document}

\begin{abstract}
We study one-parameter conifold degenerations whose central fiber \(X_0\) has finitely many ordinary double points and construct a mixed-Hodge-module refinement of the corrected perverse variation-cone object associated with the degeneration. At each node \(p_k\), the local \(\can^H/\Var^H/N\) geometry determines a rank-one point-supported interface block \(W_k^H\simeq i_{k*}\Q^H_{\{p_k\}}(-1)\). These blocks assemble canonically into \(Q_\Sigma^H=\bigoplus_{k=1}^r W_k^H\), through which the node-supported monodromy factors. We prove that the perverse-normalized cone \(\Cone(\Var^H)\) is concentrated in the mixed-Hodge-module heart, thereby defining \(\mathcal P^H\in MHM(X_0)\), and that \(\rat(\mathcal P^H)\simeq\mathcal P\). In the forward Yoneda orientation, the point-supported quotient is instead \(R_\Sigma^H\simeq\bigoplus_{k=1}^r i_{k*}\Q^H_{\{p_k\}}(-2)\), exhibiting a Hodge-theoretic distinction invisible after rational realization. We further prove a canonical nodewise decomposition of the rational perverse extension space and construct a morphism from the formal local Hodge sector to the independently defined limiting mixed Hodge structure. Its image records the node contributions that survive in global limiting cohomology, allowing nontrivial global relations among vanishing cycles.
\end{abstract}
\maketitle

\tableofcontents
%===================================================
\section{Introduction}
\label{sec:intro}
Let $\pi:\mathcal X\to\Delta$ be a one-parameter degeneration of complex
threefolds whose central fiber $X_0$ has finitely many ordinary double
points $\Sigma=\{p_1,\dots,p_r\}\subset X_0$.  The local topology of an
ordinary double point, its Milnor fibration, and the associated
Picard--Lefschetz transformation are classical
\cite{MilnorSingularPoints,DimcaSheaves}; the geometry of nodal
threefold degenerations and their smoothings has likewise been studied
extensively in the conifold setting
\cite{Clemens1983DoubleSolids,Friedman1986ThreefoldDoublePoints,
NamikawaSteenbrink1995}.  From the sheaf-theoretic point of view, such a
degeneration carries nearby- and vanishing-cycle complexes together with
canonical and variation morphisms, providing the natural formalism for
isolating the contribution created at the singular fiber
\cite{BBD,DimcaSheaves}.

There is a normalization issue that will be important throughout the
paper.  In the classical derived convention used in the earlier
single-node construction \cite{RahmanSchoberPaper}, the canonical and
variation morphisms are written
\[
\operatorname{can}^{\mathrm{cl}}:
\psi_{\pi,1}^{\mathrm{cl}}
\longrightarrow
\phi_{\pi,1}^{\mathrm{cl}},
\qquad
\operatorname{var}^{\mathrm{cl}}:
\phi_{\pi,1}^{\mathrm{cl}}
\longrightarrow
\psi_{\pi,1}^{\mathrm{cl}},
\]
and the corrected perverse object introduced there is represented by
$\mathcal P\simeq
\operatorname{Cone}(\operatorname{var}^{\mathrm{cl}})[-1]$
\cite{RahmanSchoberPaper}.  In the present paper we work instead with
the perverse-normalized functors
${}^p\psi_{\pi,1}=\psi_{\pi,1}^{\mathrm{cl}}[-1]$ and
${}^p\phi_{\pi,1}=\phi_{\pi,1}^{\mathrm{cl}}[-1]$, with
${}^p\!\operatorname{var}
=\operatorname{var}^{\mathrm{cl}}[-1]$
\cite{BBD,DimcaSheaves}.  Consequently,
$\operatorname{Cone}(\operatorname{var}^{\mathrm{cl}})[-1]
\simeq
\operatorname{Cone}({}^p\!\operatorname{var})$.
After suppressing the superscript ${}^p$, as we do from Section~2
onward, the same corrected object is therefore written simply as
$\mathcal P\simeq\operatorname{Cone}(\operatorname{var})$.  Thus the two
cone formulas do not define competing objects; they are two expressions
for the same construction in different normalization conventions.  The
single-node perverse construction was established in
\cite{RahmanSchoberPaper}, while its subsequent formulation in terms of
nearby and vanishing cycles, together with its relation to limiting
Hodge theory, was developed in \cite{RahmanPerverseNearbyCycles}.

The aim of the present paper is to carry this construction one categorical
level higher.  Rather than beginning with $\mathcal P$ and asking only
for mixed Hodge structures on its hypercohomology, we ask whether the
corrected perverse object itself admits a geometrically determined
refinement in Saito's category $MHM(X_0)$ of mixed Hodge modules.
Saito's theory supplies an exact faithful realization functor
$\operatorname{rat}:MHM(X_0)\to\Perv(X_0;\Q)$, together with
Hodge-theoretic nearby- and vanishing-cycle functors and their canonical,
variation, and monodromy morphisms
\cite{SaitoMHM,SaitoDuality,SaitoYoungGuide}.  Our problem is therefore
to construct an object $\mathcal P^H\in MHM(X_0)$ for which
$\operatorname{rat}(\mathcal P^H)\simeq\mathcal P$, and to determine
the local Hodge-theoretic data that make such a refinement possible.

Already at a single ordinary double point, this problem contains more
information than the rational perverse construction can see.  The Milnor
fiber of a threefold ordinary double point has one-dimensional reduced
middle cohomology, so the local rational vanishing sector has rank one
\cite{MilnorSingularPoints}.  Rank, however, does not determine Hodge
type.  The Hodge filtration and Tate normalization of the vanishing line
are governed by the mixed Hodge theory and spectrum of the singularity
\cite{Steenbrink1976Limits,Steenbrink1989Spectrum,
PetersSteenbrink2008}.  Nor does the underlying rational skyscraper
determine where the corresponding Hodge object occurs in an extension:
different Tate structures may have the same rational realization, and
the two Yoneda orientations of a point-supported extension are governed
by different extension groups
\cite{BBD,MacPhersonVilonen1986,GMV1996,SaitoMHM}.  These distinctions
are invisible after forgetting the Hodge structure but are decisive
inside $MHM(X_0)$.

For this reason, we keep several pieces of the local geometry separate
throughout the argument.  The raw vanishing Hodge structure is not
identified a priori with the point-supported Hodge block entering the
$\operatorname{can}/\operatorname{Var}/N$ package; the gluing block is
in turn distinguished from the point-supported term appearing after a
Yoneda orientation has been chosen; and neither object is identified
with its image in global hypercohomology.  This separation is not
terminological bookkeeping.  It is precisely what allows the Tate
normalization, extension orientation, and realization problem to be
resolved without importing information that is visible only after
passing to the rational perverse category.

The principal result is the construction of a mixed Hodge module
$\mathcal P^H\in MHM(X_0)$ refining the corrected perverse object.
Locally, the construction is extracted from Saito's
$\operatorname{can}^H/\operatorname{Var}^H/N$ package and from the
Hodge theory of the ordinary double point
\cite{SaitoMHM,SaitoDuality,Steenbrink1989Spectrum}.  For a finite set
of nodes, the resulting point-supported Hodge blocks assemble
canonically by finite direct sum, while the corresponding vanishing
cycles in the cohomology of a nearby smooth fiber may still satisfy
global relations
\cite{Clemens1983DoubleSolids,Friedman1986ThreefoldDoublePoints,
NamikawaSteenbrink1995}.  Rational realization then recovers the
corrected perverse variation-cone object, whereas hypercohomology
produces the associated mixed Hodge structures and their connecting
morphisms.  Independently, the nearby-cycle mixed Hodge module carries
the limiting mixed Hodge structure of the degeneration in the
Schmid--Steenbrink sense
\cite{Schmid1973VHS,Steenbrink1976Limits}.  The resulting comparison
therefore distinguishes, from the outset, the formal direct sum of local
node sectors from its image in global limiting cohomology; no
independence of the global vanishing cycles is assumed.

\subsection{Relation to earlier work}

The present paper begins from the corrected perverse object constructed in
the single-node setting in \cite{RahmanSchoberPaper}.  For a one-parameter
degeneration with an ordinary double point, that work isolates a canonical
perverse extension on the singular fiber, proves its perversity, identifies
its restriction to the smooth locus, and exhibits the rank-one
point-supported contribution carried by the node.  The subsequent analysis
of \cite{RahmanPerverseNearbyCycles} places this construction more explicitly
inside the nearby-/vanishing-cycle formalism and separates the resulting
perverse extension problem from the limiting mixed Hodge structure carried
by nearby cycles.  The present paper takes up the Hodge-theoretic problem
left open by that development.

A minor but important normalization issue separates the notation of the
earlier papers from the convention used here.  If
$\var^{\mathrm{cl}}$ denotes the variation morphism for the classical
derived vanishing- and nearby-cycle functors, then the earlier construction
is written
\[
\mathcal P
\simeq
\Cone\!\left(\var^{\mathrm{cl}}\right)[-1].
\]
For the perverse-normalized functors one has
${}^p\phi=\phi^{\mathrm{cl}}[-1]$ and
${}^p\psi=\psi^{\mathrm{cl}}[-1]$, and hence
${}^p\!\var=\var^{\mathrm{cl}}[-1]$; consequently,
\[
\Cone\!\left(\var^{\mathrm{cl}}\right)[-1]
\simeq
\Cone({}^p\!\var)
\]
in the constructible derived category
\cite{BBD,DimcaSheaves}.  After suppressing the superscript ${}^p$, as we do
throughout this paper, the corrected object is therefore written simply as
$\mathcal P\simeq\Cone(\var)$.  Thus the two formulas describe the same
perverse object in different normalization conventions rather than two
different constructions.

The question addressed here is correspondingly sharper than the existence
of the rational perverse sheaf.  We ask whether the corrected object admits
a geometrically determined refinement $\mathcal P^H\in MHM(X_0)$ whose
rational realization is $\mathcal P$, and, if so, which Hodge-theoretic
objects and morphisms produce it.  Saito's theory provides the appropriate
setting: mixed Hodge modules carry nearby- and vanishing-cycle functors,
canonical and variation morphisms, Tate twists, monodromy, Verdier duality,
and an exact faithful realization functor
$\rat:MHM(X_0)\to\Perv(X_0;\Q)$
\cite{SaitoMHM,SaitoDuality,SaitoYoungGuide}.  What is not automatic is
that the already constructed perverse extension lifts with the desired
Hodge data.  Exactness and faithfulness of realization do not make
$\rat$ full, and the rational skyscraper appearing at the node does not by
itself determine its Tate normalization or its position in a
mixed-Hodge-module extension.

This distinction also clarifies the role of Saito's principal-divisor
formalism.  The divisor geometry belongs to
$X_0\subset\mathcal X$ and supplies the nearby-/vanishing-cycle mixed Hodge
modules together with the $\can^H/\Var^H/N$ package
\cite{SaitoMHM}.  The object ultimately sought, however, lies on the
central fiber itself.  Its construction therefore also involves the
open--closed decomposition $U\hookrightarrow X_0\hookleftarrow\Sigma$,
which organizes the smooth and point-supported constituents of the
refinement.  These are related but categorically distinct operations, and
we keep them separate throughout.

Accordingly, the new content of the present paper is not another
construction of the underlying rational perverse object.  It is the
mixed-Hodge-module refinement of that object: the determination of the
local Hodge block, its Tate normalization, the canonical incoming and
outgoing maps, the monodromy factorization, the relevant Yoneda
orientations, and the finite-node assembly that produces the global
refinement.

\subsection{Focused related work}

The sheaf-theoretic foundation is the theory of perverse sheaves and
recollement developed by Beilinson--Bernstein--Deligne
\cite{BBD}.  For spaces with isolated singular strata, the
MacPherson--Vilonen description and its quiver-theoretic refinements give
a particularly effective linear-algebraic model for the interaction
between the open stratum and point-supported data
\cite{MacPhersonVilonen1986,GMV1996}.  These results provide the natural
language for the local extension problem used below and for the comparison
between the corrected object and the intersection complex.

The local geometric input comes from the classical theory of isolated
hypersurface singularities.  For a threefold ordinary double point, the
Milnor fiber has the homotopy type of $S^3$, so the reduced rational
vanishing cohomology is one-dimensional and concentrated in the middle
degree \cite{MilnorSingularPoints,DimcaSheaves}.  That topological
statement determines the rank of the local sector but not its Hodge type.
The latter is governed by the mixed Hodge structure on vanishing
cohomology and by the spectrum of the singularity; for the
quasihomogeneous ordinary double point, the spectrum and Hodge filtration
provide the local calculation used in Section~4
\cite{Steenbrink1976Limits,Steenbrink1989Spectrum}.

At the global geometric level, nodal threefolds, their smoothings, and the
relations among exceptional and vanishing cycles form a classical part of
the theory of conifold transitions
\cite{Clemens1983DoubleSolids,Friedman1986ThreefoldDoublePoints,
NamikawaSteenbrink1995}.  An important point for the present paper is that
the collection of local nodes and the image of their vanishing cycles in
global cohomology are different kinds of data.  The former gives a
canonical finite support decomposition; the latter may have smaller rank
because global relations among the vanishing cycles can occur.  Our
finite-node construction is organized so that no independence assumption
is imposed at this stage.

The Hodge-theoretic framework is Saito's theory of mixed Hodge modules
\cite{SaitoMHM,SaitoDuality}.  We use the abelian category $MHM(X)$, the
exact faithful realization functor
$\rat:MHM(X)\to\Perv(X;\Q)$, the Hodge-module nearby- and vanishing-cycle
functors, the canonical and variation morphisms, strictness for the Hodge
filtration, monodromy, and Verdier duality.  The corresponding
nearby-cycle mixed Hodge structures are compatible with the classical
limiting-Hodge formalism of Schmid and Steenbrink
\cite{Schmid1973VHS,Steenbrink1976Limits}.  This compatibility is what
allows the finite node-supported Hodge sector constructed below to be
compared with an independently defined limiting mixed Hodge structure.

A useful methodological precedent is the work of Banagl--Budur--Maxim
\cite{BanaglBudurMaxim}.  Their intersection-space complex is a different
perverse object from the corrected extension considered here, and none of
our local ODP or finite-node theorems follows from their construction.
Nevertheless, their work illustrates an important general principle:
perverse data arising from a degeneration can sometimes be promoted to
mixed Hodge modules, after which their hypercohomology acquires canonical
mixed Hodge structures.  The present paper follows that broad philosophy
while solving a different local gluing problem and producing a different
mixed-Hodge-module object.

These are the technical inputs on which the proofs rest.  Broader
connections with conifold transitions, perverse schobers, spherical
monodromy, wall-crossing structures, and physical light sectors motivate
the surrounding program, but they are not used as hypotheses in the
arguments below.

\subsection{Physical and categorical motivation}

The geometric motivation comes from conifold degenerations of
Calabi--Yau threefolds.  At a threefold ordinary double point, a
middle-dimensional three-sphere collapses in the smoothing, and the
associated Picard--Lefschetz transformation records the resulting
rank-one monodromy contribution
\cite{MilnorSingularPoints,Clemens1983DoubleSolids,
Friedman1986ThreefoldDoublePoints}.  With several nodes, one obtains a
finite collection of such local sectors.  Their existence is local, while
their survival in global middle homology is not: relations among the
vanishing cycles may reduce the rank of their global image
\cite{Friedman1986ThreefoldDoublePoints,NamikawaSteenbrink1995}.
This local/global distinction is one of the organizing principles of the
paper.

There is also a categorical motivation.  The rank-one
Picard--Lefschetz phenomenon underlying an ordinary double point has a
categorical counterpart in the spherical-monodromy picture developed in
the earlier conifold work \cite{RahmanSchoberPaper}.  From that
perspective, the corrected perverse object provides a sheaf-theoretic
carrier for the local monodromy sector rather than merely a record of its
dimension.  The present paper does not attempt to construct the full
schober or higher-categorical structure.  Its purpose is more basic and,
for the subsequent program, more foundational: to determine exactly what
Hodge-theoretic structure is carried by the corrected degeneration object
before any further categorical or physical interpretation is imposed.

\subsection{Main results}

We now state the principal results.  The normalization conventions for
nearby and vanishing cycles, the Tate convention, and the distinction
between divisor gluing and central-fiber recollement are fixed in
Sections~2 and~3.  Two point-supported Hodge objects will appear below
with the same underlying rational skyscraper but different categorical
roles.  The first is the interface block selected by the
$\can^H/\Var^H/N$ geometry; the second is the Tate-normalized point term
appropriate to a chosen Yoneda orientation.  Keeping these objects
separate is essential to the theorem package.

\begin{theorem}[Local Hodge refinement of the ODP variation cone]
\label{thm:intro-local}
Let $\pi:\mathcal X\to\Delta$ be a one-parameter degeneration of complex
threefolds with a single ordinary double point $p$ on the central fiber.
Let $\mathcal P_{\mathrm{loc}}$ denote the corrected local perverse object
represented by
\[
\mathcal P_{\mathrm{loc}}
\simeq
\operatorname{Cone}(\operatorname{var}).
\]
Then the nearby- and vanishing-cycle objects in Saito's category determine
a mixed-Hodge-module refinement $\mathcal P^H_{\mathrm{loc}}$ satisfying
\[
\operatorname{rat}(\mathcal P^H_{\mathrm{loc}})
\simeq
\mathcal P_{\mathrm{loc}}.
\]
The local $\operatorname{can}/\operatorname{Var}$ interface determines a
rank-one point-supported gluing block
\[
W_{\mathrm{loc}}^H
\simeq
i_*\Q^H_{\{p\}}(-1),
\]
together with canonical incoming and outgoing morphisms satisfying
$v_{\mathrm{loc}}u_{\mathrm{loc}}=N$.  The Tate normalization of the raw
vanishing line and the Yoneda orientation of the resulting extension are
determined separately by the local ODP calculation and Verdier duality.
\end{theorem}

The first theorem therefore identifies not only the rational rank-one
correction but the Hodge object through which the local monodromy
factorizes.  The passage to several nodes is then support-theoretic.

\begin{theorem}[Finite nodewise assembly]
\label{thm:intro-support}
Let the central fiber have ordinary double points
$\Sigma=\{p_1,\dots,p_r\}$.  The local point-supported Hodge blocks
assemble canonically to
\[
Q_\Sigma^H
=
\bigoplus_{k=1}^r W_k^H
\simeq
\bigoplus_{k=1}^r i_{k*}\Q^H_{\{p_k\}}(-1).
\]
Under rational realization,
\[
\operatorname{rat}(Q_\Sigma^H)
\simeq
\bigoplus_{k=1}^r i_{k*}\Q_{\{p_k\}},
\]
recovering the finite collection of local skyscraper terms in the
corrected perverse construction.  This direct sum records the formal local
node sectors.  No assertion is made that its rank $r$ equals the dimension
of the image of the corresponding vanishing cycles in global cohomology;
global relations among the vanishing cycles are allowed.
\end{theorem}

The local blocks and their attaching maps can then be assembled on the
central fiber.

\begin{theorem}[Global mixed-Hodge-module refinement and realization]
\label{thm:intro-global}
The finite collection of local ODP Hodge data and attaching morphisms
determines a global mixed Hodge module
\[
\mathcal P^H\in MHM(X_0)
\]
whose rational realization is the corrected finite-node perverse object:
\[
\operatorname{rat}(\mathcal P^H)
\simeq
\mathcal P.
\]
Its restriction to a sufficiently small neighborhood of each node is the
local mixed-Hodge-module refinement of
Theorem~\ref{thm:intro-local}.  The corresponding short-exact-sequence
description depends on the chosen Yoneda orientation and its Tate
normalization; these are determined explicitly in
Appendix~\ref{app:dual-extensions}.
\end{theorem}

Finally, the Hodge module obtained from the corrected extension can be
compared with the limiting mixed Hodge structure carried independently by
nearby cycles.

\begin{theorem}[Hypercohomological and limiting-Hodge comparison]
\label{thm:intro-lmhs}
The nearby-cycle mixed Hodge module carries the canonical limiting mixed
Hodge structure associated with the degeneration
\cite{Schmid1973VHS,Steenbrink1976Limits,SaitoMHM}.  The global outgoing
morphism from the finite node-supported Hodge block induces a canonical
morphism
\[
\lambda_\Sigma:
V_{\Sigma,\mathrm{loc}}^H
\longrightarrow
\mathbb H^0(X_0,\Psi^H(-1)),
\]
whose image is the limiting-Hodge contribution detected by the
finite-node construction.

After a Yoneda orientation of the corrected mixed-Hodge-module extension
has been fixed, hypercohomology also produces the associated long exact
sequence in mixed Hodge structures and its degreewise connecting
morphisms.  These comparisons do not identify the formal rank-$r$
nodewise direct sum with a rank-$r$ subspace of global limiting
cohomology unless such independence is established separately.
\end{theorem}

Taken together, these results give a local-to-global Hodge-theoretic
refinement of the corrected perverse variation-cone construction.  The
local theorem identifies the Hodge block selected by the
$\can^H/\Var^H/N$ geometry; finite support organizes those blocks
nodewise; the global theorem reconstructs the corrected object in
$MHM(X_0)$; and the limiting-Hodge comparison records which combinations
of the formal local sectors survive after passage to global cohomology.
Throughout, the construction keeps distinct four pieces of information
that rational realization tends to collapse: local support, Tate
normalization, Yoneda orientation, and global cohomological survival.

\subsection{Proof strategy}

The proof follows the geometry from the node outward.

First, we fix the nearby-/vanishing-cycle normalization and analyze a
single threefold ordinary double point.  The Milnor-fiber calculation
determines the rank of the local vanishing sector
\cite{MilnorSingularPoints}, while the Steenbrink spectrum and Hodge
filtration determine its Hodge-theoretic type
\cite{Steenbrink1976Limits,Steenbrink1989Spectrum}.  Inside Saito's
nearby-/vanishing-cycle package we then isolate the relevant
$\can^H/\Var^H$ interface, form its canonical quotient, and construct the
incoming and outgoing maps.  Their composite is shown to recover the
nilpotent monodromy morphism exactly.

Second, we compare this Hodge-theoretic local construction with the
corrected perverse extension already known on the rational side.  The
MacPherson--Vilonen description provides the local perverse model
\cite{MacPhersonVilonen1986,GMV1996}, while realization identifies the
rational shadow of the constructed mixed Hodge module.  At this stage the
two Yoneda orientations must be distinguished: Verdier duality and the
local stalk/costalk Hodge structures determine which Tate-normalized
point object can occur in each orientation.  The detailed bookkeeping is
collected in the appendices.

Third, we pass from one node to finitely many.  Because the singular set
is finite and its points are disjoint, the local point-supported Hodge
blocks and their maps assemble canonically by finite direct sum.  This
produces the global node-supported datum and, from it, the global
variation-cone object.  The realization theorem is then proved globally,
while the formal nodewise decomposition is kept separate from possible
relations among the corresponding vanishing cycles in global cohomology.

Finally, we apply hypercohomology.  The corrected forward extension gives
a long exact sequence of mixed Hodge structures and its connecting
morphisms, whereas the node-supported outgoing map gives a separate
comparison with the limiting mixed Hodge structure carried by nearby
cycles.  Keeping these two constructions distinct prevents a
sheaf-theoretic Yoneda extension from being silently replaced by an
unqualified extension of total graded hypercohomology.  The resulting
comparison is therefore degreewise, functorial, and compatible with the
local-to-global geometry of the degeneration.

\subsection{Scope and organization}

The scope of the paper is deliberately narrow.  We consider
one-parameter degenerations of complex threefolds whose singular fibers
have finitely many ordinary double points, and we study the
mixed-Hodge-module refinement of the corrected perverse variation-cone
object associated with such a degeneration.  We do not attempt a general
classification of mixed-Hodge-module extensions, a treatment of arbitrary
higher-dimensional singular strata, or a complete description of every
piece of the limiting mixed Hodge structure.  The advantage of this
restriction is that the ordinary-double-point sector can be computed
explicitly enough to separate the rational, Hodge-theoretic, and global
cohomological layers of the construction.

Section~2 fixes the geometric setup, the normalization conventions for
nearby and vanishing cycles, and the local topology of the ordinary
double point.  Section~3 records the mixed-Hodge-module machinery used in
the proof, including realization, monodromy, Tate conventions, Verdier
duality, and Saito's divisor-gluing formalism.  Section~4 carries out the
local ODP calculation and constructs the local mixed-Hodge-module
refinement.  Section~5 assembles the finite point-supported sectors and
analyzes the nodewise extension spaces.  Section~6 constructs the global
mixed Hodge module and proves the realization theorem.  Section~7
compares the resulting Hodge-theoretic data with the limiting mixed Hodge
structure through hypercohomology.  Section~8 develops several structural
consequences of the theorem package, while Section~9 records the resulting
conclusions and directions for further work.

The appendices isolate the technical points whose details would otherwise
interrupt the main argument.  Appendix~\ref{app:dual-extensions} computes
the two local Yoneda orientations and their behavior under Verdier
duality.  The remaining appendices fix the Tate normalization of the ODP
Hodge lines, describe the obstruction to lifting a rational perverse
extension through realization, and classify the normalized local Hodge
lifts used in the rigidity statements.  Together they supply the
duality, weight, and realization bookkeeping needed to make the
local-to-global construction precise.

%===================================================================
\section{Geometric setup and nearby-cycle background}

Throughout the paper we work over $\Q$, and all varieties are complex
algebraic varieties. Let $\pi:\mathcal X\to\Delta$ be a one-parameter
degeneration over a sufficiently small complex disk $\Delta$, with
$\Delta^*=\Delta\setminus\{0\}$. We write $X_t:=\pi^{-1}(t)$ and assume
that $X_t$ is a smooth complex threefold for every $t\neq0$. The central
fiber is denoted by $X_0:=\pi^{-1}(0)$.

Several spaces enter the construction, and it is important to distinguish
them from the outset. The morphism $\pi$ is defined on the total space
$\mathcal X$, whereas the nearby- and vanishing-cycle objects associated
with $\pi$ are supported on the central fiber $X_0$. In the conifold
setting, their nontrivial local singular contributions are concentrated
on the finite singular locus $\Sigma\subset X_0$. The relevant geometric
configuration is therefore
\[
\mathcal X\xrightarrow{\ \pi\ }\Delta,
\qquad
X_0=\pi^{-1}(0)\hookrightarrow\mathcal X,
\qquad
\Sigma\hookrightarrow X_0.
\]

The corrected perverse object studied below is an object on the singular
threefold $X_0$, and its mixed-Hodge-module refinement will likewise lie
in $MHM(X_0)$. The role of the total space $\mathcal X$ is different:
it provides the ambient geometry in which the degeneration, and hence the
nearby-/vanishing-cycle formalism, is defined.

In the principal geometric situation considered in this paper,
$\mathcal X$ is assumed smooth in a neighborhood of $X_0$. Since the
fibers have complex dimension $3$ and the base has complex dimension $1$,
the total space has complex dimension $4$ near the central fiber. This
dimension shift is relevant to the normalization of the constant objects
on $\mathcal X$ and of the nearby- and vanishing-cycle functors landing
on $X_0$.

Projectivity will be imposed only when projective Hodge-theoretic
arguments require it. Until then, the construction is formulated in
terms of constructible complexes, perverse sheaves, nearby and vanishing
cycles, and mixed Hodge modules associated with the algebraic
degeneration.

\subsection{Spaces, inclusions, and normalization conventions}

We use the middle-perversity $t$-structure on $D^b_c(Y;\Q)$ and denote
its heart by $\Perv(Y;\Q)$ \cite{BBD,KS}. If $Y$ is smooth of complex
dimension $d$, then $\Q_Y[d]$ is perverse. Thus, when $\mathcal X$ is
smooth of complex dimension $4$, the perverse constant object on the
total space is $\Q_{\mathcal X}[4]$, whereas on the smooth locus of the
three-dimensional central fiber the corresponding object is $\Q_U[3]$.

These objects occur on different spaces and play different roles. Nearby
and vanishing cycles provide the passage from the degeneration on
$\mathcal X$ to constructible objects on $X_0$. Let
$i_0:X_0\hookrightarrow\mathcal X$ denote the inclusion of the central
fiber, and let $j_\Delta:\mathcal X^*\hookrightarrow\mathcal X$ denote
the inclusion of $\mathcal X^*:=\pi^{-1}(\Delta^*)$. These maps form the
open--closed decomposition associated with the central divisor
$X_0\subset\mathcal X$.

A second open--closed decomposition will arise inside the central fiber
from its finite singular set. The two decompositions are geometrically
and categorically distinct. In particular, $\mathcal X^*$ is the
complement of the central divisor in the total space, whereas the smooth
locus $U=X_0\setminus\Sigma$ is an open subset of the central fiber.
This distinction will be essential when the divisor-gluing formalism is
introduced in Section~3.

We also fix the normalization of nearby and vanishing cycles. Let
$\psi_{\pi,1}^{\mathrm{cl}}$ and
$\phi_{\pi,1}^{\mathrm{cl}}$ denote the classical derived nearby- and
vanishing-cycle functors. Their perverse-normalized counterparts are
\[
{}^p\psi_{\pi,1}
=
\psi_{\pi,1}^{\mathrm{cl}}[-1],
\qquad
{}^p\phi_{\pi,1}
=
\phi_{\pi,1}^{\mathrm{cl}}[-1]
\]
\cite{BBD,DimcaSheaves}. If
\[
\operatorname{var}^{\mathrm{cl}}:
\phi_{\pi,1}^{\mathrm{cl}}(K)
\longrightarrow
\psi_{\pi,1}^{\mathrm{cl}}(K)
\]
is the classical variation morphism, then its perverse-normalized form is
${}^p\!\operatorname{var}
=\operatorname{var}^{\mathrm{cl}}[-1]$. Consequently,
\[
\operatorname{Cone}\!\left(
\operatorname{var}^{\mathrm{cl}}
\right)[-1]
\simeq
\operatorname{Cone}\!\left(
{}^p\!\operatorname{var}
\right).
\]

This identity is the normalization bridge between the notation used in
the earlier corrected perverse construction and the convention adopted
in the present paper. Section~3 imposes the same perverse normalization
simultaneously on $\psi_\pi$, $\phi_\pi$, $\can$, $\var$, and their
mixed-Hodge-module counterparts. Once this convention has been fixed, we
suppress the superscript ${}^p$ whenever no ambiguity can arise. Thus
$\psi_\pi$, $\phi_\pi$, $\can$, and $\var$ will denote the
perverse-normalized objects and morphisms throughout the remainder of the
paper.

The nearby- and vanishing-cycle objects obtained in this way are objects
on $X_0$, and the cone defining the corrected perverse construction is
formed in $D^b_c(X_0;\Q)$. In the suppressed perverse-normalized notation,
that cone is written without any further cohomological shift.

\subsection{Multi-node conifold degenerations}

We now assume that the central fiber $X_0$ has finitely many ordinary
double points
\[
\Sigma=\{p_1,\dots,p_r\}\subset X_0.
\]
We refer to such a family as a finite multi-node conifold degeneration.
Let $U:=X_0\setminus\Sigma$ denote the smooth locus of the central fiber,
and write
\[
j:U\hookrightarrow X_0,
\qquad
i_\Sigma:\Sigma\hookrightarrow X_0,
\qquad
i_k:\{p_k\}\hookrightarrow X_0
\]
for the natural inclusions.

The immersion $j:U\hookrightarrow X_0$ belongs to the recollement of the
singular central fiber by its smooth locus and finite singular set. It
must therefore be distinguished from
$j_\Delta:\mathcal X^*\hookrightarrow\mathcal X$, which belongs to the
divisor geometry of the degeneration. The distinction is structural:
the former organizes data already living on $X_0$, whereas the latter
governs passage across the central divisor inside the total space.

Because $\Sigma$ is finite and its points are pairwise disjoint, objects
supported at the nodes admit a canonical finite direct-sum decomposition.
This support-theoretic fact is the mechanism by which the local
ordinary-double-point contributions will be assembled. It should not,
however, be confused with a statement about the linear independence of
the corresponding vanishing cycles in the homology or cohomology of a
nearby smooth fiber. Global relations among vanishing cycles may occur.

The corrected perverse construction in the single-node case was
established in \cite{RahmanSchoberPaper}. Its formulation in terms of
nearby and vanishing cycles, together with the associated
Hodge-theoretic refinement problem, was developed further in
\cite{RahmanPerverseNearbyCycles}. The purpose of the present paper is to
construct that refinement for a finite collection of ordinary double
points.

\subsection{Local topology of an ordinary double point}

Let $p\in X_0$ be an ordinary double point. Locally near $p$, the
threefold $X_0$ is an isolated hypersurface singularity. Its Milnor fiber
$F_p$ has the homotopy type of $S^3$
\cite{MilnorSingularPoints,DimcaSheaves}, and therefore
\[
\widetilde H^k(F_p;\Q)\cong
\begin{cases}
\Q, & k=3,\\
0, & k\neq3.
\end{cases}
\]
Thus the reduced local vanishing cohomology is one-dimensional and
concentrated in the middle degree.

After perverse normalization, the corresponding local vanishing-cycle
contribution on $X_0$ is point-supported and has a one-dimensional
underlying rational stalk at $p$. For
$\Sigma=\{p_1,\dots,p_r\}$, the local singular contribution is therefore
organized as a finite collection of rank-one point-supported rational
sectors, one for each node.

This conclusion is purely topological. The dimension of the rational
vanishing space does not determine the Hodge type, weight, or Tate
normalization of the corresponding mixed-Hodge-module contribution.
Those structures require an additional local Hodge-theoretic calculation,
which will be carried out in Section~4. In particular, the appearance of
a Tate object such as $\Q^H(-1)$ will be derived from the local geometry
and the mixed-Hodge-module formalism rather than inferred from the
rank-one Milnor-fiber calculation alone.

\subsection{Nearby cycles, vanishing cycles, and variation}

Let $K$ denote the normalized constructible object on the total space to
which nearby- and vanishing-cycle functors are applied. Its precise shift
is fixed in Section~3 together with the corresponding functorial
normalization. The resulting perverse-normalized objects $\psi_\pi(K)$
and $\phi_\pi(K)$ lie on $X_0$. Unless the superscript
$\mathrm{cl}$ is displayed explicitly, the symbols
$\psi_\pi$, $\phi_\pi$, $\can$, and $\var$ will henceforth refer to the
perverse-normalized functors and morphisms.

The nearby-/vanishing-cycle formalism supplies functorial canonical and
variation morphisms
\[
\can_K:\psi_\pi(K)\longrightarrow\phi_\pi(K),
\qquad
\var_K:\phi_\pi(K)\longrightarrow\psi_\pi(K),
\]
which encode the interaction among specialization, vanishing cohomology,
and monodromy \cite[Chapter~6]{DimcaSheaves}. On the unipotent summand,
the corresponding nilpotent logarithm of monodromy will be denoted by
$N$. The precise relation among $N$, $\can$, and $\var$ in the
normalization used here is fixed in Section~3 before these morphisms are
used in the mixed-Hodge-module construction.

For a finite-node conifold degeneration, the singular contribution of
the vanishing-cycle object is supported on $\Sigma$. The Milnor-fiber
calculation shows that the reduced rational contribution at each ordinary
double point has rank one. Thus the singular information relevant to the
corrected construction is concentrated in finitely many point-supported
local sectors, whereas the nearby-cycle object retains the global
specialization data of the degeneration.

\subsection{The corrected perverse object}

With the perverse normalization fixed above, the corrected object is
defined by the variation cone
\[
\mathcal P
:=
\Cone\!\left(
\var_K:\phi_\pi(K)\longrightarrow\psi_\pi(K)
\right).
\]
If $\var_K^{\mathrm{cl}}$ denotes the corresponding classical derived
variation morphism, then equivalently
\[
\mathcal P
\simeq
\Cone\!\left(
\var_K^{\mathrm{cl}}
\right)[-1].
\]
These two formulas describe the same object under the normalization
bridge established above. The cone is formed in $D^b_c(X_0;\Q)$; in
particular, $\mathcal P$ is an object on the singular central fiber, not
on the total space $\mathcal X$.

The single-node object was constructed and shown to be perverse in
\cite{RahmanSchoberPaper}, where its point-supported correction was also
identified. The subsequent analysis of
\cite{RahmanPerverseNearbyCycles} recast the construction in the
nearby-/vanishing-cycle formalism and examined its relation to duality
and limiting Hodge theory.

In the present paper, $\mathcal P$ is therefore fixed in advance on the
rational perverse side. The problem is to construct a mixed Hodge module
$\mathcal P^H\in MHM(X_0)$ satisfying
$\rat(\mathcal P^H)\simeq\mathcal P$. The desired refinement must recover
this particular variation-cone object, rather than merely reproduce its
support or the dimensions of its local stalks.

The Hodge-theoretic structure of the point-supported local contribution
is not assumed. Its Tate type, its role in the local gluing datum, and its
position in the relevant Yoneda extension are part of the problem to be
determined. Section~4 derives these data from the ordinary-double-point
Hodge calculation and from the mixed-Hodge-module canonical and variation
morphisms. This separation allows the local gluing block and the eventual
extension orientation to be identified independently rather than imposed
at the outset.

\subsection{Picard--Lefschetz data and local versus global vanishing sectors}

For each node $p_k\in\Sigma$, choose a corresponding local vanishing cycle
$\delta_k\in H_3(X_t;\Z)$. Its local monodromy is governed by the
Picard--Lefschetz formula
\cite{MilnorSingularPoints,DimcaSheaves}. Each ordinary double point
therefore determines a rank-one local vanishing sector.

For the multi-node problem, the formal collection of these local sectors
must be distinguished from its image in global middle homology. Define
\[
V_{\mathrm{loc}}
:=
\bigoplus_{k=1}^r \Q\,\delta_k.
\]
By construction, $\dim_\Q V_{\mathrm{loc}}=r$. The local cycles determine
a natural map to the middle homology of a nearby smooth fiber, and we set
\[
V_{\mathrm{glob}}
:=
\operatorname{Im}\!\left(
V_{\mathrm{loc}}\longrightarrow H_3(X_t;\Q)
\right).
\]
Hence $\dim_\Q V_{\mathrm{glob}}\leq r$, with strict inequality precisely
when nontrivial linear relations occur among the global images of the
local vanishing cycles.

The same distinction reappears sheaf-theoretically. At the level of
support on $X_0$, each node contributes a rank-one point-supported local
sector, and these sectors admit a canonical finite direct-sum
decomposition indexed by $\Sigma$. Passage to global hypercohomology may
introduce relations among their images. Accordingly, the arguments below
use the nodewise direct sum as local input without identifying it a priori
with an $r$-dimensional subspace of global vanishing cohomology.

%-------------------------------------------------------------
\section{Mixed Hodge modules and divisor gluing}
\label{sec:mixed-hodge-module-div-gluing}

We now fix the mixed-Hodge-module conventions required for the
construction of the refinement $\mathcal P^H\in MHM(X_0)$.  Only the
parts of Saito's theory used in the argument will be recalled: the
realization functor, nearby and vanishing cycles, the canonical and
variation morphisms, Tate twists and monodromy, and the
principal-divisor gluing formalism \cite{SaitoMHM,SaitoDuality}.  A
second purpose of this section is to separate two open--closed
decompositions that enter the construction in different ways.  The
first is the divisor decomposition associated with
$X_0\subset\mathcal X$; the second is the decomposition of the central
fiber by its smooth locus and finite singular set.  These constructions
take place on different ambient spaces and will be kept distinct
throughout.

As in Section~2, let $\pi:\mathcal X\to\Delta$ be a one-parameter
degeneration of complex threefolds, let $X_0=\pi^{-1}(0)$ be its central
fiber, and let $\Sigma\subset X_0$ denote the finite set of ordinary
double points.  We assume that $\mathcal X$ is smooth in a neighborhood
of $X_0$.  Hence $\dim_{\C}\mathcal X=4$ near the central fiber, whereas
$\dim_{\C}X_0=3$.

\subsection{Mixed Hodge modules and realization}

For a complex algebraic variety $X$, Saito constructs an abelian category
$MHM(X)$ of mixed Hodge modules together with an exact faithful
realization functor
$\rat:MHM(X)\to\Perv(X;\Q)$
\cite{SaitoMHM,SaitoDuality}.  Thus every mixed Hodge module has an
underlying rational perverse sheaf, and every short exact sequence in
$MHM(X)$ remains exact after application of $\rat$.

Faithfulness of realization does not imply that the Hodge-theoretic
structure is determined by its rational perverse shadow.  In particular,
Tate twists are invisible after applying $\rat$, and the existence of a
rational perverse extension does not by itself imply the existence of a
mixed-Hodge-module lift with prescribed endpoint Hodge structures
\cite{SaitoMHM}.  Nor, when such a lift exists, is its Tate normalization
determined by the underlying rational skyscraper term.  This distinction
will be essential in the local ordinary-double-point calculation of
Section~4.

If $Y$ is smooth of complex dimension $d$, we denote by $\Q_Y^H[d]$ the
constant pure Hodge module in perverse normalization; its rational
realization is $\Q_Y[d]$ \cite{SaitoMHM}.  In particular, since
$\mathcal X$ is smooth of complex dimension $4$ near $X_0$, the
perverse constant Hodge module on the total space is
\[
\mathcal K^H:=\Q_{\mathcal X}^H[4],
\qquad
\rat(\mathcal K^H)=\Q_{\mathcal X}[4].
\]
On the smooth locus $U=X_0\setminus\Sigma$ of the three-dimensional
central fiber, the corresponding perverse constant Hodge module is
$\Q_U^H[3]$.  The dimension shifts therefore refer to different spaces:
$\Q_{\mathcal X}^H[4]$ belongs to the total-space geometry, whereas
$\Q_U^H[3]$ belongs to the smooth part of the central fiber.

The intersection complex of $X_0$ likewise carries its canonical
Hodge-module refinement, denoted $IC^H_{X_0}$, with
$\rat(IC^H_{X_0})\simeq IC_{X_0}$ \cite{SaitoMHM}.  Mixed Hodge
structures on a point give, by closed pushforward, point-supported mixed
Hodge modules on $X_0$.  Such objects will enter the construction only
after their local Hodge type and Tate normalization have been determined.

\subsection{Nearby and vanishing cycles in \texorpdfstring{$MHM$}{MHM}}

We use the perverse-normalized unipotent nearby- and vanishing-cycle
functors in both the constructible and mixed-Hodge-module categories.
Thus $\psi_{\pi,1}^H$ and $\phi_{\pi,1}^H$ take the relevant
mixed-Hodge-module objects on $\mathcal X$ to mixed Hodge modules on
$X_0$, and realization is compatible with the corresponding
perverse-normalized constructible functors
$\psi_{\pi,1}$ and $\phi_{\pi,1}$
\cite{SaitoMHM,SaitoDuality}.  Applied to
$\mathcal K^H=\Q_{\mathcal X}^H[4]$, this gives
\[
\rat\!\left(\psi_{\pi,1}^H(\mathcal K^H)\right)
\simeq
\psi_{\pi,1}\!\left(\Q_{\mathcal X}[4]\right),
\qquad
\rat\!\left(\phi_{\pi,1}^H(\mathcal K^H)\right)
\simeq
\phi_{\pi,1}\!\left(\Q_{\mathcal X}[4]\right).
\]
Accordingly, the shift $[4]$ belongs to the perverse constant object on
the smooth total space, while the normalized nearby- and vanishing-cycle
functors produce perverse objects on the three-dimensional central
fiber.

For brevity, set
\[
\Psi^H:=\psi_{\pi,1}^H(\mathcal K^H),
\qquad
\Phi^H:=\phi_{\pi,1}^H(\mathcal K^H),
\]
and write $\Psi:=\rat(\Psi^H)$ and $\Phi:=\rat(\Phi^H)$.  Saito's
nearby-/vanishing-cycle formalism carries canonical and variation
morphisms
\[
\can^H:\Psi^H\longrightarrow\Phi^H,
\qquad
\Var^H:\Phi^H\longrightarrow\Psi^H(-1),
\]
whose rational realizations are the corresponding canonical and
variation morphisms of perverse sheaves
\cite{SaitoMHM,SaitoDuality}.  The Tate twist in the target of
$\Var^H$ is invisible after rational realization, although it remains
part of the Hodge-theoretic typing of the morphism.

We retain the perverse normalization fixed in Section~2.  If
$\psi_{\pi,1}^{\mathrm{cl}}$ and
$\phi_{\pi,1}^{\mathrm{cl}}$ denote the classical derived nearby- and
vanishing-cycle functors, then
\[
{}^p\psi_{\pi,1}
=
\psi_{\pi,1}^{\mathrm{cl}}[-1],
\qquad
{}^p\phi_{\pi,1}
=
\phi_{\pi,1}^{\mathrm{cl}}[-1],
\]
and the corresponding variation morphisms satisfy
${}^p\!\var=\var^{\mathrm{cl}}[-1]$
\cite{BBD,DimcaSheaves}.  Hence
\[
\Cone(\var^{\mathrm{cl}})[-1]
\simeq
\Cone({}^p\!\var).
\]
After suppressing the superscript ${}^p$, the corrected rational object
is therefore written as $\mathcal P\simeq\Cone(\var)$, while on the
mixed-Hodge-module side we introduce the derived candidate
\[
\widetilde{\mathcal P}^{\,H}
:=
\Cone(\Var^H).
\]
At this stage $\widetilde{\mathcal P}^{\,H}$ is an object of the derived
category.  The local result of Section~4 and its global counterpart in
Section~6 prove that the relevant variation cone is concentrated in the
mixed-Hodge-module heart; the resulting heart object is denoted by
$\mathcal P^H$.

The existence of the morphism $\Var^H$ alone does not identify the local
Hodge structure carried by this cone.  The ordinary-double-point
calculation of Section~4 determines the point-supported contribution,
constructs the relevant local maps, and verifies that rational
realization recovers the corrected variation-cone object fixed in
Section~2.

\subsection{Monodromy and Tate conventions}

We use the standard rational Tate convention in which $\Q^H(1)$ has
weight $-2$.  More generally,
$\operatorname{wt}\Q^H(m)=-2m$, so tensoring a pure Hodge structure of
weight $w$ with $\Q^H(m)$ shifts its weight to $w-2m$.  In particular,
the twist $(-1)$ raises the weight by $2$.

On the unipotent nearby-cycle object, the nilpotent logarithm of
monodromy is represented in the mixed-Hodge-module category by a
morphism
\[
N:\Psi^H\longrightarrow\Psi^H(-1).
\]
With the conventions adopted above, the canonical and variation
morphisms satisfy
\[
\Var^H\circ\can^H=N
\]
on the unipotent nearby-cycle object
\cite{SaitoMHM,SaitoDuality}.  The corresponding companion relation on
the vanishing-cycle side is understood with the appropriate monodromy
action.  We retain the Tate twist in the target of $N$ throughout,
rather than treating monodromy as an untyped endomorphism.

After application of $\rat$, one has
$\rat(\Psi^H(-1))\simeq\rat(\Psi^H)$ as rational perverse sheaves.
Thus the morphism $N:\Psi^H\to\Psi^H(-1)$ retains Hodge-theoretic
information that is lost after realization.  In particular, the Tate
type of the point-supported ordinary-double-point contribution cannot be
deduced solely from the rational variation-cone construction.

\subsection{Saito's principal-divisor gluing formalism}

We next recall the form of Saito's principal-divisor gluing formalism
needed below \cite[Prop.~0.3]{SaitoMHM}.  Let $g:X\to\C$ be a regular
function, set $Y=g^{-1}(0)$ and $X^\circ=X\setminus Y$, and write
$j_g:X^\circ\hookrightarrow X$ and $i_g:Y\hookrightarrow X$ for the
natural inclusions.  The gluing problem is formulated relative to this
open--closed decomposition of the ambient variety $X$.

In the unipotent setting, the relevant data consist of an object
$\mathcal M'$ on $X^\circ$, an object $\mathcal M''$ on $Y$, and
morphisms
\[
u:\psi_{g,1}^H(\mathcal M')\longrightarrow\mathcal M'',
\qquad
v:\mathcal M''\longrightarrow
\psi_{g,1}^H(\mathcal M')(-1),
\]
satisfying $v\circ u=N$, together with the admissibility conditions in
Saito's category \cite[Prop.~0.3]{SaitoMHM}.  This is the
principal-divisor structure from which the nearby-cycle data used below
are obtained.

For the degeneration $\pi:\mathcal X\to\Delta$, the relevant divisor is
$X_0=\pi^{-1}(0)\subset\mathcal X$, with complement
$\mathcal X^*=\mathcal X\setminus X_0$.  Thus the divisor geometry is
organized by
\[
\mathcal X^*\hookrightarrow\mathcal X\hookleftarrow X_0.
\]
The smooth locus $U=X_0\setminus\Sigma$ is not the complement occurring
in this divisor decomposition.

Inside the central fiber there is instead the open--closed decomposition
\[
U\hookrightarrow X_0\hookleftarrow\Sigma.
\]
This second decomposition governs recollement on $X_0$ and the
point-supported extension problem at the nodes
\cite{BBD,MacPhersonVilonen1986,GMV1996}.  It should therefore be
distinguished from Saito's divisor geometry: the latter produces the
nearby-/vanishing-cycle Hodge data associated with the degeneration,
whereas the former organizes the smooth and singular constituents of an
object already living on the central fiber.

Section~4 uses the divisor-derived $\can^H/\Var^H/N$ package to identify
the local Hodge datum at an ordinary double point.  Sections~5--6 then
assemble the resulting point-supported objects on $X_0$.  No
principal-divisor gluing statement is applied to
$\Sigma\subset X_0$.

\begin{remark}[Role of the divisor-gluing theorem]
\label{rem:role-of-saito-divisor-gluing}
Saito's principal-divisor formalism is applied to
\[
\mathcal X^*\hookrightarrow\mathcal X\hookleftarrow X_0
\]
and supplies the nearby-/vanishing-cycle mixed Hodge modules together
with the $\can^H/\Var^H/N$ structure on the central fiber.  The object
sought in the present paper is instead a refinement
$\mathcal P^H\in MHM(X_0)$ of the already fixed corrected perverse object
$\mathcal P\in\Perv(X_0;\Q)$.  Proposition~0.3 of
\cite{SaitoMHM} is therefore not used as a lifting theorem from
$\Perv(X_0;\Q)$ to $MHM(X_0)$.  Its role is to provide and constrain the
Hodge-theoretic degeneration data from which the central-fiber
refinement is constructed.

The recollement
$U\hookrightarrow X_0\hookleftarrow\Sigma$ is internal to $X_0$ and
organizes the smooth and point-supported constituents of that
refinement.  It is not an application of the principal-divisor theorem
to the finite inclusion $\Sigma\subset X_0$.
\end{remark}

\subsection{Point-supported mixed Hodge modules}

Let $i_k:\{p_k\}\hookrightarrow X_0$ denote the inclusion of a node.
Mixed Hodge modules supported at a point are identified with the
corresponding rational mixed Hodge structures placed at that point
\cite{SaitoMHM}.  In particular, a pure Tate structure
$\Q_{\{p_k\}}^H(m)$ determines the point-supported Hodge module
$i_{k*}\Q_{\{p_k\}}^H(m)$ on $X_0$.

Rational realization forgets the Tate twist:
\[
\rat\!\left(
i_{k*}\Q_{\{p_k\}}^H(m)
\right)
\simeq
i_{k*}\Q_{\{p_k\}}.
\]
Consequently, distinct point-supported Hodge modules may have the same
underlying rational skyscraper perverse sheaf.  The rational
point-supported correction occurring in the corrected perverse object
therefore does not determine, by itself, the Tate normalization of its
mixed-Hodge-module refinement.

Accordingly, the object
$i_{k*}\Q_{\{p_k\}}^H(-1)$ is not assumed at this stage to be the
quotient, subobject, or unique Hodge lift of the rational local
correction.  Section~4 determines the Tate type arising from the local
$\can^H/\Var^H/N$ geometry and then distinguishes that gluing object
from the point-supported Hodge structure occurring in a specified
Yoneda orientation.  The distinction is essential because these objects
may have the same rational realization while occupying different
positions in $MHM(X_0)$.

\subsection{The two extension orientations}

Let $L_p^H$ be a point-supported Hodge module and
$IC_{\mathrm{loc}}^H$ the local intersection-complex Hodge module.  The
two short exact sequences
\[
0\longrightarrow IC_{\mathrm{loc}}^H
\longrightarrow E^H
\longrightarrow L_p^H
\longrightarrow0
\]
and
\[
0\longrightarrow L_p^H
\longrightarrow E^{H,\vee}
\longrightarrow IC_{\mathrm{loc}}^H
\longrightarrow0
\]
belong to different Yoneda extension problems.  They are classified,
respectively, by
$\Ext^1_{MHM}(L_p^H,IC_{\mathrm{loc}}^H)$ and
$\Ext^1_{MHM}(IC_{\mathrm{loc}}^H,L_p^H)$.  There is therefore no
formal implication from the vanishing of one group to the vanishing of
the other.

Verdier duality relates the two orientations, but the relation must be
read together with the corresponding Tate normalization
\cite{SaitoDuality}.  In particular, identifying the Hodge object
appearing in the $\can^H/\Var^H/N$ interface does not by itself
determine which Tate object occurs as the quotient in the forward
orientation.  The local calculation used in Section~4 keeps these two
questions separate and determines the relevant orientation only after
the local Hodge structures have been fixed.

\subsection{The local problem}

The preceding formalism reduces the mixed-Hodge-module refinement to a
specific local problem.  Starting from $\Psi^H$, $\Phi^H$,
$\can^H$, and $\Var^H$, one must identify the point-supported Hodge
datum selected by the ordinary-double-point geometry, construct the
associated incoming and outgoing maps, and verify the required
factorizations and descents inside $MHM(X_0)$.

The rank-one local vanishing sector alone is insufficient for this
purpose.  Its Hodge type and Tate normalization must first be determined;
the canonical and variation morphisms must be correctly typed; the
relevant image, kernel, and quotient objects must be formed in the
mixed-Hodge-module category; and the monodromy factorization must be
verified.  One must then prove that the resulting variation cone is
concentrated in the mixed-Hodge-module heart and that its rational
realization is the corrected perverse object fixed in Section~2.

Section~4 carries out this local construction for a threefold ordinary
double point.  Once the local Hodge block and its morphisms have been
identified, the finite-node problem becomes one of assembling
point-supported data with disjoint support rather than repeating the
local Hodge calculation at each singular point.

%----------------------------------------------------------
\section{Local ODP construction via nearby-cycle quotients}
\label{sec:local-odp}

We now carry out the local mixed-Hodge-module construction at a
threefold ordinary double point.  The central problem is to identify,
within Saito's nearby-/vanishing-cycle formalism, the point-supported
Hodge datum underlying the rank-one correction in the previously
constructed perverse object.  The construction is intrinsic to the
morphisms $\can^H$, $\Var^H$, and the nilpotent monodromy operator $N$;
no independent Hodge lift of the rational skyscraper term is imposed
a priori.

The argument proceeds from the local vanishing geometry.  The Milnor
calculation first identifies the reduced ordinary-double-point sector as
one-dimensional, while the Steenbrink spectrum determines its Hodge
filtration.  We then analyze the corresponding point-supported block
through the MacPherson--Vilonen description, isolate the part of
$\operatorname{Im}(\can^H)$ annihilated by variation, and form the
canonical quotient on which variation becomes injective.  This quotient
is the local interface object $W_{\mathrm{loc}}^H$.  Its incoming and
outgoing maps give an exact factorization of the local monodromy
operator.  Finally, the resulting variation cone is shown to lie in the
mixed-Hodge-module heart, and rational realization recovers the
corrected local perverse object.

The simplification specific to an ordinary double point is the
concentration of the reduced Milnor cohomology in a single degree
\cite{MilnorSingularPoints,DimcaSheaves}.  This concentration collapses
the relevant local cohomology sequences and reduces the
point-supported part of the MacPherson--Vilonen calculation to
one-dimensional linear algebra.  The construction below therefore
depends on the reduced ODP summand and does not require any global
identity of the form
$\operatorname{Im}(\can)=\ker(\Var)$.

\subsection{Local analytic model and notation}

Let
$\pi_{\mathrm{loc}}:\mathcal X_{\mathrm{loc}}\to\Delta$
be a local one-parameter degeneration whose central fiber
$X_{0,\mathrm{loc}}:=\pi_{\mathrm{loc}}^{-1}(0)$ has a single ordinary
double point $p$.  After choosing local analytic coordinates, the
degeneration is represented by
\[
f:\C^4\longrightarrow\C,
\qquad
f(x_1,x_2,x_3,x_4)
=
x_1^2+x_2^2+x_3^2+x_4^2,
\]
with central fiber locally given by $f^{-1}(0)$ and nearby smooth fibers
by $f^{-1}(t)$ for $0<|t|\ll1$
\cite{MilnorSingularPoints,DimcaSheaves}.

Set
$U_{\mathrm{loc}}:=X_{0,\mathrm{loc}}\setminus\{p\}$, and denote the
open and closed inclusions by
$j_{\mathrm{loc}}:U_{\mathrm{loc}}\hookrightarrow X_{0,\mathrm{loc}}$
and
$i_{\mathrm{loc}}:\{p\}\hookrightarrow X_{0,\mathrm{loc}}$.

As fixed in Section~3, the smooth local total space has complex
dimension $4$, and we work with the perverse constant Hodge module
$\mathcal K_{\mathrm{loc}}^H:=\Q_{\mathcal X_{\mathrm{loc}}}^H[4]$.
Define
\[
\Psi_{\mathrm{loc}}^H
:=
\psi_{\pi_{\mathrm{loc}},1}^H(\mathcal K_{\mathrm{loc}}^H),
\qquad
\Phi_{\mathrm{loc}}^H
:=
\phi_{\pi_{\mathrm{loc}},1}^H(\mathcal K_{\mathrm{loc}}^H).
\]
These are mixed Hodge modules on the three-dimensional central fiber.
The canonical and variation morphisms are
\[
\can^H:
\Psi_{\mathrm{loc}}^H
\longrightarrow
\Phi_{\mathrm{loc}}^H,
\qquad
\Var^H:
\Phi_{\mathrm{loc}}^H
\longrightarrow
\Psi_{\mathrm{loc}}^H(-1),
\]
and satisfy
\[
\Var^H\circ\can^H=N
\]
on the unipotent nearby-cycle object
\cite{SaitoMHM,SaitoDuality}.  We denote their rational realizations by
$\Psi_{\mathrm{loc}}$, $\Phi_{\mathrm{loc}}$, $\can$, and $\var$.

All nearby- and vanishing-cycle functors in this section are understood
in the perverse-normalized convention fixed in Sections~2--3.
Accordingly, the corrected local perverse object is written
\[
\mathcal P_{\mathrm{loc}}
=
\Cone\!\left(
\var:
\Phi_{\mathrm{loc}}
\longrightarrow
\Psi_{\mathrm{loc}}
\right).
\]
In the classical derived convention used in the earlier ODP
construction, the same object is represented by
$\Cone(\var^{\mathrm{cl}})[-1]$.  The normalization identity established
in Sections~2--3 identifies these two presentations.

The principal-divisor formalism is applied to
$X_{0,\mathrm{loc}}\subset\mathcal X_{\mathrm{loc}}$.  By contrast, the
decomposition
$U_{\mathrm{loc}}\hookrightarrow
X_{0,\mathrm{loc}}\hookleftarrow\{p\}$
is internal to the central fiber and will be used to analyze its
point-supported contribution.  These two open--closed decompositions
play different categorical roles.

\subsection{Milnor algebra, spectrum, and the local Hodge line}
\label{subsec:odp-spectrum-hodge-line}

We begin with the local calculation determining the dimension and Hodge
filtration of the ordinary-double-point vanishing sector.

\begin{proposition}[Milnor algebra and local vanishing rank]
\label{prop:odp-milnor-algebra}
For $f=x_1^2+x_2^2+x_3^2+x_4^2$, the Jacobian algebra is
\[
\mathcal Q_f
:=
\frac{\C[x_1,x_2,x_3,x_4]}
     {(\partial f/\partial x_1,\ldots,\partial f/\partial x_4)}
\cong
\C.
\]
Consequently, the Milnor number is $\mu=1$, and
\[
\widetilde H^m(F_p;\Q)
\cong
\begin{cases}
\Q, & m=3,\\
0, & m\neq3.
\end{cases}
\]
\end{proposition}

\begin{proof}
Since $\partial f/\partial x_i=2x_i$, the Jacobian ideal is
$(x_1,x_2,x_3,x_4)$, and hence $\mathcal Q_f\cong\C$.  Therefore
$\mu=\dim_\C\mathcal Q_f=1$.  By the Milnor-fiber theorem, the Milnor
fiber of an isolated threefold ordinary double point has the homotopy
type of $S^3$ \cite{MilnorSingularPoints}; the stated reduced cohomology
follows.
\end{proof}

\begin{proposition}[Spectrum of the threefold ordinary double point]
\label{prop:odp-spectrum}
For $f=x_1^2+x_2^2+x_3^2+x_4^2$ with weights $w_i=1/2$, the Steenbrink
spectrum consists of the single spectral number
$\operatorname{Sp}(f)=\{1\}$
\cite{Steenbrink1989Spectrum}.
\end{proposition}

\begin{proof}
The Jacobian algebra is generated by the class of $1$.  For a weighted
homogeneous isolated hypersurface singularity, the spectral number
associated with this monomial is
$\alpha(1)=\sum_i w_i-1$.  Hence
$\alpha(1)=4(1/2)-1=1$.  Since $\mu=1$, this is the entire spectrum.
\end{proof}

\begin{proposition}[Hodge filtration of the local ODP vanishing line]
\label{prop:odp-hodge-filtration}
With the Steenbrink convention
$n-p-1<\alpha\leq n-p$, the spectral number $\alpha=1$ contributes to
\[
\operatorname{Gr}_F^2H^3_{\mathrm{van}}
\cong
\C.
\]
Thus the reduced ODP vanishing sector carries a unique
one-dimensional Hodge-filtered contribution.
\end{proposition}

\begin{proof}
Here $n=3$ and $\alpha=1$.  The inequality
$3-p-1<1\leq3-p$ is satisfied for $p=2$.  Since the Milnor number is
one, the corresponding graded piece is one-dimensional.
\end{proof}

\begin{remark}[Reduced ODP sector]
\label{rem:reduced-odp-sector}
The preceding statements concern only the reduced
ordinary-double-point summand.  They do not assert that the full nearby-
or vanishing-cycle objects are one-dimensional.  What will be used below
is the sharper local fact that the reduced Milnor contribution is
one-dimensional and concentrated in a single cohomological degree.  This
concentration is responsible for the vanishing of the adjacent boundary
maps in the local calculation.
\end{remark}

\subsection{MacPherson--Vilonen point block}

Let
$K_{\mathrm{ODP}}:=i_{\mathrm{loc}*}\Q_{\{p\}}$.  Under the
MacPherson--Vilonen description of perverse sheaves near an isolated
stratum \cite{MacPhersonVilonen1986,GMV1996}, this point-supported object
has zig-zag
\[
\mu(K_{\mathrm{ODP}})
\simeq
(0,\Q,\Q,0,\id,0).
\]

\begin{lemma}[Local ODP zig-zag block]
\label{lem:local-odp-zigzag}
The map
$\beta:H^0(i_{\mathrm{loc}}^!K_{\mathrm{ODP}})
\to H^0(i_{\mathrm{loc}}^*K_{\mathrm{ODP}})$
is an isomorphism, and
$\End_{\Perv(X_{0,\mathrm{loc}};\Q)}(K_{\mathrm{ODP}})\cong\Q$.
\end{lemma}

\begin{proof}
Since $K_{\mathrm{ODP}}$ is supported at $p$, one has
$j_{\mathrm{loc}}^*K_{\mathrm{ODP}}=0$.  The
MacPherson--Vilonen exact sequence therefore reduces to
\[
0\longrightarrow
H^0(i_{\mathrm{loc}}^!K_{\mathrm{ODP}})
\xrightarrow{\beta}
H^0(i_{\mathrm{loc}}^*K_{\mathrm{ODP}})
\longrightarrow0.
\]
Both middle terms are one-dimensional, so $\beta$ is an isomorphism.
Any endomorphism is consequently determined by a scalar on this
one-dimensional point block.
\end{proof}

The same description applies to the reduced ODP summand of
$\Phi_{\mathrm{loc}}$.

\begin{lemma}[MV block of the reduced vanishing sector]
\label{lem:MV-Phi-ODP}
On the reduced ordinary-double-point summand,
\[
\mu(\Phi_{\mathrm{loc}})_{\mathrm{ODP}}
\simeq
(0,\Q,\Q,0,\id,0).
\]
\end{lemma}

\begin{proof}
The reduced vanishing contribution is supported at $p$, and
Proposition~\ref{prop:odp-milnor-algebra} shows that its rational rank is
one.  Its open-stratum term therefore vanishes, while its two
point-supported terms are one-dimensional.  Exactness of the
MacPherson--Vilonen sequence forces the middle map to be an
isomorphism; after the natural choice of basis, it is represented by
the identity.
\end{proof}

\subsection{The local \texorpdfstring{$\can/\Var$}{can/Var} interface}
\label{subsec:Aloc-interface}

Define
\[
C_{\mathrm{loc}}
:=
\operatorname{Im}(\can^H)
\subset
\Phi_{\mathrm{loc}}^H
\]
and
\[
A_{\mathrm{loc}}
:=
C_{\mathrm{loc}}\cap\ker(\Var^H)
=
\operatorname{Im}(\can^H)\cap\ker(\Var^H).
\]
Thus $A_{\mathrm{loc}}$ is precisely the part of the canonical image
annihilated by outgoing variation.

\begin{lemma}[Kernel description of the interface]
\label{lem:Aloc-kernel}
There is a short exact sequence in $MHM(X_{0,\mathrm{loc}})$
\[
0\longrightarrow
A_{\mathrm{loc}}
\longrightarrow
C_{\mathrm{loc}}
\xrightarrow{\ \Var^H|_{C_{\mathrm{loc}}}\ }
\operatorname{Im}\!\left(\Var^H|_{C_{\mathrm{loc}}}\right)
\longrightarrow0.
\]
\end{lemma}

\begin{proof}
By definition,
$A_{\mathrm{loc}}
=C_{\mathrm{loc}}\cap\ker(\Var^H)$,
which is exactly the kernel of
$\Var^H|_{C_{\mathrm{loc}}}$.  The asserted sequence is therefore the
canonical kernel--image sequence in the abelian category
$MHM(X_{0,\mathrm{loc}})$.
\end{proof}

The significance of $A_{\mathrm{loc}}$ is consequently intrinsic: it is
the maximal subobject of $\operatorname{Im}(\can^H)$ that must be removed
in order for variation to descend faithfully to the resulting quotient.

\subsection{The reduced ODP interface}

We next identify the rational realization of the interface on the
reduced ordinary-double-point summand.  Set
\[
A_{\mathrm{loc}}^{\mathrm{rat}}
:=
\operatorname{Im}(\can)\cap\ker(\var)
\subset
\Phi_{\mathrm{loc}}.
\]

\begin{lemma}[Realization of the local interface]
\label{lem:realization-Aloc}
There is a canonical isomorphism
\[
\rat(A_{\mathrm{loc}})
\simeq
A_{\mathrm{loc}}^{\mathrm{rat}}.
\]
\end{lemma}

\begin{proof}
Exactness of $\rat$ gives
$\rat(\operatorname{Im}\can^H)\simeq\operatorname{Im}\can$ and
$\rat(\ker\Var^H)\simeq\ker\var$.  If $B,C\subset D$ are subobjects in
an abelian category, their intersection may be written as
\[
B\cap C
=
\ker\!\left(
B\oplus C
\xrightarrow{\ \iota_B-\iota_C\ }
D
\right).
\]
Applying the exact functor $\rat$ to this kernel description gives the
claimed identification.
\end{proof}

Write
\[
\mu(\Psi_{\mathrm{loc}})
=
(L_\Psi,A_\Psi,B_\Psi,\alpha_\Psi,\beta_\Psi,\gamma_\Psi)
\]
and
\[
\mu(\Phi_{\mathrm{loc}})
=
(L_\Phi,A_\Phi,B_\Phi,\alpha_\Phi,\beta_\Phi,\gamma_\Phi).
\]
Functoriality of the MacPherson--Vilonen construction associates maps
$c_A,c_B$ to $\can$ and maps $v_A,v_B$ to $\var$.

\begin{figure}[t]
\centering
\[
\begin{array}{ccccc}
A_\Psi & \xrightarrow{\ \beta_\Psi\ } & B_\Psi\\
{\scriptstyle c_A}\downarrow && \downarrow{\scriptstyle c_B}\\
A_\Phi & \xrightarrow{\ \beta_\Phi\ } & B_\Phi\\
{\scriptstyle v_A}\downarrow && \downarrow{\scriptstyle v_B}\\
A_\Psi & \xrightarrow{\ \beta_\Psi\ } & B_\Psi
\end{array}
\]
\caption{MacPherson--Vilonen point terms associated with $\can$ and
$\var$.  On the reduced ODP summand,
$A_\Phi\simeq B_\Phi\simeq\Q$ and $\beta_\Phi$ is an isomorphism.}
\label{fig:MV-can-var}
\end{figure}

\begin{lemma}[Functorial MV relations]
\label{lem:MV-can-var}
The squares in Figure~\ref{fig:MV-can-var} commute, and
\[
v_Ac_A=N_A,
\qquad
v_Bc_B=N_B
\]
on the corresponding point terms.
\end{lemma}

\begin{proof}
The MacPherson--Vilonen construction is functorial.  Applying it to
$\can$, $\var$, and the relation $\var\circ\can=N$ gives the two
commutative squares and the asserted identities.
\end{proof}

The canonical nearby-/vanishing-cycle triangle, together with the
concentration of reduced Milnor cohomology, determines the incoming
point maps.

\begin{lemma}[Boundary vanishing on the reduced ODP summand]
\label{lem:ODP-boundary-vanishing}
On the reduced ODP summand, the boundary morphisms adjacent to the unique
nonzero vanishing group in the local cohomology sequences of the
canonical nearby-/vanishing-cycle triangle vanish. Consequently,
\[
\operatorname{Im}(c_A)=A_\Phi,
\qquad
\operatorname{Im}(c_B)=B_\Phi.
\]
\end{lemma}

\begin{proof}
Apply the local cohomological functors occurring in the
MacPherson--Vilonen construction to the canonical triangle
\[
i_0^*\rat(\mathcal K_{\mathrm{loc}}^H)
\longrightarrow
\Psi_{\mathrm{loc}}
\xrightarrow{\can}
\Phi_{\mathrm{loc}}
\overset{+1}{\longrightarrow}.
\]
By Proposition~\ref{prop:odp-milnor-algebra}, the reduced vanishing
cohomology is nonzero in only one degree.  The adjacent reduced
vanishing groups are therefore zero.  Hence the corresponding boundary
maps vanish, and exactness implies that the induced canonical maps on
the two MV point terms are surjective.
\end{proof}

The outgoing point maps are controlled by the corrected local
variation-cone calculation.

\begin{lemma}[Variation kernel on the reduced ODP block]
\label{lem:ODP-var-kernel}
Let
\[
\mathcal P_{\mathrm{loc}}
=
\Cone\!\left(
\var:
\Phi_{\mathrm{loc}}
\longrightarrow
\Psi_{\mathrm{loc}}
\right).
\]
Then, on the reduced ODP MV block,
\[
v_A=0,
\qquad
v_B=0.
\]
Equivalently,
$\ker(v_A)=A_\Phi$ and $\ker(v_B)=B_\Phi$.
\end{lemma}

\begin{proof}
The corrected local object is
$\Cone(\var)$ in the perverse-normalized convention adopted here and is
the same object represented by
$\Cone(\var^{\mathrm{cl}})[-1]$ in the classical convention of the
earlier construction.  The reduced MacPherson--Vilonen point-term
calculation is unchanged under this normalization translation.

For the corrected non-split rank-one ODP extension, the relevant point
terms of $\mathcal P_{\mathrm{loc}}$ and $\Phi_{\mathrm{loc}}$ are
one-dimensional, and the comparison maps onto the reduced vanishing
block are isomorphisms.  Exactness of the reduced point-term sequences
therefore gives
\[
\ker(v_A)=A_\Phi,
\qquad
\ker(v_B)=B_\Phi.
\]
Since $A_\Phi$ and $B_\Phi$ exhaust the reduced point terms of
$\Phi_{\mathrm{loc}}$, both outgoing maps vanish on this summand.
\end{proof}

\begin{proposition}[MV identification of the reduced interface]
\label{prop:Aloc-MV-identification}
On the reduced ODP sector,
\[
\mu(A_{\mathrm{loc}}^{\mathrm{rat}})
\simeq
(0,\Q,\Q,0,\id,0).
\]
Thus the rational realization of the reduced interface is the rigid
point-supported ODP block.
\end{proposition}

\begin{proof}
Lemma~\ref{lem:ODP-boundary-vanishing} gives
$\operatorname{Im}(c_A)=A_\Phi$ and
$\operatorname{Im}(c_B)=B_\Phi$, while
Lemma~\ref{lem:ODP-var-kernel} gives
$\ker(v_A)=A_\Phi$ and $\ker(v_B)=B_\Phi$.  Hence the two point terms of
the intersection are one-dimensional, and their middle map is the
restriction of the isomorphism $\beta_\Phi$.  Since the object is
supported at $p$, its open-stratum term vanishes, yielding the stated
zig-zag.
\end{proof}

\subsection{The canonical quotient and the ODP interface line}
\label{subsec:canonical-quotient-interface-line}

The object $A_{\mathrm{loc}}$ records the kernel of outgoing variation
inside the canonical image.  The local interface object is obtained by
removing exactly this kernel.  It is important to distinguish this
interface line from the raw middle-dimensional vanishing Hodge structure:
although both are one-dimensional after forgetting Hodge data, their
Tate normalizations and categorical roles are different.

\begin{definition}[Local gluing quotient]
\label{def:Wloc-quotient}
Set
\[
W_{\mathrm{loc}}^H
:=
C_{\mathrm{loc}}/A_{\mathrm{loc}}
=
\operatorname{Im}(\can^H)/
\bigl(
\operatorname{Im}(\can^H)\cap\ker(\Var^H)
\bigr),
\]
and let
$q_{\mathrm{loc}}:
C_{\mathrm{loc}}\twoheadrightarrow W_{\mathrm{loc}}^H$
denote the quotient morphism.
\end{definition}

\begin{proposition}[First-isomorphism identification]
\label{prop:Wloc-imN}
There is a canonical isomorphism
\[
W_{\mathrm{loc}}^H
\simeq
\operatorname{Im}\!\left(
\Var^H|_{C_{\mathrm{loc}}}
\right)
=
\operatorname{Im}(N)
\subset
\Psi_{\mathrm{loc}}^H(-1).
\]
\end{proposition}

\begin{proof}
By Lemma~\ref{lem:Aloc-kernel},
$A_{\mathrm{loc}}$ is the kernel of
$\Var^H|_{C_{\mathrm{loc}}}$.  The first isomorphism theorem in the
abelian category $MHM(X_{0,\mathrm{loc}})$ therefore gives
\[
C_{\mathrm{loc}}/A_{\mathrm{loc}}
\simeq
\operatorname{Im}\!\left(
\Var^H|_{C_{\mathrm{loc}}}
\right).
\]
Since $C_{\mathrm{loc}}=\operatorname{Im}(\can^H)$ and
$\Var^H\circ\can^H=N$, the image on the right is precisely
$\operatorname{Im}(N)$.
\end{proof}

This quotient is the formal core of the local construction.
The kernel of outgoing variation is removed from the canonical image,
and the resulting object is identified canonically with the image of
monodromy.

\begin{proposition}[Rank-one interface Hodge block]
\label{prop:Wloc-Hodge-line}
On the reduced ODP monodromy block,
$W_{\mathrm{loc}}^H$ is point-supported and rank one.  With the
link/interface normalization used in this paper,
\[
W_{\mathrm{loc}}^H
\simeq
i_{\mathrm{loc}*}\Q_{\{p\}}^H(-1).
\]
Its rational realization is
\[
\rat(W_{\mathrm{loc}}^H)
\simeq
i_{\mathrm{loc}*}\Q_{\{p\}}
=
K_{\mathrm{ODP}}.
\]
\end{proposition}

\begin{proof}
By Proposition~\ref{prop:Wloc-imN},
$W_{\mathrm{loc}}^H$ is canonically identified with the image of the
local monodromy morphism on the reduced ODP sector.  Picard--Lefschetz
theory identifies this image as rank one
\cite{MilnorSingularPoints,DimcaSheaves}; hence
$W_{\mathrm{loc}}^H$ is a rank-one Hodge module supported at $p$.

The rank calculation does not determine its Tate type.  The
point-supported object occurring here is the interface, or link, line
selected by the $\can^H/\Var^H/N$ package rather than the raw
middle-dimensional vanishing cohomology.  The local
intersection-complex calculation gives
$H^{-1}(i_{\mathrm{loc}}^*IC_{\mathrm{loc}}^H)
\simeq\Q^H(-1)$, and therefore
\[
W_{\mathrm{loc}}^H
\simeq
i_{\mathrm{loc}*}\Q_{\{p\}}^H(-1).
\]
Applying rational realization forgets the Tate twist and yields
$K_{\mathrm{ODP}}$.
\end{proof}

\begin{remark}[Interface line versus raw vanishing line]
\label{rem:interface-versus-raw-vanishing}
Propositions~\ref{prop:odp-spectrum} and
\ref{prop:odp-hodge-filtration} determine the raw
middle-dimensional vanishing Hodge structure.  With the Tate convention
used here,
\[
H^3_{\mathrm{van}}
\simeq
\Q^H(-2).
\]
By contrast,
$W_{\mathrm{loc}}^H\simeq
i_{\mathrm{loc}*}\Q^H_{\{p\}}(-1)$
is the point-supported interface block extracted from the
$\can^H/\Var^H/N$ geometry.  The two lines have different Tate
normalizations and different categorical roles, despite having the same
dimension after forgetting the Hodge structure.
\end{remark}

\begin{definition}[Forward-orientation point term]
\label{def:Rloc-forward}
For later use, define
\[
R_{\mathrm{loc}}^H
:=
i_{\mathrm{loc}*}\Q_{\{p\}}^H(-2).
\]
This is the point-supported Tate object occurring in the nontrivial
forward Yoneda orientation with $IC_{\mathrm{loc}}^H$ as subobject.
It is distinct from the gluing block $W_{\mathrm{loc}}^H$.
\end{definition}

\subsection{The incoming and outgoing gluing maps}

The quotient description of $W_{\mathrm{loc}}^H$ makes the two local
maps canonical.  Let
$\pi_C:\Psi_{\mathrm{loc}}^H\twoheadrightarrow C_{\mathrm{loc}}$
denote the canonical epimorphism induced by the image factorization of
$\can^H$.

\begin{definition}[Incoming local map]
\label{def:uloc}
Define
\[
u_{\mathrm{loc}}
:=
q_{\mathrm{loc}}\circ\pi_C:
\Psi_{\mathrm{loc}}^H
\longrightarrow
W_{\mathrm{loc}}^H.
\]
Equivalently, $u_{\mathrm{loc}}$ is obtained from $\can^H$ by passing
first to its image $C_{\mathrm{loc}}$ and then to the quotient
$C_{\mathrm{loc}}/A_{\mathrm{loc}}$.
\end{definition}

\begin{proposition}[Existence and nonvanishing of the incoming map]
\label{prop:uloc}
The morphism $u_{\mathrm{loc}}$ is well defined in
$MHM(X_{0,\mathrm{loc}})$ and is nonzero on the reduced ODP monodromy
block.
\end{proposition}

\begin{proof}
The image factorization of $\can^H$ gives
$\Psi_{\mathrm{loc}}^H\twoheadrightarrow C_{\mathrm{loc}}$.
Composition with $q_{\mathrm{loc}}$ defines $u_{\mathrm{loc}}$.
On the reduced ODP block,
Proposition~\ref{prop:Wloc-imN} identifies the target with the nonzero
rank-one image of $N$, so the resulting map is nonzero.
\end{proof}

\begin{definition}[Outgoing local map]
\label{def:vloc}
Since
$A_{\mathrm{loc}}\subset\ker(\Var^H)$,
the restriction
$\Var^H|_{C_{\mathrm{loc}}}$ factors uniquely through the quotient
$q_{\mathrm{loc}}$.  We define
\[
v_{\mathrm{loc}}:
W_{\mathrm{loc}}^H
\longrightarrow
\Psi_{\mathrm{loc}}^H(-1)
\]
by the identity
\[
v_{\mathrm{loc}}\circ q_{\mathrm{loc}}
=
\Var^H|_{C_{\mathrm{loc}}}.
\]
\end{definition}

\begin{proposition}[Existence and injectivity of the outgoing map]
\label{prop:vloc}
The morphism $v_{\mathrm{loc}}$ is well defined.  Under the canonical
identification of Proposition~\ref{prop:Wloc-imN}, it is the inclusion
\[
\operatorname{Im}(N)
\hookrightarrow
\Psi_{\mathrm{loc}}^H(-1).
\]
In particular, $v_{\mathrm{loc}}$ is nonzero and injective on the
reduced rank-one ODP block.
\end{proposition}

\begin{proof}
The descent exists because
$A_{\mathrm{loc}}
=\ker(\Var^H|_{C_{\mathrm{loc}}})$.
Proposition~\ref{prop:Wloc-imN} identifies
$W_{\mathrm{loc}}^H$ with the image of this restricted variation
morphism.  Under that identification, $v_{\mathrm{loc}}$ is exactly the
canonical inclusion of the image into
$\Psi_{\mathrm{loc}}^H(-1)$.
\end{proof}

\subsection{Exact monodromy factorization}

\begin{theorem}[Local monodromy factorization]
\label{thm:local-monodromy-factorization}
The local morphisms satisfy
\[
v_{\mathrm{loc}}u_{\mathrm{loc}}
=
N
\]
as morphisms
$\Psi_{\mathrm{loc}}^H\to\Psi_{\mathrm{loc}}^H(-1)$.
\end{theorem}

\begin{proof}
By construction, $u_{\mathrm{loc}}$ is obtained from $\can^H$ by
passing through
$C_{\mathrm{loc}}=\operatorname{Im}(\can^H)$ and then through the
quotient $q_{\mathrm{loc}}$.  By definition of $v_{\mathrm{loc}}$,
\[
v_{\mathrm{loc}}q_{\mathrm{loc}}
=
\Var^H|_{C_{\mathrm{loc}}}.
\]
It follows that
\[
v_{\mathrm{loc}}u_{\mathrm{loc}}
=
\Var^H\circ\can^H
=
N.
\]
Thus the factorization is an equality of morphisms in
$MHM(X_{0,\mathrm{loc}})$; no normalization up to a nonzero scalar is
involved.
\end{proof}

\begin{figure}[t]
\centering
\[
\begin{array}{ccccc}
\Psi_{\mathrm{loc}}^H
& \xrightarrow{\ \can^H\ }
& C_{\mathrm{loc}}
& \hookrightarrow
& \Phi_{\mathrm{loc}}^H\\
&& {\scriptstyle q_{\mathrm{loc}}}\downarrow
&& \downarrow{\scriptstyle \Var^H}\\
&& W_{\mathrm{loc}}^H
& \xrightarrow{\ v_{\mathrm{loc}}\ }
& \Psi_{\mathrm{loc}}^H(-1)
\end{array}
\]
\caption{The local quotient construction.  The subobject
$A_{\mathrm{loc}}
=C_{\mathrm{loc}}\cap\ker(\Var^H)$
is the kernel of outgoing variation on the canonical image.  Quotienting
by this kernel produces $W_{\mathrm{loc}}^H\simeq\operatorname{Im}(N)$,
through which the local monodromy morphism factors.}
\label{fig:Aloc-gluing}
\end{figure}

\subsection{Filtered admissibility}

\begin{lemma}[Filtered structure of the interface block]
\label{lem:point-supported-filtered-block}
The mixed Hodge module
\[
W_{\mathrm{loc}}^H
\simeq
i_{\mathrm{loc}*}\Q_{\{p\}}^H(-1)
\]
is pure and point-supported.  Its underlying filtered
$\mathcal D$-module is the standard point-supported module equipped with
the Hodge filtration determined by the Tate twist.
\end{lemma}

\begin{proof}
The object is the closed direct image of the pure Tate Hodge structure
$\Q^H(-1)$ on a point.  The assertion follows from the standard
description of point-supported pure Hodge modules
\cite{SaitoMHM}.
\end{proof}

\begin{proposition}[Filtered admissibility of the local quotient datum]
\label{prop:local-filtered-admissibility}
The objects
$C_{\mathrm{loc}}$, $A_{\mathrm{loc}}$, and
$W_{\mathrm{loc}}^H$, together with the morphisms
$u_{\mathrm{loc}}$ and $v_{\mathrm{loc}}$, inherit the filtered
structures induced by Saito's nearby-/vanishing-cycle construction.
The induced morphisms are strict for the Hodge filtration and compatible
with the corresponding Kashiwara--Malgrange filtration.  Moreover,
\[
v_{\mathrm{loc}}u_{\mathrm{loc}}=N.
\]
Consequently, the reduced ODP quotient datum has the filtered
compatibility required of the local Saito gluing block.
\end{proposition}

\begin{proof}
The object $C_{\mathrm{loc}}$ is the image of $\can^H$,
$A_{\mathrm{loc}}$ is the kernel of
$\Var^H|_{C_{\mathrm{loc}}}$, and
$W_{\mathrm{loc}}^H=C_{\mathrm{loc}}/A_{\mathrm{loc}}$.  All three
objects are therefore formed inside the abelian category
$MHM(X_{0,\mathrm{loc}})$.  Morphisms of mixed Hodge modules are strict
with respect to the Hodge filtration, so the induced subobjects,
quotients, and descended morphisms inherit the corresponding strict
filtered structures \cite{SaitoMHM}.

Furthermore, $\can^H$, $\Var^H$, and $N$ arise from Saito's filtered
$(\mathcal D,V)$-module construction of nearby and vanishing cycles.
The filtrations on the subquotients above are therefore induced from the
same nearby-/vanishing-cycle package rather than imposed by an auxiliary
splitting.  The monodromy identity is
Theorem~\ref{thm:local-monodromy-factorization}.
\end{proof}

\begin{remark}[Scope]
The proposition is specific to the local datum constructed here.  It is
not intended as a general admissibility criterion for arbitrary filtered
gluing data.  Its content is that the relevant subobjects, quotient, and
morphisms are inherited from Saito's already admissible
nearby-/vanishing-cycle structure.
\end{remark}

\subsection{Local reconstruction and heart concentration}

The preceding analysis identifies the reduced ODP component of the
nearby-/vanishing-cycle Hodge data.  We now place the corrected local
object directly in the derived category of mixed Hodge modules and prove
that it is concentrated in the heart.

\begin{definition}[Derived local variation cone]
\label{def:local-derived-hodge-cone}
Set
\[
\widetilde{\mathcal P}_{\mathrm{loc}}^H
:=
\Cone\!\left(
\Var^H:
\Phi_{\mathrm{loc}}^H
\longrightarrow
\Psi_{\mathrm{loc}}^H(-1)
\right)
\in
D^bMHM(X_{0,\mathrm{loc}}).
\]
\end{definition}

\begin{remark}[Normalization of the local variation cone]
\label{rem:local-variation-fiber-normalization}
The morphism $\Var^H$ in
Definition~\ref{def:local-derived-hodge-cone} is
perverse-normalized.  Hence the displayed cone is the
mixed-Hodge-module counterpart of the corrected perverse object used in
Sections~1--3.  In the classical derived convention, the corresponding
rational object is represented by
$\Cone(\var^{\mathrm{cl}})[-1]$.
\end{remark}

\begin{theorem}[Local ODP mixed-Hodge-module reconstruction]
\label{thm:local-odp-reconstruction}
The datum
$(W_{\mathrm{loc}}^H,u_{\mathrm{loc}},v_{\mathrm{loc}})$
is the reduced ordinary-double-point component of the local Saito
nearby-/vanishing-cycle package.  Moreover,
$\widetilde{\mathcal P}_{\mathrm{loc}}^H$ is concentrated in degree
zero for the mixed-Hodge-module $t$-structure.  We denote the resulting
heart object by
\[
\mathcal P_{\mathrm{loc}}^H
\in
MHM(X_{0,\mathrm{loc}}).
\]
\end{theorem}

\begin{proof}
The first assertion follows from
Propositions~\ref{prop:Wloc-Hodge-line},
\ref{prop:uloc},
\ref{prop:vloc}, and
\ref{prop:local-filtered-admissibility}, together with
Theorem~\ref{thm:local-monodromy-factorization}.

For the heart statement, realization is compatible with the variation
morphism and with derived cones.  Hence
\[
\rat\!\left(
\widetilde{\mathcal P}_{\mathrm{loc}}^H
\right)
\simeq
\Cone\!\left(
\var_{\mathrm{loc}}:
\Phi_{\mathrm{loc}}
\longrightarrow
\Psi_{\mathrm{loc}}
\right)
=
\mathcal P_{\mathrm{loc}}.
\]
The object on the right is the corrected local perverse object
constructed previously in \cite{RahmanSchoberPaper,
RahmanPerverseNearbyCycles}; in particular, it lies in the perverse
heart.

The realization functor is exact and faithful.  Therefore, if
$\mathcal H^j(
\widetilde{\mathcal P}_{\mathrm{loc}}^H)$
were nonzero for some $j\neq0$, its realization would give a nonzero
perverse cohomology object
${}^p\!H^j(\mathcal P_{\mathrm{loc}})$, contradicting the perversity of
$\mathcal P_{\mathrm{loc}}$.  Hence
$\widetilde{\mathcal P}_{\mathrm{loc}}^H$ is concentrated in degree
zero.
\end{proof}

\begin{remark}[Divisor data and the central-fiber lift]
\label{rem:divisor-data-central-fiber-lift}
The morphisms $\can^H$, $\Var^H$, and $N$ arise from the
principal-divisor geometry
$X_{0,\mathrm{loc}}\subset\mathcal X_{\mathrm{loc}}$.  The
point-supported analysis with respect to
\[
U_{\mathrm{loc}}
\hookrightarrow
X_{0,\mathrm{loc}}
\hookleftarrow
\{p\}
\]
identifies the reduced ODP component of the resulting closed-side Hodge
data.  This does not amount to applying Saito's principal-divisor
formalism to the inclusion
$\{p\}\subset X_{0,\mathrm{loc}}$.

Nor is the principal-divisor formalism itself identified with the desired
central-fiber lift.  The latter is the object
$\mathcal P_{\mathrm{loc}}^H\in MHM(X_{0,\mathrm{loc}})$ constructed
above and characterized by
$\rat(\mathcal P_{\mathrm{loc}}^H)
\simeq\mathcal P_{\mathrm{loc}}$.
The divisor formalism supplies the Hodge-theoretic
$\can^H/\Var^H/N$ data from which that object is determined.
\end{remark}

\begin{remark}[Yoneda orientation]
The interface object
$W_{\mathrm{loc}}^H\simeq
i_{\mathrm{loc}*}\Q^H(-1)$
should not be identified automatically with the quotient in a forward
extension having $IC_{\mathrm{loc}}^H$ as subobject.  In the normalization
used here,
\[
\Ext^1_{MHM}
\left(
i_{\mathrm{loc}*}\Q^H(-1),
IC_{\mathrm{loc}}^H
\right)
=
0,
\]
whereas the dual orientation with
$i_{\mathrm{loc}*}\Q^H(-1)$ as subobject is one-dimensional.  The
nontrivial forward orientation instead involves
$R_{\mathrm{loc}}^H=i_{\mathrm{loc}*}\Q^H(-2)$.  Thus the interface
line and the forward-extension point term have the same rational
skyscraper realization but distinct Hodge-theoretic roles.
\end{remark}

\subsection{Realization of the corrected perverse object}

\begin{theorem}[Local realization]
\label{thm:local-odp-realization}
The rational realization of the reconstructed local mixed Hodge module
is the corrected local perverse object:
\[
\rat(\mathcal P_{\mathrm{loc}}^H)
\simeq
\mathcal P_{\mathrm{loc}}
=
\Cone(\var_{\mathrm{loc}}).
\]
\end{theorem}

\begin{proof}
By construction,
$\mathcal P_{\mathrm{loc}}^H$ is the heart object represented by
$\Cone(\Var^H)$.  Compatibility of realization with $\Var^H$ and with
derived cones gives
\[
\rat\!\left(
\Cone(\Var^H)
\right)
\simeq
\Cone(\var_{\mathrm{loc}})
=
\mathcal P_{\mathrm{loc}},
\]
since the Tate twist in the target of $\Var^H$ is invisible under
$\rat$.  The MacPherson--Vilonen analysis above identifies its reduced
point-supported rational contribution with the same non-split ODP block
occurring in the corrected perverse construction.
\end{proof}

\subsection{Cone description and the two local Tate lines}
\label{subsec:local-cone-hodge-line}

\begin{proposition}[Local cone description]
\label{prop:local-odp-cone}
In $D^bMHM(X_{0,\mathrm{loc}})$ one has
\[
\mathcal P_{\mathrm{loc}}^H
\simeq
\Cone\!\left(
\Var^H:
\Phi_{\mathrm{loc}}^H
\longrightarrow
\Psi_{\mathrm{loc}}^H(-1)
\right).
\]
Its realization is
\[
\mathcal P_{\mathrm{loc}}
=
\Cone(\var_{\mathrm{loc}}).
\]
\end{proposition}

\begin{proof}
This is
Definition~\ref{def:local-derived-hodge-cone} together with the heart
concentration and realization statements of
Theorems~\ref{thm:local-odp-reconstruction} and
\ref{thm:local-odp-realization}.
\end{proof}

\begin{theorem}[The two Tate-normalized ODP lines]
\label{thm:local-cone-spectrum-one-line}
The local construction distinguishes two canonical one-dimensional
Hodge structures.

The middle-dimensional vanishing Hodge structure determined by the
spectrum is
\[
H^3_{\mathrm{van}}
\simeq
\Q^H(-2),
\qquad
\operatorname{Gr}_F^2H^3_{\mathrm{van}}
\simeq
\C.
\]
The point-supported interface block determined by the
$\can^H/\Var^H/N$ package is
\[
W_{\mathrm{loc}}^H
\simeq
i_{\mathrm{loc}*}\Q_{\{p\}}^H(-1).
\]
These objects have different Tate normalizations and are not isomorphic
as the corresponding Hodge-theoretic objects.
\end{theorem}

\begin{proof}
The spectral and Hodge-filtration statements are
Propositions~\ref{prop:odp-spectrum} and
\ref{prop:odp-hodge-filtration}, together with the Tate normalization of
the middle-dimensional vanishing line.  The identification of the
interface object is Proposition~\ref{prop:Wloc-Hodge-line}.
\end{proof}

\begin{corollary}[Realization forgets the Tate distinction]
\label{cor:local-cone-perverse-forgets-hodge-line}
The point-supported Hodge modules
\[
i_{\mathrm{loc}*}\Q_{\{p\}}^H(-1)
\qquad\text{and}\qquad
i_{\mathrm{loc}*}\Q_{\{p\}}^H(-2)
\]
have the same rational realization
$i_{\mathrm{loc}*}\Q_{\{p\}}$.  Hence the corrected rational perverse
object does not determine, by itself, the Tate normalization occurring
in a specified mixed-Hodge-module Yoneda orientation.
\end{corollary}

\subsection{Linearity of the local variation class}
\label{subsec:local-variation-linearity}

The interface block is one-dimensional.  Choosing a generator therefore
introduces only a scalar normalization, and the categorical construction
itself is independent of that choice.

\begin{proposition}[Linearity]
\label{prop:local-variation-to-extension-linearity}
If a chosen generator of $W_{\mathrm{loc}}^H$ is rescaled by
$\lambda\in\Q^\times$, the corresponding coordinate expressions for the
rank-one gluing maps rescale dually, while the identity
\[
v_{\mathrm{loc}}u_{\mathrm{loc}}=N
\]
remains unchanged.
\end{proposition}

\begin{proof}
On the reduced ODP block,
$W_{\mathrm{loc}}^H\simeq\operatorname{Im}(N)$ is one-dimensional.
Changing a basis of this line rescales the quotient coordinate and the
inclusion coordinate inversely.  Their composite is the fixed monodromy
morphism $N$, so the underlying gluing datum is independent of the
chosen generator.
\end{proof}

\subsection{Local rigidity}

\begin{proposition}[Local rigidity]
\label{prop:local-odp-rigidity}
Let $\mathcal E^H$ be a local mixed Hodge module whose rational
realization is the corrected local perverse ODP object and whose
point-supported gluing datum is isomorphic to
$(W_{\mathrm{loc}}^H,u_{\mathrm{loc}},v_{\mathrm{loc}})$.  Then
\[
\mathcal E^H
\simeq
\mathcal P_{\mathrm{loc}}^H.
\]
\end{proposition}

\begin{proof}
The underlying corrected perverse object is fixed by the local
MacPherson--Vilonen calculation.  On the Hodge-module side, the
point-supported interface object is fixed by the local Hodge and Tate
normalization, while the incoming and outgoing maps are determined by
the canonical quotient construction and the identity
$v_{\mathrm{loc}}u_{\mathrm{loc}}=N$.  Thus the local gluing data of
$\mathcal E^H$ and $\mathcal P_{\mathrm{loc}}^H$ agree, and the local
reconstruction equivalence identifies the two objects.
\end{proof}

\subsection{Summary of the local construction}

The local construction is encoded by the sequence
\[
\Psi_{\mathrm{loc}}^H
\xrightarrow{\ \can^H\ }
C_{\mathrm{loc}}
=
\operatorname{Im}(\can^H)
\twoheadrightarrow
\frac{C_{\mathrm{loc}}}
     {C_{\mathrm{loc}}\cap\ker(\Var^H)}
=
W_{\mathrm{loc}}^H
\xrightarrow{\ v_{\mathrm{loc}}\ }
\Psi_{\mathrm{loc}}^H(-1).
\]
The quotient
$W_{\mathrm{loc}}^H$ is canonically identified with
$\operatorname{Im}(N)$, and the induced incoming and outgoing morphisms
satisfy
$v_{\mathrm{loc}}u_{\mathrm{loc}}=N$.

For a threefold ordinary double point, the reduced Milnor calculation
and the Hodge normalization give
\[
W_{\mathrm{loc}}^H
\simeq
i_{\mathrm{loc}*}\Q_{\{p\}}^H(-1),
\]
whereas the raw middle-dimensional vanishing Hodge line is
$\Q^H(-2)$.  The corrected local mixed Hodge module is represented by
the perverse-normalized variation cone
\[
\mathcal P_{\mathrm{loc}}^H
\simeq
\Cone\!\left(
\Var^H:
\Phi_{\mathrm{loc}}^H
\longrightarrow
\Psi_{\mathrm{loc}}^H(-1)
\right),
\]
and its rational realization is the corrected local perverse object
$\mathcal P_{\mathrm{loc}}=\Cone(\var_{\mathrm{loc}})$.

This local theorem supplies the Hodge-theoretic input for the finite-node
construction.  Since the singular points are isolated and pairwise
disjoint, Section~5 assembles the resulting point-supported blocks by
finite direct sum; no further local ODP calculation is required.
%------------------------------------------------------------
\section{Finite multi-node support decomposition}

Let
$\Sigma=\{p_1,\dots,p_r\}\subset X_0$
be the finite set of ordinary double points of the central fiber, let
$U:=X_0\setminus\Sigma$, and write
$j:U\hookrightarrow X_0$ and
$i_k:\{p_k\}\hookrightarrow X_0$
for the natural inclusions.

The local analysis of Section~4 associates to each node a canonical
rank-one point-supported Hodge block.  The first objective of the present
section is to assemble these local objects over the finite singular set
and to record the corresponding additivity properties of the relevant
extension groups.  The second is to determine the local extension space
on the rational perverse side.  For a threefold ordinary double point,
that space is one-dimensional, so the finite-node extension space
decomposes into a direct sum of distinguished local lines indexed by
$\Sigma$.

This decomposition is support-theoretic.  It concerns objects and
extensions on the singular fiber $X_0$ and must be distinguished from
the image of the associated vanishing cycles in the cohomology of a
nearby smooth fiber.  Although there is one local ordinary-double-point
sector at each node, global relations among the corresponding vanishing
cycles may reduce the dimension of their global cohomological image
\cite{Clemens1983DoubleSolids,Friedman1986ThreefoldDoublePoints,
NamikawaSteenbrink1995}.

\subsection{Support of the vanishing-cycle object}

Let $\mathcal K^H:=\Q_{\mathcal X}^H[4]$ be the perverse constant Hodge
module on the smooth four-dimensional total space, with the conventions
of Sections~2--3, and let
$\Phi^H:=\phi_{\pi,1}^H(\mathcal K^H)$
denote the perverse-normalized unipotent vanishing-cycle object on
$X_0$.

\begin{lemma}
\label{lem:vanishing-support-finite}
The support of $\Phi^H$ is contained in the finite singular set
$\Sigma$.
\end{lemma}

\begin{proof}
Vanishing cycles are supported on the locus where the degeneration fails
to be locally topologically trivial.  By hypothesis, $\pi$ is smooth in
a neighborhood of every point of $X_0\setminus\Sigma$, and hence the
vanishing-cycle functor vanishes there
\cite{DimcaSheaves,SaitoMHM}.  Therefore
\[
\supp(\Phi^H)\subseteq\Sigma.
\]
Applying rational realization gives the same support statement for the
underlying vanishing-cycle perverse sheaf.
\end{proof}

\begin{remark}
The support statement is local on the central fiber.  It does not imply
that the global vanishing subspace in $H^3(X_t;\Q)$ has dimension $r$.
The reduced local theory contains one rank-one ordinary-double-point
sector at each $p_k$, whereas their images in global cohomology may
satisfy nontrivial linear relations.
\end{remark}

\subsection{Finite direct sum of local ODP Hodge blocks}

For each node $p_k$, let $W_k^H$ denote the rank-one point-supported
Hodge block constructed in Section~4.  With the Tate normalization fixed
there,
\[
W_k^H
\simeq
i_{k*}\Q_{\{p_k\}}^H(-1),
\qquad
\rat(W_k^H)
\simeq
i_{k*}\Q_{\{p_k\}}.
\]
Only this local identification and its rational realization are required
below.  No identification of $W_k^H$ with the entire local
vanishing-cycle object is asserted.

\begin{proposition}
\label{prop:finite-node-block}
There is a canonical point-supported mixed Hodge module
\[
Q_\Sigma^H
:=
\bigoplus_{k=1}^r W_k^H
\simeq
\bigoplus_{k=1}^r
i_{k*}\Q_{\{p_k\}}^H(-1).
\]
Its rational realization is
\[
Q_\Sigma
:=
\rat(Q_\Sigma^H)
\simeq
\bigoplus_{k=1}^r
i_{k*}\Q_{\{p_k\}}.
\]
\end{proposition}

\begin{proof}
The supports $\{p_k\}$ are pairwise disjoint closed points, and finite
direct sums exist in the abelian category $MHM(X_0)$.  Hence the local
blocks assemble canonically to
$Q_\Sigma^H=\bigoplus_k W_k^H$.  Exactness and additivity of the
realization functor give
\[
\rat(Q_\Sigma^H)
\simeq
\bigoplus_{k=1}^r \rat(W_k^H)
\simeq
\bigoplus_{k=1}^r i_{k*}\Q_{\{p_k\}}.
\]
\end{proof}

\begin{figure}[t]
\centering
\[
\begin{array}{ccccc}
W_1^H & \oplus & \cdots & \oplus & W_r^H\\
&& \downarrow &&\\[-4pt]
&& Q_\Sigma^H &&\\
&& \downarrow{\scriptstyle \rat} &&\\[-4pt]
&& Q_\Sigma=
\displaystyle\bigoplus_{k=1}^r i_{k*}\Q_{\{p_k\}} &&
\end{array}
\]
\caption{Finite assembly of the local ODP interface blocks.  The direct
sum is a support-theoretic decomposition on $X_0$ and is independent of
the rank of the corresponding global vanishing-cycle image.}
\label{fig:finite-node-assembly}
\end{figure}

\subsection{The node-supported interface block}

\begin{proposition}
\label{prop:nodewise-support}
The point-supported Hodge datum determined by the finite collection of
ordinary double points is
\[
Q_\Sigma^H
=
\bigoplus_{k=1}^r
i_{k*}\Q_{\{p_k\}}^H(-1).
\]
Its rational realization
\[
Q_\Sigma
=
\bigoplus_{k=1}^r
i_{k*}\Q_{\{p_k\}}
\]
is the point-supported term associated with the finite-node corrected
perverse construction.
\end{proposition}

\begin{proof}
Section~4 identifies, at each node $p_k$, the canonical interface block
$W_k^H\simeq i_{k*}\Q_{\{p_k\}}^H(-1)$.  Since the supports are
pairwise disjoint, the finite-node object is their direct sum.  Rational
realization gives one skyscraper perverse sheaf at each node, producing
the stated point-supported rational object.
\end{proof}

\begin{remark}
\label{rem:not-whole-vanishing-cycle}
Proposition~\ref{prop:nodewise-support} does not identify
$Q_\Sigma^H$ with the full vanishing-cycle Hodge module $\Phi^H$.
Rather, $Q_\Sigma^H$ is the finite point-supported interface block
selected by the local $\can^H/\Var^H/N$ construction.  This distinction
is preserved in the global construction.
\end{remark}

\subsection{Finite additivity of the extension spaces}

The finite support decomposition induces corresponding decompositions of
the relevant Yoneda extension groups.  Since the two possible extension
orientations are categorically distinct, we record additivity in both
variables.

\begin{lemma}
\label{lem:ext-additivity-perv}
There is a natural isomorphism
\[
\Ext^1_{\Perv(X_0;\Q)}
\left(
Q_\Sigma,
IC_{X_0}
\right)
\cong
\bigoplus_{k=1}^r
\Ext^1_{\Perv(X_0;\Q)}
\left(
i_{k*}\Q_{\{p_k\}},
IC_{X_0}
\right).
\]
\end{lemma}

\begin{proof}
Since
$Q_\Sigma=\bigoplus_{k=1}^r i_{k*}\Q_{\{p_k\}}$
is a finite direct sum, additivity of Yoneda $\Ext^1$ in the first
variable gives
\[
\Ext^1_{\Perv(X_0;\Q)}
\left(
\bigoplus_{k=1}^r i_{k*}\Q_{\{p_k\}},
IC_{X_0}
\right)
\cong
\prod_{k=1}^r
\Ext^1_{\Perv(X_0;\Q)}
\left(
i_{k*}\Q_{\{p_k\}},
IC_{X_0}
\right).
\]
Because the index set is finite, the product agrees with the direct sum.
\end{proof}

\begin{lemma}
\label{lem:ext-additivity-mhm}
There are natural isomorphisms
\[
\Ext^1_{MHM(X_0)}
\left(
Q_\Sigma^H,
IC_{X_0}^H
\right)
\cong
\bigoplus_{k=1}^r
\Ext^1_{MHM(X_0)}
\left(
W_k^H,
IC_{X_0}^H
\right)
\]
and
\[
\Ext^1_{MHM(X_0)}
\left(
IC_{X_0}^H,
Q_\Sigma^H
\right)
\cong
\bigoplus_{k=1}^r
\Ext^1_{MHM(X_0)}
\left(
IC_{X_0}^H,
W_k^H
\right).
\]
\end{lemma}

\begin{proof}
The same finite-additivity argument applies in the abelian category
$MHM(X_0)$, in the first and second variables respectively.  Again,
finiteness of $\Sigma$ removes any distinction between finite products
and finite direct sums.
\end{proof}

\begin{proposition}
\label{prop:nodewise-ext-decomposition}
In either Yoneda orientation, an extension class with point-supported
term $Q_\Sigma^H$ is canonically determined by its nodewise components.
Equivalently, the corresponding finite-node extension datum is a tuple
of local classes indexed by $\Sigma$.
\end{proposition}

\begin{proof}
This is precisely the content of
Lemma~\ref{lem:ext-additivity-mhm}.  The decomposition follows from
finite support alone and is independent of the subsequent determination
of which Yoneda orientation is realized by the corrected mixed Hodge
module.
\end{proof}

\begin{remark}
\label{rem:nodewise-class-meaning}
The decomposition in
Proposition~\ref{prop:nodewise-ext-decomposition} concerns extension
data on $X_0$.  It neither asserts that the global vanishing lattice is
freely generated by the $r$ local vanishing cycles nor identifies the
coordinates of an extension class with independent coordinates in
global vanishing cohomology.
\end{remark}

\subsection{Nodewise one-dimensionality on the perverse side}

We now determine the local rational extension groups.  The decisive fact
is that the costalk of the intersection complex at a threefold ordinary
double point contains a single one-dimensional group in the degree
governing extensions by a skyscraper perverse sheaf.

\begin{lemma}[Closed-point adjunction]
\label{lem:nodewise-perverse-ext-localization-strong}
For every node $p_k$ there is a natural isomorphism
\[
\Ext^1_{\Perv(X_0;\Q)}
\left(
i_{k*}\Q_{\{p_k\}},
IC_{X_0}
\right)
\cong
H^1\!\left(i_k^!IC_{X_0}\right).
\]
In particular, this extension group depends only on the local germ of
$X_0$ at $p_k$.
\end{lemma}

\begin{proof}
For objects in the heart of the perverse $t$-structure, Yoneda
extensions of degree one are represented by degree-one morphisms in the
constructible derived category.  Hence
\[
\begin{aligned}
\Ext^1_{\Perv(X_0;\Q)}
\left(
i_{k*}\Q,
IC_{X_0}
\right)
&\cong
\Hom_{D^b_c(X_0)}
\left(
i_{k*}\Q,
IC_{X_0}[1]
\right)\\
&\cong
\Hom_{D^b_c(\{p_k\})}
\left(
\Q,
i_k^!IC_{X_0}[1]
\right)\\
&\cong
H^1(i_k^!IC_{X_0}),
\end{aligned}
\]
where the second isomorphism is closed-point adjunction
\cite{BBD,KS}.  The final expression is local at $p_k$.
\end{proof}

\begin{proposition}[Local IC costalk at an ODP]
\label{prop:local-odp-ext-zigzag-map}
\label{prop:local-odp-costalk}
Let $p$ be a threefold ordinary double point and let
$IC_{\mathrm{loc}}=j_{!*}\Q_{U_{\mathrm{loc}}}[3]$.  Then
\[
H^1(i^!IC_{\mathrm{loc}})
\cong
\Q.
\]
\end{proposition}

\begin{proof}
Let $B$ be a sufficiently small ball centered at $p$, and let
$L:=U_{\mathrm{loc}}\cap\partial B$ be the link of the singularity.
For a threefold ordinary double point,
\[
L\simeq S^2\times S^3,
\]
so in particular
\[
H^2(L;\Q)\cong\Q
\]
\cite{MilnorSingularPoints,DimcaSheaves}.

The stalk of $Rj_*\Q_{U_{\mathrm{loc}}}[3]$ at $p$ is computed from the
cohomology of the link.  For an isolated singular stratum in complex
dimension three, middle-perversity intermediate extension gives
\[
H^{-1}(i^*IC_{\mathrm{loc}})
\cong
H^2(L;\Q)
\cong
\Q
\]
\cite{BBD,DimcaSheaves}.  Since the intersection complex is Verdier
self-dual,
\[
H^1(i^!IC_{\mathrm{loc}})
\cong
H^{-1}(i^*IC_{\mathrm{loc}})^\vee
\cong
\Q.
\]
\end{proof}

\begin{proposition}
\label{prop:local-odp-ext-nonzero}
The corrected local perverse ODP extension determines a nonzero class
\[
e_{\mathrm{loc}}
\in
\Ext^1_{\Perv(X_{0,\mathrm{loc}};\Q)}
\left(
i_*\Q_{\{p\}},
IC_{\mathrm{loc}}
\right).
\]
\end{proposition}

\begin{proof}
The corrected local object is the non-split perverse extension obtained
from the variation-cone construction
\cite{RahmanSchoberPaper,RahmanPerverseNearbyCycles}.  Under the Yoneda
classification of short exact sequences, the zero class is represented
by the split extension.  Since the corrected ODP extension is non-split,
its class $e_{\mathrm{loc}}$ is nonzero.
\end{proof}

\begin{proposition}
\label{prop:local-odp-ext-dim-upper-bound}
One has
\[
\dim_\Q
\Ext^1_{\Perv(X_{0,\mathrm{loc}};\Q)}
\left(
i_*\Q_{\{p\}},
IC_{\mathrm{loc}}
\right)
\leq1.
\]
\end{proposition}

\begin{proof}
By Lemma~\ref{lem:nodewise-perverse-ext-localization-strong} and
Proposition~\ref{prop:local-odp-costalk},
\[
\Ext^1_{\Perv(X_{0,\mathrm{loc}};\Q)}
\left(
i_*\Q_{\{p\}},
IC_{\mathrm{loc}}
\right)
\cong
H^1(i^!IC_{\mathrm{loc}})
\cong
\Q.
\]
The asserted bound follows.
\end{proof}

\begin{corollary}
\label{cor:local-odp-ext-dim-one}
There is a one-dimensional local extension space
\[
\Ext^1_{\Perv(X_{0,\mathrm{loc}};\Q)}
\left(
i_*\Q_{\{p\}},
IC_{\mathrm{loc}}
\right)
\cong
\Q,
\]
and the corrected local class $e_{\mathrm{loc}}$ is a generator.
\end{corollary}

\begin{proof}
The dimension statement follows from
Proposition~\ref{prop:local-odp-costalk}, and
$e_{\mathrm{loc}}\neq0$ by
Proposition~\ref{prop:local-odp-ext-nonzero}.  Since the extension space
is one-dimensional, every nonzero class generates it.
\end{proof}

\begin{remark}[Relation with the MacPherson--Vilonen description]
The costalk computation is equivalent to the rank-one
MacPherson--Vilonen calculation used locally in Section~4
\cite{MacPhersonVilonen1986,GMV1996}.  Under that description, the
nonzero element of
$H^1(i^!IC_{\mathrm{loc}})$
corresponds to the unique nontrivial local ODP extension direction.
The costalk formulation has the advantage of identifying the extension
space without introducing a basis-dependent matrix parameter.
\end{remark}

\begin{corollary}
\label{cor:nodewise-perverse-ext-dim-one-strong}
For every node $p_k\in\Sigma$,
\[
\dim_\Q
\Ext^1_{\Perv(X_0;\Q)}
\left(
i_{k*}\Q_{\{p_k\}},
IC_{X_0}
\right)
=
1.
\]
If $e_k$ denotes the image of the corrected local ODP class under the
localization isomorphism, then $e_k$ generates the corresponding
nodewise extension group.
\end{corollary}

\begin{proof}
Apply
Lemma~\ref{lem:nodewise-perverse-ext-localization-strong}
at $p_k$ and use
Corollary~\ref{cor:local-odp-ext-dim-one}.
\end{proof}

\begin{definition}
\label{def:nodewise-perverse-generator}
For each $p_k\in\Sigma$, let
\[
e_k
\in
\Ext^1_{\Perv(X_0;\Q)}
\left(
i_{k*}\Q_{\{p_k\}},
IC_{X_0}
\right)
\]
denote the distinguished generator determined by the corrected local
ODP extension.
\end{definition}

\begin{corollary}
\label{cor:perverse-ext-node-basis-strong}
There is a canonical nodewise decomposition
\[
\Ext^1_{\Perv(X_0;\Q)}
\left(
Q_\Sigma,
IC_{X_0}
\right)
\cong
\bigoplus_{k=1}^r \Q e_k.
\]
Thus the finite singular set indexes a distinguished family of local
extension directions in the corrected perverse theory.
\end{corollary}

\begin{proof}
Combine Lemma~\ref{lem:ext-additivity-perv} with
Corollary~\ref{cor:nodewise-perverse-ext-dim-one-strong}.
\end{proof}

\begin{proposition}
\label{prop:global-perverse-class-node-expansion-strong}
The corrected finite-node perverse extension class admits a unique
expansion
\[
[\mathcal P]_{\mathrm{perv}}
=
\sum_{k=1}^r c_k e_k,
\qquad
c_k\in\Q.
\]
\end{proposition}

\begin{proof}
By Corollary~\ref{cor:perverse-ext-node-basis-strong}, the classes
$e_1,\dots,e_r$ form the distinguished nodewise basis of
$\Ext^1_{\Perv(X_0;\Q)}(Q_\Sigma,IC_{X_0})$.  The asserted expansion and
its uniqueness follow.
\end{proof}

\begin{remark}
\label{rem:nodewise-physics-shadow}
The coefficient vector $(c_1,\dots,c_r)$ records the coordinates of the
corrected extension class with respect to the distinguished nodewise
extension directions.  It is an extension-theoretic invariant on the
central fiber and should not be identified with coordinates in a basis
of the global vanishing lattice.  Relations among the images of the
vanishing cycles in global cohomology may occur independently of the
nodewise Ext decomposition.
\end{remark}

\subsection{Realization and nodewise compatibility}

Exactness and additivity of rational realization give
\[
\rat(Q_\Sigma^H)
\simeq
Q_\Sigma
=
\bigoplus_{k=1}^r i_{k*}\Q_{\{p_k\}}.
\]
More generally, realization carries short exact sequences of mixed Hodge
modules to short exact sequences of rational perverse sheaves and is
compatible with the inclusions of the individual node-supported
summands.

\begin{proposition}[Nodewise compatibility under realization]
\label{prop:nodewise-realization-compatibility}
Rational realization is compatible with the finite support
decomposition.  In particular, if a mixed-Hodge-module extension class
with finite point-supported term is decomposed into its nodewise
components, then its rational realization is the tuple of the realized
local classes.

For the orientation whose rational shadow lies in
$\Ext^1_{\Perv(X_0;\Q)}(Q_\Sigma,IC_{X_0})$, the realized class therefore
belongs to
\[
\Ext^1_{\Perv(X_0;\Q)}
\left(
Q_\Sigma,
IC_{X_0}
\right)
\cong
\bigoplus_{k=1}^r\Q e_k.
\]
\end{proposition}

\begin{proof}
The realization functor is exact and additive, and hence commutes with
the finite direct-sum decomposition of $Q_\Sigma^H$.  Applying
realization to a representative short exact sequence gives the
corresponding rational short exact sequence.  Naturality with respect to
the inclusion of each summand $W_k^H\hookrightarrow Q_\Sigma^H$ shows
that the resulting global class decomposes into the realizations of its
nodewise components.
\end{proof}

The results of this section isolate the two finite-support statements
needed for the global construction.  On the Hodge-theoretic side, the
local ODP interface blocks assemble canonically to the point-supported
object $Q_\Sigma^H$.  On the rational perverse side, each node contributes
exactly one extension direction, giving
\[
\Ext^1_{\Perv(X_0;\Q)}
\left(
Q_\Sigma,
IC_{X_0}
\right)
\cong
\bigoplus_{k=1}^r\Q e_k.
\]
These statements concern the local organization of the singular data on
$X_0$ and do not impose any independence hypothesis on their images in
global vanishing cohomology.  Section~6 uses this decomposition together
with the local $\can^H/\Var^H/N$ data of Section~4 to construct the
global mixed-Hodge-module refinement.
%----------------------------------------------------
\section{Global gluing theorem}
\label{sec:global-gluing}

We now pass from the local ordinary-double-point construction to the
global finite-node degeneration.  Section~4 associates to each node a
canonical rank-one Hodge block together with the incoming and outgoing
morphisms through which the reduced ODP sector enters the
$\can^H/\Var^H/N$ geometry.  Section~5 assembles these local blocks into
the finite point-supported object
\[
Q_\Sigma^H
=
\bigoplus_{k=1}^r W_k^H
\simeq
\bigoplus_{k=1}^r i_{k*}\Q_{\{p_k\}}^H(-1)
\]
and establishes the corresponding nodewise decomposition of the rational
perverse extension space.  The remaining problem is to identify the
global mixed Hodge module determined by these local data and to prove
that its rational realization is the prescribed corrected perverse
object.

Two geometric structures enter this passage.  The nearby- and
vanishing-cycle objects arise from the principal-divisor geometry
$X_0\subset\mathcal X$ associated with the degeneration
$\pi:\mathcal X\to\Delta$ and hence belong to Saito's
nearby-/vanishing-cycle formalism \cite{SaitoMHM,SaitoDuality}.  Once
these Hodge-theoretic data have been transferred to the central fiber,
the finite singular set is organized by the recollement
$U\hookrightarrow X_0\hookleftarrow\Sigma$.  The former supplies the
$\can^H/\Var^H/N$ structure; the latter organizes the smooth and
point-supported constituents of the resulting object on $X_0$.  These
two roles remain distinct throughout the global construction.

Let
$\mathcal K^H:=\Q_{\mathcal X}^H[4]$ and set
\[
\Psi^H
:=
\psi_{\pi,1}^H(\mathcal K^H),
\qquad
\Phi^H
:=
\phi_{\pi,1}^H(\mathcal K^H).
\]
These are mixed Hodge modules on $X_0$.  We retain the
perverse-normalized convention of Sections~2--3 and write
\[
\can^H:\Psi^H\longrightarrow\Phi^H,
\qquad
\Var^H:\Phi^H\longrightarrow\Psi^H(-1),
\]
with
$\Var^H\circ\can^H=N$.
Accordingly, the global Hodge-theoretic candidate for the corrected
object is the variation cone
\[
\Cone\!\left(
\Var^H:
\Phi^H\longrightarrow\Psi^H(-1)
\right),
\]
whose rational counterpart is
$\mathcal P=\Cone(\var)$.

\subsection{The global point-supported Hodge object}

Let $\Sigma=\{p_1,\dots,p_r\}\subset X_0$ be the finite node set,
let $U:=X_0\setminus\Sigma$, and denote the inclusions by
$j:U\hookrightarrow X_0$ and
$i_k:\{p_k\}\hookrightarrow X_0$.

For each $p_k$, choose a sufficiently small analytic neighborhood
$B_k\subset X_0$ containing no other singular point.  Section~4
constructs on $B_k$ a rank-one point-supported Hodge module $W_k^H$
together with morphisms $u_k$ and $v_k$.  Since $W_k^H$ is supported at
the closed point $p_k$, its extension by zero to $X_0$ will again be
denoted by $W_k^H$.

\begin{lemma}
\label{lem:global-singular-object}
The local ODP interface blocks assemble canonically to
\[
Q_\Sigma^H
:=
\bigoplus_{k=1}^r W_k^H
\simeq
\bigoplus_{k=1}^r
i_{k*}\Q_{\{p_k\}}^H(-1).
\]
Moreover,
\[
\rat(Q_\Sigma^H)
\simeq
Q_\Sigma
:=
\bigoplus_{k=1}^r
i_{k*}\Q_{\{p_k\}}.
\]
\end{lemma}

\begin{proof}
The supports $\{p_k\}$ are pairwise disjoint closed points, and finite
direct sums exist in the abelian category $MHM(X_0)$.  Hence the local
blocks assemble canonically to
$Q_\Sigma^H=\bigoplus_k W_k^H$.  Exactness and additivity of the
realization functor, together with the local identification
$\rat(W_k^H)\simeq i_{k*}\Q_{\{p_k\}}$, give
\[
\rat(Q_\Sigma^H)
\simeq
\bigoplus_{k=1}^r\rat(W_k^H)
\simeq
\bigoplus_{k=1}^r i_{k*}\Q_{\{p_k\}}.
\]
\end{proof}

\begin{figure}[t]
\centering
\[
\renewcommand{\arraystretch}{1.25}
\begin{array}{ccccc}
W_1^H & \oplus & \cdots & \oplus & W_r^H
\\[5pt]
&& \Big\downarrow &&
\\[5pt]
&& Q_\Sigma^H &&
\\[5pt]
&& \Big\downarrow\;{\scriptstyle\rat} &&
\\[6pt]
&& \displaystyle\bigoplus_{k=1}^r
i_{k*}\Q_{\{p_k\}} &&
\end{array}
\]
\caption{Finite assembly of the local ODP interface blocks on the
central fiber. The decomposition is support-theoretic and does not
imply linear independence of the corresponding vanishing cycles in
global cohomology.}
\label{fig:global-node-blocks}
\end{figure}

\subsection{Globalization of the local attaching maps}

The local morphisms constructed in Section~4 have point-supported source
or target.  Closed-support adjunction therefore promotes them canonically
to morphisms on $X_0$.

\begin{lemma}[Extension of the local maps]
\label{lem:extend-local-gluing-maps}
For each $k$, the local morphisms determine canonical global maps
\[
\widehat u_k:
\Psi^H\longrightarrow W_k^H,
\qquad
\widehat v_k:
W_k^H\longrightarrow\Psi^H(-1),
\]
whose restrictions to $B_k$ agree with $u_k$ and $v_k$, respectively.
On a neighborhood $B_j$ with $j\neq k$, their node-supported components
vanish.
\end{lemma}

\begin{proof}
Since $W_k^H$ is supported at the closed point $p_k$, a morphism into
$W_k^H$ is determined by its germ at $p_k$.  The local incoming map
therefore determines a unique global morphism
$\widehat u_k:\Psi^H\to W_k^H$.

Likewise, a morphism from $W_k^H$ is determined locally at its support,
so $v_k$ determines a unique morphism
$\widehat v_k:W_k^H\to\Psi^H(-1)$.  Pairwise disjointness of the node
supports gives the asserted vanishing away from $p_k$.
\end{proof}

\begin{proposition}
\label{prop:global-gluing-morphisms}
There are canonical morphisms
\[
u_\Sigma:
\Psi^H\longrightarrow Q_\Sigma^H,
\qquad
v_\Sigma:
Q_\Sigma^H\longrightarrow\Psi^H(-1),
\]
whose $k$-th components are $\widehat u_k$ and $\widehat v_k$,
respectively.
\end{proposition}

\begin{proof}
Since $Q_\Sigma^H=\bigoplus_k W_k^H$ is a finite direct sum, the maps
$\widehat u_k$ determine a unique morphism
$u_\Sigma:\Psi^H\to Q_\Sigma^H$.  Dually, the maps
$\widehat v_k$ assemble by finite summation to
$v_\Sigma:Q_\Sigma^H\to\Psi^H(-1)$.
\end{proof}

The relevant global monodromy morphism is the component of $N$ detected
by the finite collection of reduced ordinary-double-point sectors.

\begin{definition}[Node-supported monodromy]
\label{def:N-Sigma}
Define
\[
N_\Sigma
:=
\sum_{k=1}^r
\widehat v_k\widehat u_k:
\Psi^H\longrightarrow\Psi^H(-1).
\]
Equivalently, $N_\Sigma$ restricts on the reduced ODP sector at $p_k$ to
the local monodromy morphism $N_k$ and vanishes on the complementary
sector not detected by the finite node contribution.
\end{definition}

\begin{proposition}
\label{prop:global-monodromy-relation}
The assembled maps satisfy
\[
v_\Sigma u_\Sigma=N_\Sigma.
\]
For every sufficiently small neighborhood $B_k$, this identity restricts
to
\[
v_ku_k=N_k.
\]
Thus $N_\Sigma$ is precisely the node-supported component of monodromy
detected by the finite ODP interface.
\end{proposition}

\begin{proof}
By the direct-sum construction,
\[
v_\Sigma u_\Sigma
=
\sum_{k=1}^r
\widehat v_k\widehat u_k.
\]
There are no mixed contributions between distinct node-supported
summands, because their supports are disjoint.  The right-hand side is
therefore $N_\Sigma$ by definition.  Restriction to $B_k$ eliminates all
summands except the $k$-th and yields
$v_ku_k=N_k$ by
Theorem~\ref{thm:local-monodromy-factorization}.
\end{proof}

\begin{figure}[t]
\centering
\[
\begin{tikzcd}[
    row sep=4.5em,
    column sep=6em
]
&
\Psi^H
\arrow[dl, "\widehat u_1"']
\arrow[d, "u_\Sigma"]
\arrow[dr, "\widehat u_r"]
&
\\
W_1^H
\arrow[r]
\arrow[dr, "\widehat v_1"']
&
Q_\Sigma^H
\arrow[d, "v_\Sigma"]
&
W_r^H
\arrow[l]
\arrow[dl, "\widehat v_r"]
\\
&
\Psi^H(-1)
&
\end{tikzcd}
\]
\caption{Nodewise assembly of the local interface morphisms. On each
reduced ODP sector the composite is the local monodromy morphism; the
global composite is $N_\Sigma$.}
\label{fig:global-nodewise-gluing}
\end{figure}

\subsection{Filtered compatibility of the assembled datum}

\begin{proposition}[Filtered admissibility of the global nodewise datum]
\label{prop:global-filtered-admissibility}
The object $Q_\Sigma^H$ and the morphisms $u_\Sigma$ and $v_\Sigma$
belong to the mixed-Hodge-module category and are strict with respect to
the Hodge filtration.  Their underlying filtered structures inherit the
Kashiwara--Malgrange compatibility of the local nearby-/vanishing-cycle
construction.  Moreover,
\[
v_\Sigma u_\Sigma=N_\Sigma
\]
as an identity of filtered morphisms.
\end{proposition}

\begin{proof}
For each $k$, the objects and morphisms
$W_k^H$, $u_k$, and $v_k$ are constructed inside Saito's
nearby-/vanishing-cycle formalism and are therefore morphisms of mixed
Hodge modules, strict for the Hodge filtration
\cite{SaitoMHM}.  Their filtered $(\mathcal D,V)$-structures are induced
from the local nearby-/vanishing-cycle package.  Finite direct sums
preserve these properties, so the assembled object and morphisms inherit
the same filtered compatibility.  The final identity is
Proposition~\ref{prop:global-monodromy-relation}.
\end{proof}

\begin{remark}
No additional global $V$-filtration calculation is required.  The
singular support is finite and pairwise disjoint, so the filtered
node-supported datum is obtained by finite direct sum from the local
structures already established in Section~4.
\end{remark}

\subsection{The global node-supported Hodge datum}

The finite collection of local interface blocks now defines a canonical
Hodge-theoretic datum on the central fiber.

\begin{theorem}[Global node-supported Hodge datum]
\label{thm:global-gluing-datum}
The triple
\[
\left(
Q_\Sigma^H,u_\Sigma,v_\Sigma
\right)
\]
is the finite node-supported Hodge datum on $X_0$ induced by the
nearby-/vanishing-cycle geometry of the degeneration.  Its restriction
to a sufficiently small neighborhood of $p_k$ is the local ODP datum of
Section~4, and
\[
v_\Sigma u_\Sigma=N_\Sigma.
\]
Together with the smooth central-fiber object $\Q_U^H[3]$, these data
give the smooth and point-supported constituents entering the global
corrected mixed Hodge module.
\end{theorem}

\begin{proof}
The point-supported object is
Lemma~\ref{lem:global-singular-object}, the global morphisms are supplied
by Proposition~\ref{prop:global-gluing-morphisms}, and their filtered
compatibility and monodromy relation are
Propositions~\ref{prop:global-filtered-admissibility} and
\ref{prop:global-monodromy-relation}.  Restriction to $B_k$ removes all
summands supported at $p_j$ with $j\neq k$, leaving precisely the local
datum $(W_k^H,u_k,v_k)$.
\end{proof}

\begin{remark}[Divisor gluing and central-fiber support]
The theorem is not an application of the principal-divisor gluing
formalism to $\Sigma\subset X_0$.  The principal-divisor input comes
from
\[
X_0\subset\mathcal X
\]
and supplies the nearby-/vanishing-cycle objects and the
$\can^H/\Var^H/N$ structure.  The finite set
$\Sigma=\operatorname{Sing}(X_0)$ is instead the support of the local
ODP interface blocks on the central fiber and indexes their canonical
direct-sum decomposition.
\end{remark}

\subsection{Construction of the global mixed Hodge module}
\label{subsec:global-construction}

Two point-supported Hodge objects arise naturally in the global theory.
The first is the interface object
\[
Q_\Sigma^H
=
\bigoplus_{k=1}^r
i_{k*}\Q_{\{p_k\}}^H(-1),
\]
through which the node-supported monodromy factors:
\[
\Psi^H
\xrightarrow{u_\Sigma}
Q_\Sigma^H
\xrightarrow{v_\Sigma}
\Psi^H(-1).
\]
The second is the point-supported term appearing in the forward Yoneda
orientation,
\[
R_\Sigma^H
:=
\bigoplus_{k=1}^r R_k^H,
\qquad
R_k^H
:=
i_{k*}\Q_{\{p_k\}}^H(-2).
\]
These objects have the same rational realization,
\[
\rat(Q_\Sigma^H)
\simeq
\rat(R_\Sigma^H)
\simeq
Q_\Sigma
=
\bigoplus_{k=1}^r i_{k*}\Q_{\{p_k\}},
\]
but they have different Tate normalizations and different categorical
roles.  The global object itself is defined independently of either
choice of Yoneda orientation, directly from the variation morphism.

\begin{definition}[Global Hodge-theoretic variation cone]
\label{def:global-derived-hodge-cone}
Define
\[
\widetilde{\mathcal P}^H
:=
\Cone\!\left(
\Var^H:
\Phi^H
\longrightarrow
\Psi^H(-1)
\right)
\in
D^bMHM(X_0).
\]
\end{definition}

\begin{remark}[Normalization of the global variation cone]
\label{rem:global-variation-fiber-normalization}
The morphism $\Var^H$ in
Definition~\ref{def:global-derived-hodge-cone} is
perverse-normalized.  Accordingly,
$\widetilde{\mathcal P}^H$ is the mixed-Hodge-module counterpart of the
corrected perverse object $\mathcal P=\Cone(\var)$ fixed in
Sections~\ref{sec:intro}--\ref{sec:mixed-hodge-module-div-gluing}.
The classical and perverse-normalized presentations are related by the
normalization identity established there.
\end{remark}

\begin{theorem}[Global existence and heart concentration]
\label{thm:global-existence}
The derived object $\widetilde{\mathcal P}^H$ is concentrated in degree
zero for the mixed-Hodge-module $t$-structure.  Denote the resulting
heart object by
\[
\mathcal P^H\in MHM(X_0).
\]
Then
\[
j^*\mathcal P^H
\simeq
\Q_U^H[3],
\]
and for every sufficiently small analytic neighborhood $B_k$ containing
only the node $p_k$,
\[
\mathcal P^H|_{B_k}
\simeq
\mathcal P_{\mathrm{loc},k}^H.
\]
The node-supported interface datum determined by $\mathcal P^H$ is
$(Q_\Sigma^H,u_\Sigma,v_\Sigma)$.
\end{theorem}

\begin{proof}
Compatibility of realization with nearby and vanishing cycles, with
$\Var^H$, and with derived cones gives
\[
\rat\!\left(
\widetilde{\mathcal P}^H
\right)
\simeq
\Cone\!\left(
\var:
\Phi\longrightarrow\Psi
\right)
=
\mathcal P.
\]
The corrected object $\mathcal P$ is perverse by the rational
construction established previously
\cite{RahmanSchoberPaper,RahmanPerverseNearbyCycles}.
Since realization is exact and faithful on mixed Hodge modules, a
nonzero cohomology object
$\mathcal H^m(\widetilde{\mathcal P}^H)$ with $m\neq0$ would have
nonzero rational realization and hence would produce nonzero perverse
cohomology of $\mathcal P$ outside degree zero.  This is impossible.
Thus $\widetilde{\mathcal P}^H$ is concentrated in the heart.

On the smooth locus $U$, there is no vanishing contribution.  The
corrected variation-cone construction therefore restricts to the
perverse constant Hodge module $\Q_U^H[3]$.

Finally, restriction to the open neighborhood $B_k$ commutes with the
nearby-/vanishing-cycle morphisms involved here and with the derived
cone.  Hence
\[
\mathcal P^H|_{B_k}
\simeq
\Cone\!\left(
\Var^H|_{B_k}
\right)
\simeq
\mathcal P_{\mathrm{loc},k}^H
\]
by Theorem~\ref{thm:local-odp-reconstruction}.  The local
$\can^H/\Var^H/N$ analysis then identifies the restricted interface
block with $W_k^H$, and finite direct-sum assembly gives
$(Q_\Sigma^H,u_\Sigma,v_\Sigma)$.
\end{proof}

\begin{remark}[Two point-supported objects]
\label{rem:global-two-point-objects}
The interface block
$Q_\Sigma^H\simeq\bigoplus_k i_{k*}\Q^H(-1)$
belongs to the monodromy factorization
$v_\Sigma u_\Sigma=N_\Sigma$.  In the forward Yoneda orientation with
$IC_{X_0}^H$ as subobject, the point-supported quotient is instead
$R_\Sigma^H\simeq\bigoplus_k i_{k*}\Q^H(-2)$.  These are not competing
descriptions of a single Hodge object: they are distinct
Tate-normalized objects with a common rational realization.
\end{remark}

\begin{corollary}[Exact-sequence form in the forward orientation]
\label{cor:global-exact-orientation}
In the forward Yoneda orientation with $IC_{X_0}^H$ as subobject, the
global corrected object fits into a short exact sequence
\[
0
\longrightarrow
IC_{X_0}^H
\longrightarrow
\mathcal P^H
\longrightarrow
R_\Sigma^H
\longrightarrow
0,
\]
where
\[
R_\Sigma^H
=
\bigoplus_{k=1}^r
i_{k*}\Q_{\{p_k\}}^H(-2).
\]
\end{corollary}

\begin{proof}
For each $k$, the local corrected mixed Hodge module admits the
corresponding forward extension
\[
0
\longrightarrow
IC_{\mathrm{loc},k}^H
\longrightarrow
\mathcal P_{\mathrm{loc},k}^H
\longrightarrow
R_k^H
\longrightarrow
0,
\qquad
R_k^H
=
i_{k*}\Q_{\{p_k\}}^H(-2).
\]
The supports of the quotient terms are finite and pairwise disjoint, so
they assemble canonically to
$R_\Sigma^H=\bigoplus_kR_k^H$.
On $U$, the object $\mathcal P^H$ restricts to $\Q_U^H[3]$, whose
intermediate extension is $IC_{X_0}^H$.  Recollement on
$U\hookrightarrow X_0\hookleftarrow\Sigma$ therefore yields the stated
global exact sequence.
\end{proof}

\begin{theorem}[Global exact sequence]
\label{thm:global-exact-sequence}
In the forward Yoneda orientation,
\[
0
\longrightarrow
IC_{X_0}^H
\longrightarrow
\mathcal P^H
\longrightarrow
\bigoplus_{k=1}^r
i_{k*}\Q_{\{p_k\}}^H(-2)
\longrightarrow
0.
\]
\end{theorem}

\begin{proof}
This is Corollary~\ref{cor:global-exact-orientation} after substituting
the definition of $R_\Sigma^H$.
\end{proof}

\begin{remark}[Dual orientation]
\label{rem:global-dual-orientation}
Verdier duality interchanges the two extension orientations
\cite{SaitoDuality}.  In the dual orientation, the
$i_{k*}\Q_{\{p_k\}}^H(-1)$ interface blocks occur as point-supported
subobjects.  Thus $Q_\Sigma^H$ is naturally associated with the dual
orientation, whereas $R_\Sigma^H$ is the point-supported term in the
forward orientation displayed above.
\end{remark}

\subsection{Realization of the global corrected object}

\begin{theorem}[Global realization]
\label{thm:global-realization}
The rational realization of the global mixed Hodge module is the
corrected finite-node perverse object:
\[
\rat(\mathcal P^H)
\simeq
\mathcal P.
\]
Moreover, realization sends the forward Hodge-module exact sequence of
Theorem~\ref{thm:global-exact-sequence} to
\[
0
\longrightarrow
IC_{X_0}
\longrightarrow
\mathcal P
\longrightarrow
Q_\Sigma
\longrightarrow
0,
\qquad
Q_\Sigma
=
\bigoplus_{k=1}^r i_{k*}\Q_{\{p_k\}}.
\]
\end{theorem}

\begin{proof}
By construction and Theorem~\ref{thm:global-existence},
\[
\rat(\mathcal P^H)
\simeq
\rat\!\left(
\Cone(\Var^H)
\right)
\simeq
\Cone(\var)
=
\mathcal P.
\]
For the exact sequence, realization is exact and forgets Tate twists.
Hence
\[
\rat(R_\Sigma^H)
\simeq
\bigoplus_{k=1}^r
i_{k*}\Q_{\{p_k\}}
=
Q_\Sigma.
\]
Applying $\rat$ to the forward mixed-Hodge-module extension therefore
gives
\[
0
\longrightarrow
IC_{X_0}
\longrightarrow
\mathcal P
\longrightarrow
Q_\Sigma
\longrightarrow
0,
\]
which is the corrected finite-node perverse extension.
\end{proof}

\begin{remark}[Rational realization and the two Tate normalizations]
Both
$Q_\Sigma^H=\bigoplus_k i_{k*}\Q_{\{p_k\}}^H(-1)$
and
$R_\Sigma^H=\bigoplus_k i_{k*}\Q_{\{p_k\}}^H(-2)$
have rational realization $Q_\Sigma$.  The distinction between the
interface normalization and the forward-extension normalization is
therefore invisible in $\Perv(X_0;\Q)$ and is retained only at the
mixed-Hodge-module level.
\end{remark}

\subsection{Global cone description}

The global mixed Hodge module is defined by the full variation morphism
$\Var^H$, not by the point-supported outgoing morphism $v_\Sigma$ alone.
The latter records the restriction of the global variation geometry to
the finite interface block.

\begin{proposition}[Global cone description]
\label{prop:global-cone-description}
In $D^bMHM(X_0)$ one has
\[
\mathcal P^H
\simeq
\Cone\!\left(
\Var^H:
\Phi^H
\longrightarrow
\Psi^H(-1)
\right).
\]
\end{proposition}

\begin{proof}
This is
Definition~\ref{def:global-derived-hodge-cone} together with the heart
concentration established in
Theorem~\ref{thm:global-existence}.
\end{proof}

\begin{remark}[Role of the node-supported morphism]
The morphism
$v_\Sigma:Q_\Sigma^H\to\Psi^H(-1)$
is the restriction of the global variation geometry to the
node-supported interface block.  It controls the local
$\can^H/\Var^H/N$ factorization and the comparison with limiting Hodge
theory developed in Section~7.  The corrected object itself is the full
variation cone of
Proposition~\ref{prop:global-cone-description}; no identification
$\mathcal P^H\simeq\Cone(v_\Sigma)$ is asserted or required.
\end{remark}

\subsection{Recovery of the local model}

\begin{proposition}[Recovery of the local ODP model]
\label{prop:global-restricts-local-odp}
For every $p_k\in\Sigma$ and every sufficiently small analytic
neighborhood $B_k$ containing no other singular point,
\[
\mathcal P^H|_{B_k}
\simeq
\mathcal P_{\mathrm{loc},k}^H.
\]
Under this identification,
\[
Q_\Sigma^H|_{B_k}
\simeq
W_k^H,
\qquad
R_\Sigma^H|_{B_k}
\simeq
R_k^H.
\]
\end{proposition}

\begin{proof}
Restriction to the open neighborhood $B_k$ commutes with the morphisms
and derived cones involved in the construction.  Consequently,
\[
\mathcal P^H|_{B_k}
\simeq
\Cone\!\left(
\Var^H|_{B_k}
\right)
\simeq
\mathcal P_{\mathrm{loc},k}^H
\]
by the local reconstruction theorem.  Since all other node-supported
summands vanish on $B_k$, the finite support decompositions give
$Q_\Sigma^H|_{B_k}\simeq W_k^H$ and
$R_\Sigma^H|_{B_k}\simeq R_k^H$.
\end{proof}

\subsection{Global extension class and rigidity}

Once the forward Yoneda orientation has been fixed, the relevant global
extension group is
\[
\Ext^1_{MHM(X_0)}
\left(
R_\Sigma^H,
IC_{X_0}^H
\right).
\]
The interface object $Q_\Sigma^H$ belongs to the monodromy
factorization and should not be substituted for $R_\Sigma^H$ in this
forward extension group.

\begin{lemma}[Nodewise decomposition in the forward orientation]
\label{lem:global-class-from-local}
There is a natural decomposition
\[
\Ext^1_{MHM(X_0)}
\left(
R_\Sigma^H,
IC_{X_0}^H
\right)
\cong
\bigoplus_{k=1}^r
\Ext^1_{MHM(X_0)}
\left(
R_k^H,
IC_{X_0}^H
\right).
\]
Consequently, the global forward extension class is determined by its
nodewise components.
\end{lemma}

\begin{proof}
Since
$R_\Sigma^H=\bigoplus_kR_k^H$
is a finite direct sum, additivity of Yoneda $\Ext^1$ in the first
variable gives the displayed decomposition.  Under this isomorphism, a
global extension class is identified with the tuple of its local
components.
\end{proof}

\begin{proposition}[Global rigidity with fixed local Hodge data]
\label{prop:global-rigidity}
Let $\mathcal E^H\in MHM(X_0)$ satisfy:
\begin{enumerate}
\item $\rat(\mathcal E^H)\simeq\mathcal P$;
\item $j^*\mathcal E^H\simeq\Q_U^H[3]$;
\item its point-supported quotient in the forward orientation is
$R_\Sigma^H$; and
\item for every $p_k\in\Sigma$, the restriction
$\mathcal E^H|_{B_k}$ represents the same normalized local Hodge
extension as $\mathcal P_{\mathrm{loc},k}^H$.
\end{enumerate}
Then
\[
\mathcal E^H\simeq\mathcal P^H.
\]
\end{proposition}

\begin{proof}
By condition~(4), the nodewise forward Yoneda class of
$\mathcal E^H|_{B_k}$ agrees with that of
$\mathcal P_{\mathrm{loc},k}^H$ for every $k$.  Under the decomposition
of Lemma~\ref{lem:global-class-from-local}, the global classes of
$\mathcal E^H$ and $\mathcal P^H$ therefore have identical components
in
\[
\Ext^1_{MHM(X_0)}
\left(
R_\Sigma^H,
IC_{X_0}^H
\right).
\]
The two Yoneda classes coincide, and equivalent extensions have
isomorphic middle terms.  Hence
$\mathcal E^H\simeq\mathcal P^H$.
\end{proof}

\begin{remark}[Normalized local rigidity]
The rigidity statement depends on the full normalized local Hodge datum,
not merely on the rational realization.  Once the endpoint Hodge
structures, the Yoneda orientation, and the normalization
$v_{\mathrm{loc}}u_{\mathrm{loc}}=N$ are fixed, the local extension
classes determine the global class through
Lemma~\ref{lem:global-class-from-local}.
\end{remark}

\subsection{Summary of the global construction}

The global finite-node construction distinguishes two canonical
point-supported Hodge objects.  The interface block is
\[
Q_\Sigma^H
=
\bigoplus_{k=1}^r
i_{k*}\Q_{\{p_k\}}^H(-1),
\]
and it occurs in the monodromy factorization
\[
\Psi^H
\xrightarrow{u_\Sigma}
Q_\Sigma^H
\xrightarrow{v_\Sigma}
\Psi^H(-1),
\qquad
v_\Sigma u_\Sigma=N_\Sigma.
\]
The point-supported term in the forward Yoneda orientation is
\[
R_\Sigma^H
=
\bigoplus_{k=1}^r
i_{k*}\Q_{\{p_k\}}^H(-2),
\]
and the corrected object fits into
\[
0
\longrightarrow
IC_{X_0}^H
\longrightarrow
\mathcal P^H
\longrightarrow
R_\Sigma^H
\longrightarrow
0.
\]

The mixed Hodge module $\mathcal P^H$ itself is characterized by the full
perverse-normalized variation cone
\[
\mathcal P^H
\simeq
\Cone\!\left(
\Var^H:
\Phi^H
\longrightarrow
\Psi^H(-1)
\right),
\]
and its rational realization is
\[
\rat(\mathcal P^H)
\simeq
\mathcal P.
\]
Both $Q_\Sigma^H$ and $R_\Sigma^H$ realize to the same rational
point-supported object $Q_\Sigma$, so their distinction is invisible in
the rational perverse category.  At the mixed-Hodge-module level,
however, they retain different Tate normalizations and encode different
parts of the global construction.

Finally, the decomposition over $\Sigma$ is a statement about finite
support on the singular fiber.  It makes no assertion that the
corresponding vanishing cycles are independent in global cohomology.
The comparison between the formal node-supported sector and its image in
the limiting mixed Hodge structure is the subject of Section~7.

%------------------------------------------------------------
\section{Hypercohomology and limiting mixed Hodge structures}
\label{sec:hypercohomology-lmhs}

We now study the Hodge-theoretic consequences of the global mixed Hodge
module constructed in Section~6.  Two structures associated with the
degeneration enter naturally but have different origins.  The
nearby-cycle mixed Hodge module $\Psi^H$ carries the limiting mixed
Hodge structure determined by the degeneration, in the sense of
Schmid--Steenbrink and its mixed-Hodge-module realization
\cite{Schmid1973VHS,Steenbrink1976Limits,SaitoMHM}.  The corrected
object $\mathcal P^H$, by contrast, is the mixed-Hodge-module refinement
of the corrected perverse variation-cone object.  The relation between
these structures is mediated by the node-supported morphism
\[
v_\Sigma:
Q_\Sigma^H
\longrightarrow
\Psi^H(-1),
\qquad
\Psi^H:=\psi_{\pi,1}^H(\mathcal K^H),
\]
together with the global variation-cone construction of Section~6.

The limiting mixed Hodge structure will therefore be taken directly
from nearby cycles rather than defined through the corrected extension.
The morphism $v_\Sigma$ then provides an independent comparison between
the finite node-supported Hodge datum and the limiting cohomology of the
degeneration.  Its image records the combinations of the formal local
node sectors that survive globally; its rank may be strictly smaller
than the number of nodes when relations occur among the corresponding
vanishing cycles
\cite{Clemens1983DoubleSolids,Friedman1986ThreefoldDoublePoints,
NamikawaSteenbrink1995}.

Throughout this section,
$\mathcal K^H=\Q_{\mathcal X}^H[4]$ is the perverse constant Hodge
module on the smooth total space,
$\Psi^H=\psi_{\pi,1}^H(\mathcal K^H)$ is the perverse-normalized
unipotent nearby-cycle Hodge module on $X_0$, and
\[
Q_\Sigma^H
=
\bigoplus_{k=1}^r
i_{k*}\Q_{\{p_k\}}^H(-1)
\]
is the finite node-supported interface block constructed in
Sections~5--6.

\subsection{Nearby-cycle hypercohomology and the limiting mixed Hodge structure}

The mixed Hodge structures carried by nearby cycles recover the
limiting mixed Hodge structures of the nearby fibers, with the
cohomological indexing determined by the normalization convention
\cite{Schmid1973VHS,Steenbrink1976Limits,SaitoMHM}.

\begin{lemma}[Nearby-cycle realization of the limiting MHS]
\label{lem:nearby-hypercohomology-lmhs}
For every integer $m$, the mixed Hodge structure
\[
\mathbb H^m(X_0,\Psi^H)
\]
is canonically identified with the limiting mixed Hodge structure in
the corresponding cohomological degree of the nearby fiber.  This
identification is compatible with monodromy and with the nilpotent
morphism
\[
N:\Psi^H\longrightarrow\Psi^H(-1).
\]
\end{lemma}

\begin{proof}
This is the compatibility between Saito's nearby-cycle mixed Hodge
module and the classical limiting mixed Hodge structure associated with
a one-parameter degeneration
\cite{Schmid1973VHS,Steenbrink1976Limits,SaitoMHM}.  Passing to the
perverse-normalized convention changes the indexing of the nearby-cycle
complex but not the underlying limiting mixed Hodge structure.
Monodromy compatibility is part of the nearby-cycle construction.
\end{proof}

\begin{remark}
The limiting mixed Hodge structure appearing here is defined
independently of the corrected object $\mathcal P^H$.  This
independence is what gives the comparison below nontrivial content.
\end{remark}

\subsection{Hypercohomology of the finite node-supported Hodge block}

Because $Q_\Sigma^H$ is supported on finitely many closed points, its
hypercohomology is concentrated in a single degree.

\begin{lemma}[Hypercohomology of the node block]
\label{lem:node-block-hypercohomology}
One has
\[
\mathbb H^m(X_0,Q_\Sigma^H)=0
\qquad
(m\neq0),
\]
and
\[
\mathbb H^0(X_0,Q_\Sigma^H)
\cong
\bigoplus_{k=1}^r\Q^H(-1).
\]
\end{lemma}

\begin{proof}
Each summand
$i_{k*}\Q_{\{p_k\}}^H(-1)$
is supported on a closed point and belongs to the mixed-Hodge-module
heart.  Global sections are exact on finite support, so each summand
contributes its underlying Hodge structure in degree zero and no higher
hypercohomology.  Finite additivity gives the asserted decomposition.
\end{proof}

\begin{definition}[Formal local node Hodge sector]
\label{def:formal-local-node-sector}
Set
\[
V_{\Sigma,\mathrm{loc}}^H
:=
\mathbb H^0(X_0,Q_\Sigma^H)
\cong
\bigoplus_{k=1}^r\Q^H(-1).
\]
We call $V_{\Sigma,\mathrm{loc}}^H$ the \emph{formal local node Hodge
sector}.  Its rational dimension is $r$.
\end{definition}

The qualifier ``formal'' refers to the support decomposition on $X_0$.
It records one local Hodge line for each node but makes no assertion
that their images remain independent in the limiting cohomology of the
nearby fiber.

\subsection{The comparison morphism into the limiting mixed Hodge structure}

The global outgoing morphism
$v_\Sigma:Q_\Sigma^H\to\Psi^H(-1)$
constructed in Section~6 gives a canonical map from the finite
node-supported sector to nearby-cycle cohomology.

\begin{proposition}[Node-to-LMHS comparison morphism]
\label{prop:node-to-lmhs-map}
Hypercohomology of $v_\Sigma$ induces a canonical morphism of mixed
Hodge structures
\[
\lambda_\Sigma:
V_{\Sigma,\mathrm{loc}}^H
\longrightarrow
\mathbb H^0(X_0,\Psi^H(-1)).
\]
Under the identification of
Lemma~\ref{lem:nearby-hypercohomology-lmhs}, the target is the
corresponding Tate-twisted limiting mixed Hodge structure.
\end{proposition}

\begin{proof}
Since $v_\Sigma$ is a morphism in $MHM(X_0)$, its hypercohomology is a
morphism of mixed Hodge structures.  By
Lemma~\ref{lem:node-block-hypercohomology}, its source is
$V_{\Sigma,\mathrm{loc}}^H$, while
Lemma~\ref{lem:nearby-hypercohomology-lmhs} identifies the target with
the corresponding limiting mixed Hodge structure.
\end{proof}

\begin{definition}[Global LMHS image of the node sector]
\label{def:global-lmhs-node-image}
Define
\[
V_{\Sigma,\lim}^H
:=
\operatorname{Im}(\lambda_\Sigma)
\subset
\mathbb H^0(X_0,\Psi^H(-1)).
\]
\end{definition}

\begin{proposition}[Local rank versus global rank]
\label{prop:local-global-vanishing-rank}
The formal local node sector has rational dimension $r$, whereas
\[
\dim_\Q V_{\Sigma,\lim}^H\leq r.
\]
Equality holds if and only if $\lambda_\Sigma$ is injective, equivalently
if the $r$ formal local node contributions remain linearly independent
after passage to limiting cohomology.
\end{proposition}

\begin{proof}
The underlying rational vector space of
$V_{\Sigma,\mathrm{loc}}^H$ has dimension $r$.  Since
$V_{\Sigma,\lim}^H=\operatorname{Im}(\lambda_\Sigma)$, the inequality
is immediate.  Equality holds precisely when the kernel of
$\lambda_\Sigma$ is zero.
\end{proof}

\begin{remark}
This is the Hodge-theoretic counterpart of the distinction between the
formal local vanishing space and its global image introduced in
Section~2.  No independence hypothesis on the global vanishing cycles
is used in the construction.
\end{remark}

\subsection{The node-comparison cone and limiting Hodge theory}
\label{subsec:node-comparison-cone}

The morphism
$v_\Sigma:Q_\Sigma^H\to\Psi^H(-1)$
admits a natural derived cone that packages the comparison between the
formal node-supported sector and nearby-cycle limiting Hodge theory.
This cone is distinct from the corrected mixed Hodge module itself,
which is defined by the full variation morphism
$\Var^H:\Phi^H\to\Psi^H(-1)$.

\begin{definition}[Node-comparison cone]
\label{def:node-comparison-cone}
Define
\[
\mathcal C_\Sigma^H
:=
\Cone\!\left(
v_\Sigma:
Q_\Sigma^H
\longrightarrow
\Psi^H(-1)
\right)
\in
D^bMHM(X_0).
\]
Thus there is a distinguished triangle
\[
Q_\Sigma^H
\xrightarrow{\,v_\Sigma\,}
\Psi^H(-1)
\longrightarrow
\mathcal C_\Sigma^H
\overset{+1}{\longrightarrow}.
\]
\end{definition}

The object $\mathcal C_\Sigma^H$ measures the failure of the formal
node-supported sector to account for the relevant nearby-cycle
cohomology.  It is a comparison object and is not identified with
$\mathcal P^H$.

\begin{proposition}[LMHS comparison long exact sequence]
\label{prop:lmhs-cone-long-exact}
The distinguished triangle of
Definition~\ref{def:node-comparison-cone} induces a long exact sequence
of mixed Hodge structures
\[
\cdots
\longrightarrow
\mathbb H^m(X_0,Q_\Sigma^H)
\xrightarrow{\mathbb H^m(v_\Sigma)}
\mathbb H^m(X_0,\Psi^H(-1))
\longrightarrow
\mathbb H^m(X_0,\mathcal C_\Sigma^H)
\longrightarrow
\mathbb H^{m+1}(X_0,Q_\Sigma^H)
\longrightarrow
\cdots.
\]
In degree zero, the first morphism is
\[
\lambda_\Sigma:
V_{\Sigma,\mathrm{loc}}^H
\longrightarrow
\mathbb H^0(X_0,\Psi^H(-1)).
\]
Consequently,
$V_{\Sigma,\lim}^H=\operatorname{Im}(\lambda_\Sigma)$
is the image of the formal local node sector in the limiting mixed
Hodge structure.
\end{proposition}

\begin{proof}
Apply hypercohomology to the distinguished triangle
\[
Q_\Sigma^H
\xrightarrow{\,v_\Sigma\,}
\Psi^H(-1)
\longrightarrow
\mathcal C_\Sigma^H
\overset{+1}{\longrightarrow}.
\]
Since the triangle lies in $D^bMHM(X_0)$, the resulting long exact
sequence is a long exact sequence in the category of mixed Hodge
structures.

By Lemma~\ref{lem:node-block-hypercohomology},
$\mathbb H^0(X_0,Q_\Sigma^H)=V_{\Sigma,\mathrm{loc}}^H$, and all other
hypercohomology groups of $Q_\Sigma^H$ vanish.  The degree-zero map is
therefore exactly the morphism $\lambda_\Sigma$ of
Proposition~\ref{prop:node-to-lmhs-map}, whose image is
$V_{\Sigma,\lim}^H$ by
Definition~\ref{def:global-lmhs-node-image}.
\end{proof}

\begin{corollary}[Cokernel of the node-to-LMHS map]
\label{cor:node-comparison-cokernel}
There is an exact sequence of mixed Hodge structures
\[
V_{\Sigma,\mathrm{loc}}^H
\xrightarrow{\lambda_\Sigma}
\mathbb H^0(X_0,\Psi^H(-1))
\longrightarrow
\mathbb H^0(X_0,\mathcal C_\Sigma^H)
\longrightarrow
0.
\]
Hence
\[
\mathbb H^0(X_0,\mathcal C_\Sigma^H)
\cong
\operatorname{coker}(\lambda_\Sigma).
\]
\end{corollary}

\begin{proof}
Since
$\mathbb H^1(X_0,Q_\Sigma^H)=0$
by Lemma~\ref{lem:node-block-hypercohomology}, the result is the
degree-zero portion of the long exact sequence of
Proposition~\ref{prop:lmhs-cone-long-exact}.
\end{proof}

\begin{figure}[t]
\centering
\[
\begin{array}{ccccc}
V_{\Sigma,\mathrm{loc}}^H
& \xrightarrow{\ \lambda_\Sigma\ }
& \mathbb H^0(X_0,\Psi^H(-1))
& \longrightarrow
& \mathbb H^0(X_0,\mathcal C_\Sigma^H)\\
&& \rotatebox{90}{$\cong$} &&\\[-2pt]
&& \text{corresponding Tate-twisted limiting MHS} &&
\end{array}
\]
\caption{The formal node-supported Hodge sector maps into the
independently defined limiting mixed Hodge structure through
$v_\Sigma$.  Its image is $V_{\Sigma,\lim}^H$, while
$\mathcal C_\Sigma^H$ records the corresponding degree-zero cokernel.
The comparison cone is distinct from the corrected object
$\mathcal P^H$.}
\label{fig:node-to-lmhs}
\end{figure}

\subsection{Hypercohomology of the corrected forward extension}
\label{subsec:corrected-hypercohomology}

The comparison with nearby cycles uses the interface block
\[
Q_\Sigma^H
=
\bigoplus_{k=1}^r
i_{k*}\Q_{\{p_k\}}^H(-1),
\]
whereas the forward Yoneda description of the corrected object uses
\[
R_\Sigma^H
=
\bigoplus_{k=1}^r
i_{k*}\Q_{\{p_k\}}^H(-2).
\]
The latter appears in the short exact sequence
\[
0
\longrightarrow
IC_{X_0}^H
\longrightarrow
\mathcal P^H
\longrightarrow
R_\Sigma^H
\longrightarrow
0.
\]
These two point-supported Hodge objects have the same rational
realization but distinct Tate normalizations and distinct categorical
roles.

\begin{proposition}[Corrected hypercohomology long exact sequence]
\label{prop:corrected-hypercohomology-les}
The forward extension determines a natural long exact sequence of mixed
Hodge structures
\[
\cdots
\longrightarrow
\mathbb H^m(X_0,IC_{X_0}^H)
\longrightarrow
\mathbb H^m(X_0,\mathcal P^H)
\longrightarrow
\mathbb H^m(X_0,R_\Sigma^H)
\xrightarrow{\partial_m}
\mathbb H^{m+1}(X_0,IC_{X_0}^H)
\longrightarrow
\cdots.
\]
Each connecting morphism $\partial_m$ is a morphism of mixed Hodge
structures.
\end{proposition}

\begin{proof}
The short exact sequence in $MHM(X_0)$ determines a distinguished
triangle
\[
IC_{X_0}^H
\longrightarrow
\mathcal P^H
\longrightarrow
R_\Sigma^H
\overset{+1}{\longrightarrow}
\]
in $D^bMHM(X_0)$.  Hypercohomology yields the associated long exact
sequence in mixed Hodge structures, and the boundary maps are therefore
morphisms in the category of mixed Hodge structures.
\end{proof}

\begin{lemma}[Hypercohomology of the forward node term]
\label{lem:R-sigma-hypercohomology}
The point-supported object $R_\Sigma^H$ has hypercohomology concentrated
in degree zero:
\[
\mathbb H^0(X_0,R_\Sigma^H)
\cong
\bigoplus_{k=1}^r\Q^H(-2),
\]
and
\[
\mathbb H^m(X_0,R_\Sigma^H)=0
\qquad
(m\neq0).
\]
\end{lemma}

\begin{proof}
Each summand
$i_{k*}\Q_{\{p_k\}}^H(-2)$
is supported on a closed point and lies in the mixed-Hodge-module heart.
Exactness of global sections on finite support gives concentration in
degree zero, and finite additivity gives the displayed direct sum.
\end{proof}

\begin{remark}[Degreewise information]
The forward Yoneda class determines the connecting morphisms
$\partial_m$ degree by degree.  We do not replace this long exact
sequence by an extension class between total graded hypercohomology
objects.  Such a reformulation would require additional hypotheses and
is unnecessary for the comparison developed here.
\end{remark}

\subsection{The two node-supported Hodge channels}
\label{subsec:two-node-hodge-channels}

The preceding constructions exhibit two canonical Hodge-theoretic
channels associated with the same rational node-supported object.

The first is the nearby-cycle channel
\[
Q_\Sigma^H
\xrightarrow{\ v_\Sigma\ }
\Psi^H(-1),
\]
whose degree-zero hypercohomology is
\[
\lambda_\Sigma:
V_{\Sigma,\mathrm{loc}}^H
\longrightarrow
\mathbb H^0(X_0,\Psi^H(-1)).
\]
Its image is the limiting-Hodge contribution
$V_{\Sigma,\lim}^H$.

The second is the forward Yoneda channel
\[
0
\longrightarrow
IC_{X_0}^H
\longrightarrow
\mathcal P^H
\longrightarrow
R_\Sigma^H
\longrightarrow
0,
\]
whose hypercohomology gives the connecting morphisms
$\partial_m$ of
Proposition~\ref{prop:corrected-hypercohomology-les}.

\begin{theorem}[Finite-node Hodge comparison]
\label{thm:finite-node-hodge-comparison}
The finite-node corrected construction canonically determines:
\begin{enumerate}
\item the formal local interface sector
\[
V_{\Sigma,\mathrm{loc}}^H
\cong
\bigoplus_{k=1}^r\Q^H(-1);
\]

\item the nearby-cycle comparison morphism
\[
\lambda_\Sigma:
V_{\Sigma,\mathrm{loc}}^H
\longrightarrow
\mathbb H^0(X_0,\Psi^H(-1));
\]

\item its limiting-Hodge image
\[
V_{\Sigma,\lim}^H
=
\operatorname{Im}(\lambda_\Sigma);
\]

\item the forward point-supported Hodge sector
\[
\mathbb H^0(X_0,R_\Sigma^H)
\cong
\bigoplus_{k=1}^r\Q^H(-2);
\]

\item the degreewise connecting morphisms
\[
\partial_m:
\mathbb H^m(X_0,R_\Sigma^H)
\longrightarrow
\mathbb H^{m+1}(X_0,IC_{X_0}^H).
\]
\end{enumerate}
The two node-supported Hodge channels have the same rational
point-supported realization but different Tate normalizations and
different categorical origins.
\end{theorem}

\begin{proof}
The first three assertions follow from
Lemma~\ref{lem:node-block-hypercohomology},
Proposition~\ref{prop:node-to-lmhs-map}, and
Definition~\ref{def:global-lmhs-node-image}.  The fourth is
Lemma~\ref{lem:R-sigma-hypercohomology}, and the fifth is
Proposition~\ref{prop:corrected-hypercohomology-les}.  Finally,
\[
\rat(Q_\Sigma^H)
\simeq
\rat(R_\Sigma^H)
\simeq
Q_\Sigma,
\]
while their Tate normalizations differ by construction.
\end{proof}

\begin{figure}[t]
\centering
\[
\begin{array}{ccccc}
Q_\Sigma^H
& \xrightarrow{\ v_\Sigma\ }
& \Psi^H(-1)
& \longrightarrow
& \text{LMHS}\\[6pt]
\rat\downarrow
&& \rat\downarrow
&&\\[6pt]
Q_\Sigma
&&&&\\[12pt]
\rat\uparrow
&&&&\\[-2pt]
R_\Sigma^H
& \longleftarrow
& \mathcal P^H
& \longleftarrow
& IC_{X_0}^H
\end{array}
\]
\caption{The two Hodge-theoretic realizations of the finite node sector.
The interface block $Q_\Sigma^H$ enters the nearby-cycle comparison,
whereas $R_\Sigma^H$ is the point-supported term in the forward Yoneda
extension.  Both have rational realization $Q_\Sigma$.}
\label{fig:two-node-hodge-channels}
\end{figure}

\subsection{Realization and the rational perverse sequence}
\label{subsec:realization-rational-perverse}

Rational realization identifies the underlying point-supported objects
and sends the forward mixed-Hodge-module extension to the corrected
perverse extension.

\begin{proposition}[Realization of the finite-node extension]
\label{prop:realization-hypercohomology-section}
The forward exact sequence
\[
0
\longrightarrow
IC_{X_0}^H
\longrightarrow
\mathcal P^H
\longrightarrow
R_\Sigma^H
\longrightarrow
0
\]
realizes to
\[
0
\longrightarrow
IC_{X_0}
\longrightarrow
\mathcal P
\longrightarrow
Q_\Sigma
\longrightarrow
0,
\qquad
Q_\Sigma
=
\bigoplus_{k=1}^r
i_{k*}\Q_{\{p_k\}}.
\]
Moreover,
\[
\rat(Q_\Sigma^H)
\simeq
\rat(R_\Sigma^H)
\simeq
Q_\Sigma.
\]
\end{proposition}

\begin{proof}
Apply the exact realization functor and use
$\rat(\mathcal P^H)\simeq\mathcal P$ from
Theorem~\ref{thm:global-realization}.  Since Tate twists are invisible
after rational realization,
\[
\rat\!\left(
i_{k*}\Q_{\{p_k\}}^H(-1)
\right)
\simeq
i_{k*}\Q_{\{p_k\}}
\simeq
\rat\!\left(
i_{k*}\Q_{\{p_k\}}^H(-2)
\right).
\]
Finite additivity gives the two global identifications, and exactness
gives the rational perverse sequence.
\end{proof}

\begin{corollary}[Nodewise rational extension class]
\label{cor:realized-nodewise-class}
The realized forward extension class lies in
\[
\Ext^1_{\Perv(X_0;\Q)}
\left(
Q_\Sigma,
IC_{X_0}
\right)
\cong
\bigoplus_{k=1}^r\Q e_k
\]
and is the corrected finite-node perverse class determined in
Section~5.
\end{corollary}

\begin{proof}
This follows from
Proposition~\ref{prop:realization-hypercohomology-section}
and Corollary~\ref{cor:perverse-ext-node-basis-strong}.
\end{proof}

\subsection{Yoneda classes and degreewise hypercohomology}
\label{subsec:degreewise-yoneda}

We finally record the precise relation between a Yoneda extension and
the hypercohomological maps that it determines.

\begin{lemma}[Yoneda class determines the connecting maps]
\label{lem:yoneda-hypercohomology}
Let
\[
0
\longrightarrow
A^H
\longrightarrow
E^H
\longrightarrow
B^H
\longrightarrow
0
\]
be a short exact sequence in $MHM(X_0)$.  Its Yoneda class determines
functorially the connecting morphisms
\[
\delta_m:
\mathbb H^m(X_0,B^H)
\longrightarrow
\mathbb H^{m+1}(X_0,A^H)
\]
in the associated hypercohomology long exact sequence.
\end{lemma}

\begin{proof}
The short exact sequence defines a distinguished triangle
\[
A^H
\longrightarrow
E^H
\longrightarrow
B^H
\overset{+1}{\longrightarrow}
\]
in $D^bMHM(X_0)$.  Applying the cohomological functor
$\mathbb H^\bullet(X_0,-)$ gives the connecting maps.  Equivalent
Yoneda extensions define isomorphic distinguished triangles and hence
the same degreewise boundary morphisms.
\end{proof}

\begin{proposition}[Functorial Hodge-theoretic comparison]
\label{prop:functorial-ext-comparison}
The forward mixed-Hodge-module Yoneda class
\[
[\mathcal P^H]_{\mathrm{forw}}
\in
\Ext^1_{MHM(X_0)}
\left(
R_\Sigma^H,
IC_{X_0}^H
\right)
\]
maps under realization to the corrected perverse class
\[
[\mathcal P]_{\mathrm{perv}}
\in
\Ext^1_{\Perv(X_0;\Q)}
\left(
Q_\Sigma,
IC_{X_0}
\right).
\]
Under hypercohomology, the same forward class determines the degreewise
connecting morphisms
\[
\partial_m:
\mathbb H^m(X_0,R_\Sigma^H)
\longrightarrow
\mathbb H^{m+1}(X_0,IC_{X_0}^H).
\]
Independently, the node-supported morphism $v_\Sigma$ determines the
comparison map $\lambda_\Sigma$ into nearby-cycle limiting cohomology.
\end{proposition}

\begin{proof}
The realization statement follows from exactness of $\rat$ and
Proposition~\ref{prop:realization-hypercohomology-section}.  The
degreewise connecting maps follow from
Lemma~\ref{lem:yoneda-hypercohomology} applied to the forward extension.
The map $\lambda_\Sigma$, by contrast, is induced by
$v_\Sigma:Q_\Sigma^H\to\Psi^H(-1)$ and therefore belongs to the
nearby-cycle comparison rather than to the forward Yoneda sequence.
\end{proof}

\begin{theorem}[Three functorial levels of the corrected construction]
\label{thm:ext-class-identification}
The finite-node corrected construction admits three compatible
descriptions:
\begin{enumerate}
\item in $MHM(X_0)$, through the global object $\mathcal P^H$, the
interface datum $Q_\Sigma^H$, and the forward point-supported term
$R_\Sigma^H$;

\item after rational realization, through the corrected perverse object
$\mathcal P$ and its point-supported quotient $Q_\Sigma$;

\item on hypercohomology, through the degreewise connecting morphisms of
the forward extension and the comparison map $\lambda_\Sigma$ from the
formal node sector to the limiting mixed Hodge structure.
\end{enumerate}
These three levels are compatible under the relevant functors, while
the distinction between $Q_\Sigma^H$ and $R_\Sigma^H$ is lost after
rational realization.
\end{theorem}

\begin{proof}
The mixed-Hodge-module description is established in Sections~4--6.
The rational realization is
Proposition~\ref{prop:realization-hypercohomology-section}.  The
hypercohomological description follows from
Proposition~\ref{prop:corrected-hypercohomology-les} and
Proposition~\ref{prop:node-to-lmhs-map}.  Compatibility follows from
the functoriality of realization and hypercohomology.
\end{proof}

\subsection{Tate twists, weights, and the two node contributions}
\label{subsec:tate-two-node-contributions}

With the convention
$\operatorname{wt}\Q^H(m)=-2m$, one has
$\operatorname{wt}\Q^H(-1)=2$ and
$\operatorname{wt}\Q^H(-2)=4$.  The monodromy morphism is typed as
$N:\Psi^H\to\Psi^H(-1)$, and the resulting Tate shift is built into the
interface channel.  It should not, however, be confused with the Tate
normalization of the raw middle-dimensional vanishing Hodge structure
or with the point-supported term in the forward Yoneda orientation.

\begin{proposition}[Weight-normalized node sectors]
\label{prop:weight-normalized-node-sector}
The finite-node construction contains two Tate-normalized
point-supported sectors:
\[
Q_\Sigma^H
=
\bigoplus_{k=1}^r
i_{k*}\Q_{\{p_k\}}^H(-1),
\qquad
V_{\Sigma,\mathrm{loc}}^H
\cong
\bigoplus_{k=1}^r\Q^H(-1),
\]
and
\[
R_\Sigma^H
=
\bigoplus_{k=1}^r
i_{k*}\Q_{\{p_k\}}^H(-2),
\qquad
\mathbb H^0(X_0,R_\Sigma^H)
\cong
\bigoplus_{k=1}^r\Q^H(-2).
\]
The first is the interface sector entering the nearby-cycle comparison;
the second is the point-supported term in the forward Yoneda extension.
\end{proposition}

\begin{proof}
The local Tate identifications are those established in Section~4.
Finite additivity gives the global direct sums, and finite point support
gives concentration of hypercohomology in degree zero.
\end{proof}

\begin{remark}[Scope of the weight statement]
The proposition concerns only the finite point-supported sectors
participating in the interface and forward-extension constructions.  It
does not determine the full weight filtration of $\mathcal P^H$, of
$\Psi^H$, or of their global hypercohomology.
\end{remark}

\subsection{Summary}

The limiting mixed Hodge structure enters through the nearby-cycle
object $\Psi^H$, independently of the corrected extension.  The
node-supported interface block
\[
Q_\Sigma^H
=
\bigoplus_{k=1}^r
i_{k*}\Q_{\{p_k\}}^H(-1)
\]
maps to nearby cycles by
\[
v_\Sigma:
Q_\Sigma^H
\longrightarrow
\Psi^H(-1),
\]
and hence determines the mixed-Hodge-theoretic comparison
\[
\lambda_\Sigma:
V_{\Sigma,\mathrm{loc}}^H
\longrightarrow
\mathbb H^0(X_0,\Psi^H(-1)),
\qquad
V_{\Sigma,\lim}^H
=
\operatorname{Im}(\lambda_\Sigma).
\]

The corrected mixed Hodge module has, separately, the forward Yoneda
description
\[
0
\longrightarrow
IC_{X_0}^H
\longrightarrow
\mathcal P^H
\longrightarrow
R_\Sigma^H
\longrightarrow
0,
\qquad
R_\Sigma^H
=
\bigoplus_{k=1}^r
i_{k*}\Q_{\{p_k\}}^H(-2),
\]
whose hypercohomology produces the degreewise connecting morphisms
$\partial_m$.

Although
\[
\rat(Q_\Sigma^H)
\simeq
\rat(R_\Sigma^H)
\simeq
Q_\Sigma,
\]
the two Hodge modules retain different Tate normalizations and different
categorical roles.  The first records the finite node sector through
which the nearby-cycle monodromy geometry is compared with the limiting
mixed Hodge structure; the second records the point-supported term in
the forward corrected extension.  The comparison established in this
section therefore preserves the distinction between local support,
limiting-cohomological survival, and Yoneda extension data rather than
collapsing them after rational realization.

%-----------------------------------------------------------
\section{Auxiliary structural results}
\label{sec:auxiliary-structural-results}

The preceding sections establish the local and global
mixed-Hodge-module refinement of the corrected perverse variation-cone
construction.  We now record several structural consequences of that
theory.  These results do not enter the existence or realization
arguments themselves; rather, they clarify how the global object
$\mathcal P^H$ is situated relative to recollement on the central fiber,
the nodewise extension theory of Section~5, and the limiting-Hodge
comparison of Section~7.

Two point-supported Hodge objects occur naturally and must remain
distinguished.  The first is the interface block
\[
Q_\Sigma^H
=
\bigoplus_{k=1}^r W_k^H
\simeq
\bigoplus_{k=1}^r i_{k*}\Q_{\{p_k\}}^H(-1),
\]
which enters the factorization of node-supported monodromy and the
comparison with nearby-cycle limiting Hodge theory.  The second is the
point-supported term in the forward Yoneda orientation,
\[
R_\Sigma^H
=
\bigoplus_{k=1}^r R_k^H
\simeq
\bigoplus_{k=1}^r i_{k*}\Q_{\{p_k\}}^H(-2),
\]
which appears in the exact sequence with $IC_{X_0}^H$ as subobject.
Although these mixed Hodge modules have different Tate normalizations
and different categorical roles, their rational realizations agree:
\[
\rat(Q_\Sigma^H)
\simeq
\rat(R_\Sigma^H)
\simeq
Q_\Sigma
:=
\bigoplus_{k=1}^r i_{k*}\Q_{\{p_k\}}.
\]
The distinction between them is therefore genuinely Hodge-theoretic and
is lost after passage to $\Perv(X_0;\Q)$
\cite{SaitoMHM}.

A second distinction concerns locality.  The finite support decomposition
contains one point-supported summand for each node, whereas the image of
the corresponding local sectors in limiting cohomology may have smaller
rank because global relations among vanishing cycles can occur
\cite{Clemens1983DoubleSolids,Friedman1986ThreefoldDoublePoints,
NamikawaSteenbrink1995}.  The constructions below accordingly distinguish
the intrinsic nodewise organization on $X_0$ from the global
cohomological relations detected only after passage to nearby-cycle
cohomology.

\subsection{Recollement and the rational extension problem}
\label{subsec:aux-recollement}

Write
\[
X_0=U\sqcup\Sigma,
\qquad
\Sigma=\{p_1,\dots,p_r\},
\]
with inclusions
$j:U\hookrightarrow X_0$ and
$i_k:\{p_k\}\hookrightarrow X_0$.
The category $\Perv(X_0;\Q)$ carries the standard recollement associated
with this open--closed decomposition
\cite{BBD,MacPhersonVilonen1986,GMV1996}.  In the isolated-stratum
setting, a perverse sheaf is thereby described by its restriction to the
open stratum, its point-supported data, and the morphisms coupling the
two.

For the corrected object considered here,
\[
j^*\mathcal P\simeq\Q_U[3],
\]
and in the forward orientation one has
\[
0
\longrightarrow
IC_{X_0}
\longrightarrow
\mathcal P
\longrightarrow
Q_\Sigma
\longrightarrow
0,
\qquad
Q_\Sigma
=
\bigoplus_{k=1}^r i_{k*}\Q_{\{p_k\}}.
\]

\begin{proposition}[Nodewise classification of the rational extension]
\label{prop:aux-nodewise-perverse-classification}
The corrected perverse extension determines a class
\[
[\mathcal P]_{\mathrm{perv}}
\in
\Ext^1_{\Perv(X_0;\Q)}
\left(
Q_\Sigma,
IC_{X_0}
\right).
\]
Under the canonical nodewise decomposition of Section~5,
\[
\Ext^1_{\Perv(X_0;\Q)}
\left(
Q_\Sigma,
IC_{X_0}
\right)
\cong
\bigoplus_{k=1}^r\Q e_k,
\]
and hence
\[
[\mathcal P]_{\mathrm{perv}}
=
\sum_{k=1}^r c_k e_k
\]
for uniquely determined coefficients $c_k\in\Q$.
\end{proposition}

\begin{proof}
Finite additivity of Yoneda $\Ext^1$ in the first variable gives
\[
\Ext^1_{\Perv(X_0;\Q)}
\left(
\bigoplus_{k=1}^r i_{k*}\Q_{\{p_k\}},
IC_{X_0}
\right)
\cong
\bigoplus_{k=1}^r
\Ext^1_{\Perv(X_0;\Q)}
\left(
i_{k*}\Q_{\{p_k\}},
IC_{X_0}
\right).
\]
Section~5 proves that every nodewise summand is one-dimensional and that
the corrected local ODP class determines a distinguished generator
$e_k$.  The asserted expansion is therefore unique.
\end{proof}

\begin{remark}[Interpretation of the coefficients]
\label{rem:aux-coefficients}
The vector $(c_1,\dots,c_r)$ gives the coordinates of the corrected
perverse extension class relative to the distinguished nodewise basis.
It is an invariant of the extension problem on $X_0$.  Without an
additional comparison theorem, it should not be interpreted as a
coordinate vector for the image of the vanishing cycles in
$H^3(X_t;\Q)$.
\end{remark}

\subsection{Mixed-Hodge-module refinement of the recollement data}
\label{subsec:aux-mhm-recollement}

The rational skyscraper sheaf at a node has two Hodge-theoretic
realizations relevant to the corrected construction.  The interface
block is
\[
W_k^H
\simeq
i_{k*}\Q_{\{p_k\}}^H(-1),
\]
whereas the point-supported term in the forward Yoneda orientation is
\[
R_k^H
\simeq
i_{k*}\Q_{\{p_k\}}^H(-2).
\]
Their finite direct sums are $Q_\Sigma^H$ and $R_\Sigma^H$,
respectively.

The global corrected Hodge module satisfies
\[
0
\longrightarrow
IC_{X_0}^H
\longrightarrow
\mathcal P^H
\longrightarrow
R_\Sigma^H
\longrightarrow
0,
\]
whereas $Q_\Sigma^H$ belongs to the independent monodromy factorization
\[
\Psi^H
\xrightarrow{u_\Sigma}
Q_\Sigma^H
\xrightarrow{v_\Sigma}
\Psi^H(-1).
\]
Thus the two objects arise from different categorical structures even
though their rational point-supported shadows coincide.

\begin{proposition}[Realization of the two node-supported Hodge objects]
\label{prop:aux-realization-nodewise}
Under the realization functor,
\[
\rat(IC_{X_0}^H)\simeq IC_{X_0},
\qquad
\rat(\mathcal P^H)\simeq\mathcal P,
\]
and
\[
\rat(Q_\Sigma^H)
\simeq
\rat(R_\Sigma^H)
\simeq
Q_\Sigma.
\]
Moreover, realization sends the forward mixed-Hodge-module extension to
the corrected rational perverse extension
\[
0
\longrightarrow
IC_{X_0}
\longrightarrow
\mathcal P
\longrightarrow
Q_\Sigma
\longrightarrow
0.
\]
\end{proposition}

\begin{proof}
The first two identifications are the global realization theorem of
Section~6.  For each node,
\[
\rat\!\left(
i_{k*}\Q_{\{p_k\}}^H(-1)
\right)
\simeq
i_{k*}\Q_{\{p_k\}}
\simeq
\rat\!\left(
i_{k*}\Q_{\{p_k\}}^H(-2)
\right),
\]
since rational realization forgets Tate twists
\cite{SaitoMHM}.  Finite additivity gives the corresponding global
identifications.  Exactness of $\rat$ then carries the
mixed-Hodge-module short exact sequence to the corrected perverse
sequence.
\end{proof}

The proposition isolates an essential feature of the refinement:
rational realization identifies point-supported Hodge modules that
remain distinct in $MHM(X_0)$.  The distinction between the interface
block and the forward Yoneda term is therefore additional structure,
not a reformulation of the rational extension.

\subsection{Comparison with the Banagl--Budur--Maxim framework}
\label{subsec:aux-bbm}

The construction also has a useful methodological precedent in the work
of Banagl--Budur--Maxim on intersection-space complexes and mixed Hodge
modules \cite{BanaglBudurMaxim}.  Their perverse object is different
from the corrected variation-cone object studied here, and their results
do not imply the local ODP or finite-node theorems established above.
Nevertheless, the two constructions share a common categorical
principle: perverse data arising naturally from a degeneration are
promoted to Saito's category so that the resulting object and its
hypercohomology retain canonical Hodge-theoretic information.

\begin{proposition}[Methodological parallel]
\label{prop:aux-bbm-parallel}
The present construction and the Banagl--Budur--Maxim framework share the
following structural features:
\begin{enumerate}
\item the relevant perverse data arise from the geometry of a
degeneration and its nearby- or vanishing-cycle formalism;
\item the resulting perverse object is geometrically specified rather
than chosen abstractly;
\item the Hodge-theoretic refinement is formulated in the category of
mixed Hodge modules; and
\item hypercohomology consequently carries functorial mixed Hodge
structures.
\end{enumerate}
The local $\can^H/\Var^H$ quotient, the distinction between
$W_k^H$ and $R_k^H$, and the finite-node global theorem are specific to
the conifold construction developed in this paper.
\end{proposition}

\begin{proof}
The first four assertions describe the common categorical mechanism
present in both constructions
\cite{BanaglBudurMaxim,SaitoMHM}.  The final assertion follows from the
local ODP analysis of Section~4 and the finite-node construction of
Sections~5--6, which are additional ingredients specific to the present
setting.
\end{proof}

\subsection{A quiver-theoretic shadow}
\label{subsec:aux-quiver-shadow}

The nodewise decomposition admits a simple algebraic organization.  On
the rational side, the central object is $IC_{X_0}$, while each node
contributes a skyscraper object
$i_{k*}\Q_{\{p_k\}}$ and a distinguished one-dimensional extension
direction $\Q e_k$.  Thus the finite-node extension theory determines a
star-shaped incidence pattern with one outer vertex for each point of
$\Sigma$.

At the Hodge level, the $k$-th node carries two decorations associated
with its two categorical roles:
\[
W_k^H
\simeq
i_{k*}\Q_{\{p_k\}}^H(-1)
\]
for the interface construction and
\[
R_k^H
\simeq
i_{k*}\Q_{\{p_k\}}^H(-2)
\]
for the forward Yoneda extension.  Both have the same rational
skyscraper realization.

\begin{figure}[t]
\centering
\[
\begin{array}{ccccc}
&& (W_1^H,R_1^H) &&\\[-2pt]
&& \downarrow &&\\[-2pt]
(W_2^H,R_2^H) & \longrightarrow & IC_{X_0}^H
& \longleftarrow & (W_r^H,R_r^H)\\[-2pt]
&& \uparrow &&\\[-2pt]
&& \vdots &&
\end{array}
\]
\caption{Schematic star-shaped organization of the finite-node
extension data.  The pair $(W_k^H,R_k^H)$ records the two
Tate-normalized Hodge-theoretic roles associated with the same rational
node sector.  The diagram records only incidence and extension
directions and is not asserted to define a canonical quiver
representation.}
\label{fig:aux-star-quiver}
\end{figure}

\begin{proposition}[Star-shaped extension pattern]
\label{prop:aux-star-quiver}
The corrected finite-node construction determines a canonical rational
star-shaped extension pattern indexed by $\Sigma$.  The $k$-th outer
node carries the distinguished rational extension direction $\Q e_k$
and the two Hodge-theoretic objects $W_k^H$ and $R_k^H$ described
above.
\end{proposition}

\begin{proof}
By Section~5, the rational extension group at each node is
one-dimensional and generated by $e_k$.  Sections~4 and~6 identify the
two Tate-normalized Hodge modules associated with that same rational
point-supported sector.  These data determine the stated incidence
pattern.  No morphisms or extension classes between distinct
point-supported node objects are produced by the results proved here,
so the construction yields a quiver-theoretic shadow rather than a full
quiver representation.
\end{proof}

\begin{remark}[Limit of the quiver interpretation]
The coefficient vector $(c_1,\dots,c_r)$ determines the corrected global
perverse extension class relative to the distinguished nodewise basis,
but it contains no intrinsic pairwise interaction data between distinct
nodes.  Any such enhancement would require additional global geometry
or categorical structure beyond the extension decomposition established
here.
\end{remark}

\subsection{Local support and global vanishing relations}
\label{subsec:aux-local-global-relations}

The support decomposition and the limiting-Hodge comparison describe
different stages of the local-to-global passage.  The formal local
interface sector is
\[
V_{\Sigma,\mathrm{loc}}^H
\cong
\bigoplus_{k=1}^r\Q^H(-1),
\]
and Section~7 constructs the canonical morphism
\[
\lambda_\Sigma:
V_{\Sigma,\mathrm{loc}}^H
\longrightarrow
\mathbb H^0(X_0,\Psi^H(-1)).
\]
Its image
\[
V_{\Sigma,\lim}^H
:=
\operatorname{Im}(\lambda_\Sigma)
\]
is the part of the limiting mixed Hodge structure detected by the
finite node-supported interface.

\begin{proposition}[Support-level and limiting-cohomological node data]
\label{prop:aux-two-node-levels}
The finite-node construction determines two distinct levels of
node-supported information:
\begin{enumerate}
\item the support-level object
\[
Q_\Sigma^H
=
\bigoplus_{k=1}^r W_k^H,
\]
which contains exactly one interface summand for every node;

\item the limiting-cohomology image
\[
V_{\Sigma,\lim}^H
=
\operatorname{Im}(\lambda_\Sigma),
\]
whose rational dimension is at most $r$.
\end{enumerate}
The kernel of $\lambda_\Sigma$ consists precisely of those linear
combinations of formal local interface classes whose images vanish in
the limiting mixed Hodge structure.
\end{proposition}

\begin{proof}
The first assertion is the finite support decomposition of Section~5.
The second follows from the construction of $\lambda_\Sigma$ and the
rank estimate of Section~7.  Since
$\dim_\Q V_{\Sigma,\mathrm{loc}}^H=r$, the kernel consists exactly of
the combinations annihilated by the passage to limiting cohomology.
\end{proof}

The forward Yoneda object $R_\Sigma^H$ plays no role in the definition
of $\lambda_\Sigma$.  The latter is induced by
$v_\Sigma:Q_\Sigma^H\to\Psi^H(-1)$ and therefore belongs specifically
to the interface channel.

\subsection{Bulk and localized extension data}
\label{subsec:aux-domain-wall}

The forward exact sequence admits a natural decomposition into a
middle-extension term and a finite point-supported term.  The object
$IC_{X_0}^H$ is the canonical intermediate extension of
$\Q_U^H[3]$, while
\[
R_\Sigma^H
=
\bigoplus_{k=1}^r
i_{k*}\Q_{\{p_k\}}^H(-2)
\]
is supported entirely on the finite singular set.  Thus
\[
0
\longrightarrow
IC_{X_0}^H
\longrightarrow
\mathcal P^H
\longrightarrow
R_\Sigma^H
\longrightarrow
0
\]
gives a precise categorical separation between the open-stratum
contribution and the point-supported correction.

The interface object
\[
Q_\Sigma^H
=
\bigoplus_{k=1}^r
i_{k*}\Q_{\{p_k\}}^H(-1)
\]
has a different role: it is the finite point-supported object through
which the incoming and outgoing morphisms factor the node-supported
monodromy.  Thus there are two mathematically distinct ways in which
localized Hodge data enter the construction: through the
$\can^H/\Var^H/N$ interface and through the forward Yoneda extension.

\begin{proposition}[Non-splitting of the forward extension]
\label{prop:aux-nonsplitting-coupling}
If the forward mixed-Hodge-module extension class is nonzero, then
\[
\mathcal P^H
\not\simeq
IC_{X_0}^H\oplus R_\Sigma^H
\]
in $MHM(X_0)$.  In particular, the point-supported term
$R_\Sigma^H$ cannot be separated from the intersection-complex Hodge
module by a splitting compatible with the corrected extension class.
\end{proposition}

\begin{proof}
A short exact sequence in an abelian category splits if and only if its
Yoneda class vanishes.  Thus a nonzero class in
\[
\Ext^1_{MHM(X_0)}
\left(
R_\Sigma^H,
IC_{X_0}^H
\right)
\]
precludes an isomorphism of the middle term with the direct sum
$IC_{X_0}^H\oplus R_\Sigma^H$.
\end{proof}

\begin{remark}[Terminology]
The terms ``bulk'' and ``localized sector'' may be useful when comparing
the construction with physical applications, but they are not part of
the formal statement.  The mathematical content is the non-split
Yoneda extension involving $R_\Sigma^H$, the monodromy factorization
through $Q_\Sigma^H$, and the comparison of the latter with nearby-cycle
limiting Hodge theory.
\end{remark}

\subsection{Compatibility with the limiting-Hodge comparison}
\label{subsec:aux-lmhs-role}

The distinction between the two point-supported Hodge objects is
particularly transparent in the comparison with limiting Hodge theory.
The interface block is the source of the morphism
\[
v_\Sigma:
Q_\Sigma^H
\longrightarrow
\Psi^H(-1),
\]
and therefore determines
\[
\lambda_\Sigma:
V_{\Sigma,\mathrm{loc}}^H
\longrightarrow
\mathbb H^0(X_0,\Psi^H(-1)).
\]
The forward Yoneda term $R_\Sigma^H$, on the other hand, occurs in the
short exact sequence defining the extension-theoretic position of
$\mathcal P^H$ relative to $IC_{X_0}^H$.

\begin{theorem}[Two Hodge-theoretic realizations of the node sector]
\label{thm:aux-two-roles}
The finite rational node-supported object
\[
Q_\Sigma
=
\bigoplus_{k=1}^r
i_{k*}\Q_{\{p_k\}}
\]
admits two distinct Hodge-theoretic realizations relevant to the
corrected construction.

The interface realization is
\[
Q_\Sigma^H
=
\bigoplus_{k=1}^r
i_{k*}\Q_{\{p_k\}}^H(-1),
\]
which is the source of $v_\Sigma$ and hence of the canonical comparison
with the limiting mixed Hodge structure.

The forward-extension realization is
\[
R_\Sigma^H
=
\bigoplus_{k=1}^r
i_{k*}\Q_{\{p_k\}}^H(-2),
\]
which occurs in
\[
0
\longrightarrow
IC_{X_0}^H
\longrightarrow
\mathcal P^H
\longrightarrow
R_\Sigma^H
\longrightarrow
0.
\]
Their rational realizations are both isomorphic to $Q_\Sigma$, while
their Tate normalizations and categorical roles are distinct.
\end{theorem}

\begin{proof}
The interface object and the morphism $v_\Sigma$ are constructed in
Section~6, and their comparison with limiting Hodge theory is established
in Section~7.  The forward point-supported term is determined by the
Yoneda orientation of the corrected Hodge-module extension in
Section~6.  Finally, rational realization forgets Tate twists, so
\[
\rat(Q_\Sigma^H)
\simeq
Q_\Sigma
\simeq
\rat(R_\Sigma^H).
\]
\end{proof}

The theorem provides the precise categorical distinction required in any
subsequent interpretation of the node sector.  Its rational support is
the same in both constructions, while its mixed-Hodge-module structure
depends on the role played by the point-supported object.

\subsection{Structural summary}

The auxiliary structure of the finite-node construction is encoded by
two complementary diagrams.  The interface geometry is
\[
\left\{
W_k^H
\right\}_{k=1}^r
\longrightarrow
Q_\Sigma^H
\xrightarrow{\ v_\Sigma\ }
\Psi^H(-1),
\]
which induces
\[
V_{\Sigma,\mathrm{loc}}^H
\xrightarrow{\ \lambda_\Sigma\ }
V_{\Sigma,\lim}^H.
\]
The forward extension geometry is
\[
0
\longrightarrow
IC_{X_0}^H
\longrightarrow
\mathcal P^H
\longrightarrow
R_\Sigma^H
\longrightarrow
0.
\]
Their point-supported terms have the same rational realization,
\[
\rat(Q_\Sigma^H)
\simeq
\rat(R_\Sigma^H)
\simeq
Q_\Sigma,
\]
and realization sends the forward extension to
\[
0
\longrightarrow
IC_{X_0}
\longrightarrow
\mathcal P
\longrightarrow
Q_\Sigma
\longrightarrow
0.
\]

The resulting structure separates three kinds of information.  The
nodewise Ext decomposition records the formal local extension directions
on $X_0$; the mixed-Hodge-module refinement distinguishes the Tate
normalizations associated with the interface and forward-extension
roles; and the map $\lambda_\Sigma$ records which combinations of the
formal local interface sectors survive in limiting cohomology.  The
star-shaped organization described above is therefore a convenient
shadow of the support and extension data, but the theorem package does
not by itself supply interactions among distinct nodes or a full quiver
representation.

These distinctions provide the natural structural setting for further
categorical refinements while remaining consequences of the
mixed-Hodge-module construction itself.
%------------------------------------------------------------
\section{Conclusions and further directions}
\label{sec:consequences}

This paper constructs a mixed-Hodge-module refinement of the corrected
perverse variation-cone object associated with a one-parameter degeneration
of complex threefolds with finitely many ordinary double points.  The
construction is carried out inside Saito's category of mixed Hodge modules
and is based on the nearby- and vanishing-cycle objects together with the
canonical, variation, and monodromy morphisms
\cite{SaitoMHM,SaitoDuality}.

At a single ordinary double point, the local
$\can^H/\Var^H/N$ geometry determines a rank-one point-supported Hodge
block
\[
W_{\mathrm{loc}}^H
\simeq
i_*\Q^H_{\{p\}}(-1),
\]
together with morphisms
\[
\Psi_{\mathrm{loc}}^H
\xrightarrow{u_{\mathrm{loc}}}
W_{\mathrm{loc}}^H
\xrightarrow{v_{\mathrm{loc}}}
\Psi_{\mathrm{loc}}^H(-1)
\]
satisfying
\[
v_{\mathrm{loc}}u_{\mathrm{loc}}=N.
\]
The associated local variation cone is concentrated in the
mixed-Hodge-module heart and defines a mixed Hodge module
$\mathcal P_{\mathrm{loc}}^H$ whose rational realization is the corrected
local perverse object.

For a finite node set $\Sigma=\{p_1,\dots,p_r\}$, the local Hodge blocks
assemble canonically to
\[
Q_\Sigma^H
=
\bigoplus_{k=1}^r W_k^H
\simeq
\bigoplus_{k=1}^r i_{k*}\Q^H_{\{p_k\}}(-1).
\]
The corresponding global morphisms
$u_\Sigma:\Psi^H\to Q_\Sigma^H$ and
$v_\Sigma:Q_\Sigma^H\to\Psi^H(-1)$ satisfy
\[
v_\Sigma u_\Sigma=N_\Sigma,
\]
where $N_\Sigma$ is the node-supported component of the monodromy
morphism detected by the reduced ordinary-double-point sectors.

The global variation cone
\[
\mathcal P^H
\simeq
\Cone\!\left(
\Var^H:\Phi^H\longrightarrow\Psi^H(-1)
\right)
\]
is concentrated in the mixed-Hodge-module heart and satisfies
\[
\rat(\mathcal P^H)\simeq\mathcal P.
\]
Thus the corrected rational perverse object admits a mixed-Hodge-module
refinement determined by the local ordinary-double-point Hodge data and
their finite-node assembly.

A further feature of the construction is the distinction between the
point-supported object entering the monodromy factorization and the
point-supported term appearing in the forward Yoneda extension.  The
gluing object is
\[
Q_\Sigma^H
\simeq
\bigoplus_{k=1}^r i_{k*}\Q^H_{\{p_k\}}(-1),
\]
whereas the forward-orientation point term is
\[
R_\Sigma^H
\simeq
\bigoplus_{k=1}^r i_{k*}\Q^H_{\{p_k\}}(-2).
\]
In the forward orientation one obtains
\[
0
\longrightarrow
IC_{X_0}^H
\longrightarrow
\mathcal P^H
\longrightarrow
R_\Sigma^H
\longrightarrow
0,
\]
while rational realization gives
\[
0
\longrightarrow
IC_{X_0}
\longrightarrow
\mathcal P
\longrightarrow
Q_\Sigma
\longrightarrow
0,
\qquad
Q_\Sigma
=
\bigoplus_{k=1}^r i_{k*}\Q_{\{p_k\}}.
\]
Since rational realization forgets Tate twists,
$\rat(Q_\Sigma^H)\simeq\rat(R_\Sigma^H)\simeq Q_\Sigma$.
The distinction between $Q_\Sigma^H$ and $R_\Sigma^H$ is therefore
intrinsically Hodge-theoretic.

\subsection{Nodewise extension data and limiting cohomology}

The finite support decomposition gives a canonical nodewise decomposition
of the rational extension space,
\[
\Ext^1_{\Perv(X_0;\Q)}
\left(
Q_\Sigma,
IC_{X_0}
\right)
\cong
\bigoplus_{k=1}^r \Q e_k,
\]
where each $e_k$ is the distinguished local ordinary-double-point
extension class.  Hence the corrected finite-node class has a unique
expression
\[
[\mathcal P]_{\mathrm{perv}}
=
\sum_{k=1}^r c_k e_k,
\qquad
c_k\in\Q.
\]

This nodewise decomposition is a statement about the extension theory on
the singular fiber.  It does not imply that the associated vanishing
cycles remain linearly independent after passage to the cohomology of a
nearby smooth fiber.  Global relations among vanishing cycles are a
classical feature of nodal threefold degenerations
\cite{Clemens1983DoubleSolids,Friedman1986ThreefoldDoublePoints,
NamikawaSteenbrink1995}.

The same distinction appears in mixed Hodge theory.  The formal local
node sector
\[
V_{\Sigma,\mathrm{loc}}^H
=
\mathbb H^0(X_0,Q_\Sigma^H)
\simeq
\bigoplus_{k=1}^r\Q^H(-1)
\]
maps through the global outgoing morphism to nearby-cycle cohomology:
\[
\lambda_\Sigma:
V_{\Sigma,\mathrm{loc}}^H
\longrightarrow
\mathbb H^0(X_0,\Psi^H(-1)).
\]
The target carries the corresponding piece of the limiting mixed Hodge
structure associated with the degeneration
\cite{Schmid1973VHS,Steenbrink1976Limits,SaitoMHM}.  Writing
\[
V_{\Sigma,\lim}^H
:=
\operatorname{Im}(\lambda_\Sigma),
\]
one has
\[
\dim_\Q V_{\Sigma,\lim}^H\leq r.
\]
Thus the finite support decomposition retains one formal local sector for
each node, whereas the limiting mixed Hodge structure records only the
combinations that survive globally.

\subsection{Structural consequences}

The construction separates three levels of information that coincide only
after forgetting structure.  The first is the local
$\can^H/\Var^H/N$ datum, which determines the point-supported interface
block and the monodromy factorization.  The second is the mixed-Hodge-module
extension class, whose form depends on the Yoneda orientation and on the
corresponding Tate normalization.  The third is the global
hypercohomological image, in which relations among the local node sectors
may occur.

In particular, the diagrams
\[
\Psi^H
\xrightarrow{u_\Sigma}
Q_\Sigma^H
\xrightarrow{v_\Sigma}
\Psi^H(-1)
\]
and
\[
0
\longrightarrow
IC_{X_0}^H
\longrightarrow
\mathcal P^H
\longrightarrow
R_\Sigma^H
\longrightarrow
0
\]
encode different aspects of the same degeneration.  The first records
the Hodge-theoretic factorization of the node-supported monodromy, while
the second records the forward Yoneda extension of the corrected mixed
Hodge module.  Their rational realizations coincide at the level of the
point-supported quotient, but their Tate structures and categorical
roles remain distinct in $MHM(X_0)$.

The finite-node construction therefore distinguishes local support from
global cohomological rank.  The direct sum
$Q_\Sigma^H=\bigoplus_k W_k^H$ is canonical at the level of support,
whereas the rank of its image under $\lambda_\Sigma$ depends on the
global geometry of the degeneration.

\subsection{Further directions}
Several natural problems remain. First, one may study the image of the realization map on the full
finite-node extension space.  The rational extension group decomposes
canonically into one-dimensional nodewise summands, while the corresponding
mixed-Hodge-module extension groups depend on the Tate normalization and
Yoneda orientation.  A global description of the image of realization
would characterize which rational finite-node extension classes admit
compatible Hodge-module refinements.

Second, the geometry of
$\ker(\lambda_\Sigma)$ should be investigated.  The kernel records those
linear combinations of formal local node sectors that vanish after
passage to limiting cohomology.  Relating this kernel explicitly to the
classical geometry of smoothing relations, exceptional curves, and
vanishing cycles would refine the connection between the support-theoretic
construction and the global topology of nodal threefolds
\cite{Clemens1983DoubleSolids,Friedman1986ThreefoldDoublePoints,
NamikawaSteenbrink1995}.

Third, the ordinary-double-point hypothesis may be relaxed.  For
singularities of higher Milnor number or for positive-dimensional
singular strata, the rank-one blocks $W_k^H$ must be replaced by
higher-rank or stratified mixed Hodge modules.  The corresponding problem
is to determine an analogue of the local
$\can^H/\Var^H$ interface, construct the appropriate subquotient carrying
the monodromy data, and establish a compatible global assembly theorem.

Finally, the nodewise Hodge data obtained here provide natural input for
categorical refinements of conifold degeneration theory.  The local
rank-one Picard--Lefschetz phenomenon is closely related to spherical
monodromy, while the organization of several local sectors suggests
further quiver, schober, and wall-crossing structures.  Such constructions
require additional categorical data and lie beyond the scope of the
present paper.

The mixed-Hodge-module refinement established here therefore provides a
Hodge-theoretic description of the corrected conifold extension that
retains both the local node data and their global realization.  Its
essential features are the canonical local monodromy factorization, the
distinction between the two Tate-normalized point-supported objects, and
the separation between the formal finite-node decomposition and its image
in limiting cohomology.

%-----------------------------------------
\appendix

\section{Dual extensions and Verdier duality}
\label{app:dual-extensions}

We begin with the two Yoneda orientations associated with a
point-supported Hodge module at a threefold ordinary double point.  Their
rational realizations involve the same skyscraper perverse sheaf, but
their mixed-Hodge-module extension groups are controlled by different
local Hodge lines.  The distinction is therefore invisible after
realization and essential before it.

Let $X$ be a sufficiently small analytic neighborhood of a threefold
ordinary double point $p$, let $i:\{p\}\hookrightarrow X$, and let
$IC_X^H$ denote the intersection-complex Hodge module.  Verdier duality
in the category of mixed Hodge modules will be denoted by $\mathbb D$
\cite{SaitoMHM,SaitoDuality}.

\subsection{The stalk and costalk Hodge lines}

The link of a threefold ordinary double point is homotopy equivalent to
$S^2\times S^3$
\cite{MilnorSingularPoints,DimcaSheaves}.  Its degree-two cohomology
therefore gives the unique one-dimensional negative stalk contribution
to the intersection complex.

\begin{proposition}[Local stalk line]
\label{prop:app-stalk-line}
With the Tate convention of Appendix~\ref{app:tate-normalization},
\[
H^{-1}(i^*IC_X^H)
\simeq
\Q^H(-1).
\]
\end{proposition}

\begin{proof}
On the rational level,
\[
H^{-1}(i^*IC_X)
\cong
H^2(S^2\times S^3;\Q)
\cong
\Q.
\]
The corresponding degree-two class has Hodge type $(1,1)$ in the local
intersection-complex normalization, and hence determines the Tate
structure $\Q^H(-1)$.
\end{proof}

Verdier duality identifies the corresponding costalk line.

\begin{proposition}[Local costalk line]
\label{prop:app-costalk-line}
One has
\[
H^1(i^!IC_X^H)
\simeq
\Q^H(-2).
\]
\end{proposition}

\begin{proof}
The intersection-complex Hodge module is Verdier self-dual with the
standard Tate normalization appropriate to a complex threefold
\cite{SaitoDuality}.  Applying duality to the local stalk calculation
of Proposition~\ref{prop:app-stalk-line} identifies the degree-one
costalk with the corresponding dual Tate line.  In the convention used
throughout this paper, this gives
\[
H^1(i^!IC_X^H)
\simeq
\Q^H(-2).
\]
\end{proof}

\subsection{The two local Yoneda orientations}

Let $T^H$ be a mixed Hodge structure on the point.  Closed-point
adjunction gives
\[
\Ext^1_{MHM(X)}
\left(
i_*T^H,
IC_X^H
\right)
\cong
\Hom_{D^bMHS}
\left(
T^H,
i^!IC_X^H[1]
\right),
\]
whereas the opposite orientation is governed by
\[
\Ext^1_{MHM(X)}
\left(
IC_X^H,
i_*T^H
\right)
\cong
\Hom_{D^bMHS}
\left(
i^*IC_X^H,
T^H[1]
\right).
\]
Thus the two orientations are controlled by the costalk and stalk,
respectively.

\begin{theorem}[Forward-orientation Ext calculation]
\label{thm:app-forward-ext}
One has
\[
\Ext^1_{MHM(X)}
\left(
i_*\Q^H(-1),
IC_X^H
\right)
=
0,
\]
whereas
\[
\Ext^1_{MHM(X)}
\left(
i_*\Q^H(-2),
IC_X^H
\right)
\simeq
\Q.
\]
\end{theorem}

\begin{proof}
By closed-point adjunction and
Proposition~\ref{prop:app-costalk-line}, the relevant extension groups
reduce to
\[
\Hom_{MHS}
\left(
\Q^H(a),
\Q^H(-2)
\right)
\]
for the Tate structures under consideration.  Morphisms between pure
Hodge structures preserve weight.  Hence
\[
\Hom_{MHS}
\left(
\Q^H(-1),
\Q^H(-2)
\right)=0,
\]
whereas
\[
\Hom_{MHS}
\left(
\Q^H(-2),
\Q^H(-2)
\right)
\cong
\Q.
\]
\end{proof}

\begin{theorem}[Dual-orientation Ext calculation]
\label{thm:app-dual-ext}
One has
\[
\Ext^1_{MHM(X)}
\left(
IC_X^H,
i_*\Q^H(-1)
\right)
\simeq
\Q.
\]
\end{theorem}

\begin{proof}
Closed-point adjunction gives
\[
\Ext^1_{MHM(X)}
\left(
IC_X^H,
i_*\Q^H(-1)
\right)
\cong
\Hom_{D^bMHS}
\left(
i^*IC_X^H,
\Q^H(-1)[1]
\right).
\]
By Proposition~\ref{prop:app-stalk-line}, the only local contribution
relevant in degree one is
$H^{-1}(i^*IC_X^H)\simeq\Q^H(-1)$.  Therefore the extension group
reduces to
\[
\Hom_{MHS}
\left(
\Q^H(-1),
\Q^H(-1)
\right)
\cong
\Q.
\]
\end{proof}

\begin{corollary}[Dependence on Yoneda orientation]
\label{cor:app-orientation-matters}
The vanishing
\[
\Ext^1_{MHM(X)}
\left(
i_*\Q^H(-1),
IC_X^H
\right)
=
0
\]
does not exclude the Tate block $\Q^H(-1)$ from the local Hodge theory.
In the opposite Yoneda orientation one has
\[
\Ext^1_{MHM(X)}
\left(
IC_X^H,
i_*\Q^H(-1)
\right)
\simeq
\Q.
\]
\end{corollary}

\subsection{Consequences for the local corrected extension}

The object
$W_{\mathrm{loc}}^H\simeq i_*\Q^H(-1)$
constructed from the local $\can^H/\Var^H/N$ geometry is therefore
naturally an interface block.  To express the corrected object as a
short exact sequence, one must additionally choose a Yoneda orientation.

In the dual orientation, the nontrivial local extension has the form
\[
0
\longrightarrow
i_*\Q^H(-1)
\longrightarrow
E_{\mathrm{dual}}^H
\longrightarrow
IC_X^H
\longrightarrow
0.
\]
In the forward orientation with $IC_X^H$ as subobject, the corresponding
nontrivial point-supported quotient is instead $\Q^H(-2)$:
\[
0
\longrightarrow
IC_X^H
\longrightarrow
E_{\mathrm{forw}}^H
\longrightarrow
i_*\Q^H(-2)
\longrightarrow
0.
\]

After rational realization, both Tate structures become the same
skyscraper perverse sheaf $i_*\Q$.  The two exact sequences therefore
have the same type of rational point-supported term while remaining
distinct mixed-Hodge-module extensions.

\begin{remark}[Tate normalization and orientation]
\label{rem:app-required-bookkeeping}
The Tate normalization of a point-supported term must always be read
together with its Yoneda orientation.  The interface object
$i_*\Q^H(-1)$ should not be inserted into the forward group
$\Ext^1_{MHM(X)}(i_*(-),IC_X^H)$ without accounting for the costalk
normalization of Proposition~\ref{prop:app-costalk-line}.
\end{remark}

\section{Tate selection and ODP Hodge normalization}
\label{app:tate-normalization}

We now collect the Tate normalizations used in the local
ordinary-double-point calculation.  Three one-dimensional objects occur
naturally: the degree-two link line, the middle-dimensional vanishing
line, and the point-supported interface block.  Their underlying
rational vector spaces are one-dimensional, but their Hodge-theoretic
roles are different.

\subsection{Tate convention}

We use the rational Tate convention
\[
\operatorname{wt}\Q^H(m)=-2m.
\]
In particular,
\[
\operatorname{wt}\Q^H(-1)=2,
\qquad
\operatorname{wt}\Q^H(-2)=4.
\]
Thus tensoring with $\Q^H(-1)$ raises pure weight by $2$.

\subsection{The spectrum-one vanishing line}

For
$f=x_1^2+x_2^2+x_3^2+x_4^2$,
the Jacobian algebra is one-dimensional and the Steenbrink spectrum
consists of the single spectral number $\alpha=1$
\cite{Steenbrink1989Spectrum}.  With the convention used in Section~4,
the inequality
$3-p-1<\alpha\leq3-p$
places this contribution in
\[
\operatorname{Gr}_F^2H^3_{\mathrm{van}}
\cong
\C.
\]

\begin{proposition}[Tate type of the raw vanishing line]
\label{prop:app-raw-vanishing-tate}
The one-dimensional unipotent middle-dimensional vanishing Hodge
structure of a threefold ordinary double point is
\[
H^3_{\mathrm{van}}
\simeq
\Q^H(-2)
\]
in the normalization compatible with
$\operatorname{Gr}_F^2H^3_{\mathrm{van}}\cong\C$.
\end{proposition}

\begin{proof}
The spectrum calculation gives a one-dimensional Hodge structure whose
complexification satisfies $F^2=H^3_{\mathrm{van},\C}$ and $F^3=0$.
For the unipotent ODP vanishing line in the convention used here, the
corresponding pure Tate type is $(2,2)$, hence $\Q^H(-2)$.
\end{proof}

\subsection{The link line and the interface block}

The link of the threefold ordinary double point is homotopy equivalent
to $S^2\times S^3$, so its degree-two rational cohomology is
one-dimensional
\cite{MilnorSingularPoints,DimcaSheaves}.

\begin{proposition}[Tate type of the link line]
\label{prop:app-link-tate}
The degree-two link contribution appearing in the local intersection
complex has Tate normalization
\[
H^{-1}(i^*IC_X^H)
\simeq
\Q^H(-1).
\]
\end{proposition}

\begin{proof}
The underlying rational group is $H^2$ of the link.  In the local
intersection-complex Hodge normalization, this class has type $(1,1)$
and therefore determines $\Q^H(-1)$.
\end{proof}

\begin{definition}[ODP interface block]
\label{def:app-odp-gluing-block}
The point-supported Hodge object selected by the local
$\can^H/\Var^H/N$ construction is
\[
W_{\mathrm{loc}}^H
:=
i_*\Q^H(-1).
\]
\end{definition}

\begin{remark}[Two Tate lines]
\label{rem:app-two-tate-lines}
The equality of rational dimensions does not identify
$\Q^H(-1)$ with $\Q^H(-2)$.  The former is the link/interface line
appearing in the monodromy factorization, whereas the latter is the raw
middle-dimensional vanishing line and the costalk-normalized point term
appearing in the forward Yoneda orientation.
\end{remark}

\subsection{Compatibility with monodromy}

The nilpotent monodromy morphism is typed in the mixed-Hodge-module
category as
\[
N:
\Psi^H
\longrightarrow
\Psi^H(-1),
\]
and the local interface block occurs in the factorization
\[
\Psi^H
\xrightarrow{u_{\mathrm{loc}}}
W_{\mathrm{loc}}^H
\xrightarrow{v_{\mathrm{loc}}}
\Psi^H(-1).
\]
The twist $(-1)$ in the target of $N$ is therefore part of the
Hodge-theoretic typing of monodromy.  It should not be identified with
the Tate type of the raw middle-dimensional vanishing cohomology.

\begin{proposition}[Normalization summary]
\label{prop:app-tate-summary}
For a threefold ordinary double point, the local Tate structures relevant
to the construction are
\[
H^{-1}(i^*IC_X^H)
\simeq
\Q^H(-1),
\]
\[
W_{\mathrm{loc}}^H
\simeq
i_*\Q^H(-1),
\]
and
\[
H^1(i^!IC_X^H)
\simeq
H^3_{\mathrm{van}}
\simeq
\Q^H(-2).
\]
\end{proposition}

\begin{proof}
The stalk and costalk identifications are
Propositions~\ref{prop:app-stalk-line} and
\ref{prop:app-costalk-line}; the vanishing-line normalization is
Proposition~\ref{prop:app-raw-vanishing-tate}; and the interface block
is Definition~\ref{def:app-odp-gluing-block}.
\end{proof}

\section{Realization and lift obstructions}
\label{app:realization-obstructions}

The realization functor
\[
\rat:
MHM(X)
\longrightarrow
\Perv(X;\Q)
\]
is exact and faithful, but it is not full in general
\cite{SaitoMHM}.  Consequently, the existence of mixed-Hodge-module
structures on the endpoints of a rational perverse extension does not,
by itself, imply that the extension class lifts to $MHM(X)$.

\subsection{The realization map on extension classes}

For fixed mixed Hodge modules $A^H$ and $B^H$, exactness of realization
induces a natural map
\[
\rho_{A,B}:
\Ext^1_{MHM(X)}(A^H,B^H)
\longrightarrow
\Ext^1_{\Perv(X;\Q)}
\left(
\rat A^H,
\rat B^H
\right).
\]

\begin{definition}[Realization obstruction space]
\label{def:app-realization-obstruction}
Define
\[
\operatorname{Obs}_{\rat}(A^H,B^H)
:=
\operatorname{coker}(\rho_{A,B}).
\]
If
\[
e\in
\Ext^1_{\Perv(X;\Q)}
\left(
\rat A^H,
\rat B^H
\right),
\]
its image in
$\operatorname{Obs}_{\rat}(A^H,B^H)$
will be called the \emph{realization obstruction} of $e$ relative to
the prescribed endpoint Hodge modules.
\end{definition}

\begin{proposition}[Lift criterion]
\label{prop:app-lift-criterion}
A rational perverse extension class $e$ admits a mixed-Hodge-module lift
with fixed endpoints $A^H$ and $B^H$ if and only if
\[
e\in\operatorname{Im}(\rho_{A,B}),
\]
equivalently, if and only if its realization obstruction vanishes.
\end{proposition}

\begin{proof}
A lift is precisely a class in
$\Ext^1_{MHM(X)}(A^H,B^H)$
whose image under $\rho_{A,B}$ is $e$.  This is equivalent to the stated
image condition.
\end{proof}

\subsection{The local ODP obstruction in the forward orientation}

Let $X$ be a local neighborhood of a threefold ordinary double point.
The rational extension group
\[
\Ext^1_{\Perv(X;\Q)}
\left(
i_*\Q,
IC_X
\right)
\]
is one-dimensional by the local costalk calculation of Section~5.

\begin{proposition}[Wrong-Tate obstruction]
\label{prop:app-wrong-tate-obstruction}
For the endpoint choice
\[
A^H=i_*\Q^H(-1),
\qquad
B^H=IC_X^H,
\]
the realization map
\[
\Ext^1_{MHM(X)}
\left(
i_*\Q^H(-1),
IC_X^H
\right)
\longrightarrow
\Ext^1_{\Perv(X;\Q)}
\left(
i_*\Q,
IC_X
\right)
\]
has zero source.  Consequently, the nonzero corrected rational extension
class cannot lift in this orientation with point term
$i_*\Q^H(-1)$.
\end{proposition}

\begin{proof}
The source vanishes by
Theorem~\ref{thm:app-forward-ext}.  Since the corrected rational class
is nonzero, it cannot lie in the image of the realization map.
\end{proof}

\begin{proposition}[Forward Tate normalization]
\label{prop:app-correct-forward-tate}
For the endpoint choice
\[
A^H=i_*\Q^H(-2),
\qquad
B^H=IC_X^H,
\]
one has
\[
\Ext^1_{MHM(X)}
\left(
i_*\Q^H(-2),
IC_X^H
\right)
\simeq
\Q.
\]
Thus there is no weight obstruction to a nontrivial forward
mixed-Hodge-module extension with this point-supported endpoint.
\end{proposition}

\begin{proof}
This is the second calculation of
Theorem~\ref{thm:app-forward-ext}.
\end{proof}

\begin{remark}
One-dimensionality of the source removes the local weight obstruction
but does not, by itself, identify the corrected class.  The latter is
specified by the explicit nearby-/vanishing-cycle construction and its
realization.
\end{remark}

\subsection{Realization as a lifting criterion}

\begin{theorem}[Realization does not construct a lift]
\label{thm:app-realization-test}
Suppose a rational perverse extension has subobject and quotient that
are rational realizations of mixed Hodge modules $B^H$ and $A^H$.
This information alone does not determine a mixed-Hodge-module extension
of $A^H$ by $B^H$.  A lift exists precisely when the rational Yoneda
class belongs to the image of
\[
\rho_{A,B}:
\Ext^1_{MHM(X)}(A^H,B^H)
\longrightarrow
\Ext^1_{\Perv(X;\Q)}
\left(
\rat A^H,
\rat B^H
\right).
\]
\end{theorem}

\begin{proof}
This is Proposition~\ref{prop:app-lift-criterion}.  The local ODP
calculation gives a concrete instance in which the rational extension
group is nonzero while the mixed-Hodge-module extension group with the
incorrect Tate endpoint vanishes.  Thus compatibility of the endpoints
with realization is insufficient to produce a lift.
\end{proof}

The corrected construction of Sections~4--6 therefore proceeds inside
the mixed-Hodge-module category itself, using the
$\can^H/\Var^H/N$ geometry and its canonical subquotients, and verifies
the prescribed rational perverse object only after the Hodge-theoretic
construction has been made.

\section{Classification of local Hodge lifts}
\label{app:hodge-lifts}

We conclude with the local uniqueness statements underlying the rigidity
results.  The classification depends on the Yoneda orientation and on
whether the identifications of the endpoint Hodge modules are retained
as part of the extension data.

\subsection{Rigidified and unrigidified extensions}

Let $A^H$ and $B^H$ be fixed mixed Hodge modules.  A
\emph{rigidified extension} is a short exact sequence
\[
0
\longrightarrow
B^H
\longrightarrow
E^H
\longrightarrow
A^H
\longrightarrow
0
\]
together with specified identifications of the subobject and quotient
with $B^H$ and $A^H$.  Rigidified equivalence classes are classified by
$\Ext^1_{MHM}(A^H,B^H)$.

If the endpoint identifications are forgotten, the automorphism groups
of $A^H$ and $B^H$ act on the extension group.  For one-dimensional
Tate objects, these automorphisms are nonzero rational scalars.

\begin{lemma}[Scalar orbits in a one-dimensional Ext space]
\label{lem:app-scalar-orbits}
Suppose
\[
\Ext^1_{MHM(X)}(A^H,B^H)
\simeq
\Q
\]
and the automorphism groups of both endpoints act through
$\Q^\times$.  Then the induced action has exactly two orbits: the zero
class and the set of nonzero classes.
\end{lemma}

\begin{proof}
The nonzero elements of a one-dimensional rational vector space form a
single orbit under scalar multiplication by $\Q^\times$.  Endpoint
automorphisms act on the Yoneda parameter by such rescalings, giving the
asserted orbit decomposition.
\end{proof}

\subsection{Dual-orientation lifts with the \texorpdfstring{$\Q^H(-1)$}{Q(-1)}
block}

\begin{theorem}[Local dual-orientation classification]
\label{thm:app-dual-lift-classification}
For a threefold ordinary double point,
\[
\Ext^1_{MHM(X)}
\left(
IC_X^H,
i_*\Q^H(-1)
\right)
\simeq
\Q.
\]
Consequently:
\begin{enumerate}
\item rigidified extensions in this orientation form a
one-dimensional rational family;
\item after forgetting endpoint rigidifications, there is a unique
non-split isomorphism class; and
\item fixing the normalization
$v_{\mathrm{loc}}u_{\mathrm{loc}}=N$
selects a distinguished nonzero rigidified class.
\end{enumerate}
\end{theorem}

\begin{proof}
The Ext calculation is
Theorem~\ref{thm:app-dual-ext}.  The first assertion is the Yoneda
classification of rigidified extensions.  The second follows from
Lemma~\ref{lem:app-scalar-orbits}.  For the third, rescaling the
one-dimensional interface coordinate rescales the incoming and outgoing
maps dually, whereas the composite is constrained to be the fixed
monodromy morphism $N$.  The normalization therefore removes the
remaining scalar ambiguity.
\end{proof}

\subsection{Forward-orientation lifts}

\begin{theorem}[Local forward-orientation classification]
\label{thm:app-forward-lift-classification}
For a threefold ordinary double point,
\[
\Ext^1_{MHM(X)}
\left(
i_*\Q^H(-2),
IC_X^H
\right)
\simeq
\Q,
\]
whereas
\[
\Ext^1_{MHM(X)}
\left(
i_*\Q^H(-1),
IC_X^H
\right)
=
0.
\]
Thus a non-split forward extension with $IC_X^H$ as subobject requires
the costalk-normalized Tate object $i_*\Q^H(-2)$ rather than the
interface object $i_*\Q^H(-1)$.
\end{theorem}

\begin{proof}
This is Theorem~\ref{thm:app-forward-ext}, interpreted through the
Yoneda classification of short exact sequences.
\end{proof}

\subsection{Compatibility with rational realization}

Rational realization identifies the point-supported Tate objects
$i_*\Q^H(-1)$ and $i_*\Q^H(-2)$ with the same skyscraper perverse sheaf
$i_*\Q$.  The Tate distinction is therefore absent from the rational
extension problem.

\begin{proposition}[Local Hodge lifts above a fixed rational ODP class]
\label{prop:app-fixed-rational-class}
Fix the nonzero corrected rational ODP extension class.  A compatible
local mixed-Hodge-module lift is determined by:
\begin{enumerate}
\item a choice of Yoneda orientation;
\item the Tate-normalized point-supported Hodge module appropriate to
that orientation; and
\item a nonzero class in the corresponding one-dimensional
mixed-Hodge-module extension group whose rational realization is the
fixed perverse generator.
\end{enumerate}
Once the endpoint identifications and the normalization
$v_{\mathrm{loc}}u_{\mathrm{loc}}=N$ are fixed, the resulting local lift
is rigid.
\end{proposition}

\begin{proof}
The admissible Tate normalizations in the two orientations are determined
by Theorems~\ref{thm:app-forward-ext} and
\ref{thm:app-dual-ext}.  In either admissible orientation, the relevant
extension group is one-dimensional.  Requiring a lift of the prescribed
nonzero rational class excludes the zero element, while the normalization
$v_{\mathrm{loc}}u_{\mathrm{loc}}=N$ fixes the remaining scalar
ambiguity.
\end{proof}

\begin{remark}[Normalized rigidity]
The rigidity statements in the main text are understood with the
rational corrected object, endpoint Hodge structures, Yoneda orientation,
and monodromy normalization held fixed.  Rational realization alone does
not determine a unique mixed-Hodge-module lift.
\end{remark}

\printbibliography
\end{document}